\documentclass[10pt]{article}

\usepackage[margin=1in]{geometry}

\usepackage{microtype}
\usepackage{std,amsthm,graphicx,alignenv,mathtools,verbatim,url}
\usepackage[subrefformat=parens,labelformat=parens]{subfig}
\usepackage[utf8]{inputenc}
\usepackage[only,shortrightarrow]{stmaryrd}
\usepackage{mathrsfs}
\usepackage[bottom]{footmisc}
\usepackage{mathabx}

\usepackage[margin=10pt,font=small,labelfont=bf]{caption}
\usepackage[scaled]{helvet}		
\usepackage{courier}			
\normalfont
\usepackage[T1]{fontenc}

\usepackage{amsthm}

\theoremstyle{plain}
\newtheorem{theorem}{Theorem}
\newtheorem{proposition}[theorem]{Proposition}
\newtheorem{lemma}[theorem]{Lemma}
\newtheorem{corollary}[theorem]{Corollary}
		
\theoremstyle{definition}
\newtheorem*{defn}{Definition}
\newtheorem{example}{Example}

\theoremstyle{remark}

\newtheorem*{remark}{Remark}
\newtheorem*{remarks}{Remarks}
\numberwithin{theorem}{section}
\numberwithin{example}{section}

\numberwithin{equation}{section}
\numberwithin{figure}{section}

\newcommand\sref[1]{\textsection\ref{#1}}					

\newcommand\nobelowdisplayskip
	{\setlength\belowdisplayskip{0pt}}						

\makeatletter
\def\blfootnote{\gdef\@thefnmark{}\@footnotetext}
\makeatother

\usepackage{scalerel}
\newlength\bshft
\def\fakebold#1{\ThisStyle{\ooalign{$\SavedStyle#1$\cr%
  \kern-\bshft$\SavedStyle#1$\cr%
  \kern\bshft$\SavedStyle#1$}}}

\newcommand\bRR{\fakebold\RR}
\newcommand\bTheta{\boldsymbol\Theta}
\newcommand\bOmega{\boldsymbol\Omega}
\newcommand\bGamma{\boldsymbol\Gamma}

\newcommand\To{\smash{\mathring T}\vphantom{T}}			
\let\d\bdy 						
\DeclareMathOperator\WF{WF}		
\let\gradient\nabla				
\newcommand\PsiDO{\ensuremath{\Psi\text{\normalfont DO}}}	

\newcommand\Loc{_{\text{loc}}}		

\newcommand\En{\mathbf E}		
\newcommand\KE{\mathbf{KE}}		

\newcommand\sC{\mathscr C}		

\newcommand\DT{\ensuremath{_{\text{\normalfont DT}}}}		
\newcommand\MDT{\ensuremath{_{\text{\normalfont MDT}}}}	

\newcommand\JCS{J_{\text C\shortrightarrow \text S}}		
\newcommand\JCSp{J_{\text C\shortrightarrow \text S+}}	
\newcommand\JCB{J_{\text C\shortrightarrow \bdy}}		
\newcommand\JBS{J_{\bdy\shortrightarrow \text S}}		
\newcommand\JBB{J_{\bdy\shortrightarrow \bdy}}			

\newcommand\subI{_{\text I}}						
\newcommand\subR{_{\text R}}						
\newcommand\subT{_{\text T}}						

\newcommand\supi{^{\text i}}						
\newcommand\supo{^{\text o}}						
\newcommand\supio{^{\smash{\text i/\text o}}}			

\newcommand\eqml{\Eq}					
\newcommand\eqdef{\overset{\text{def}}{=}}		
\newcommand\eqFSP{\overset{\text{FSP}}{=}}		

\newcommand\Garding{G\aa{}rding}

\newcommand{\bmat}{\left[\begin{matrix}}
\newcommand{\emat}{\end{matrix}\right]}
\newcommand{\col}[1]{\begin{bmatrix}#1\end{bmatrix}}

\def\beq{\begin{equation} }
\def\eeq{\end{equation}}

\def\ben{\begin{enumerate} }
\def\een{\end{enumerate} }

\def \p { \partial}

\def \u{ \mathbf{u}} \def \f{\mathbf{f}}
\def \g{ \mathbf{g}}
\def \R{\mathcal{R}}
\def \M{ \mathcal{M}}

\newcommand\tail{_{\textup{tail}}}

\title{\vspace{-1em}Scattering Control for the Wave Equation with Unknown Wave Speed}

\author{Peter Caday$^{\text{*}}$\!,\,
		Maarten V.~de Hoop$^{\text{\textdagger}}$\!,\,
		Vitaly Katsnelson$^{\text{\textdaggerdbl}}$\!,\, and
		Gunther Uhlmann$^\parallel$}

\date\today

\begin{document}

\maketitle

\begin{abstract}
	Consider the acoustic wave equation with unknown wave speed $c$, not necessarily smooth. We propose and study an iterative control procedure that erases the history of a wave field up to a given depth in a medium, without any knowledge of $c$. In the context of seismic or ultrasound imaging, this can be viewed as removing multiple reflections from normal-directed wavefronts.
\end{abstract}

\section{Introduction}								\label{s:intro}

\blfootnote{\!\!$^{\text{*}}$ \url{pac5@rice.edu}\qquad $^{\text{\textdagger}}$ \url{mdehoop@rice.edu} \qquad $^{\text{\textdaggerdbl}}$ \url{vk17@rice.edu} \qquad $^\parallel$ \url{gunther@math.washington.edu}}
\blfootnote{\!\!$^{\text{*}}$ $^{\text{\textdagger}}$ $^{\text{\textdaggerdbl}}$ Department of Computational and Applied Mathematics, Rice University.}
\blfootnote{\!\!$^\parallel$ Department of Mathematics, University of Washington and Institute for Advanced Study, Hong Kong University of Science and Technology.}

Consider the acoustic wave equation with an unknown wave speed $c$, not necessarily smooth, on a finite or infinite domain $\Omega\subset\RR^n$. Assume that we can probe our domain $\Omega$ with arbitrary Cauchy data outside of $\Omega$, and measure the reflected waves outside $\Omega$ for sufficiently large time. The inverse problem is to deduce $c$ from these reflection data, and this is the basis for many wave-based imaging methods, including seismic and ultrasound imaging.

Toward this goal, we will define and study a time reversal-type iterative process, the \emph{scattering control series}. We were inspired by the work of Rose~\cite{Rose02} in one dimension, who developed a ``single-sided autofocusing'' procedure and identified it as Volterra iteration for the classical Marchenko equation. The Marchenko equation solves the inverse problem for the one-dimensional acoustic wave equation\footnote{More precisely, the Marchenko equation treats the constant-speed wave equation with potential, to which the one-dimensional acoustic wave equation can be reduced by a change of coordinates.}, recovering $c$ on a half-line from measurements made on the boundary. 
In the course of our research, it became evident that the new procedure is quite closely linked to boundary control problems~\cite{Belishev97,DKO}, and has similar properties to Bingham et al.'s iterative time-reversal control procedure~\cite{BKLS}.

In essence, scattering control allows us to isolate the deepest portion of a wave field generated by given Cauchy data--- behavior we demonstrate with both an exact and microlocal (asymptotically high-frequency) analysis. Along the way we present several applications of scattering control, including the removal of multiple reflections and the measurement of energy content of a wave field at a particular depth in $\Omega$. In a future paper, we anticipate illustrating how to locate discontinuities in $c$ and recover $c$ itself.

In the mathematical literature, the inverse problem's data are typically given on the boundary of $\Omega$, in terms of the Dirichlet-to-Neumann map or its inverse. We find that the Cauchy data-based reflection map allows us a much cleaner analysis. It is not hard to see (cf.~Proposition~\ref{p:DN-determines-Cauchy}) that the Dirichlet-to-Neumann map determines the Cauchy data reflection map, so no extra information is needed.

We start with an informal, graphical introduction to the problem. Section~\ref{s:exact} defines the scattering control series rigorously and provides an exact analysis of its behavior and convergence properties. Section~\ref{s:microlocal} pursues the same questions from a microlocal perspective. The discrepancy that arises between the exact and microlocal analyses allows us to provide more insight on convergence in Section~\ref{s:compare}. Section~\ref{s:marchenko} concludes by connecting our work to that of Rose and Marchenko.

\subsection{Motivation}							\label{s:motivation}

Before defining the scattering control equation and series, we begin by motivating our problem with a graphical example. 
In Figure~\ref{f:mr-demo}, the domain is $\Omega=\{x>0\}\subset\RR$, with a piecewise constant wave speed $c$ having two discontinuities. We extend $c$ to all of $\RR$, but assume it is known only outside $\Omega$. Now consider the solution of the acoustic wave equation on $\RR$ for time $t\in[0,2T]$, with rightward-traveling Cauchy data $h_0$ supported outside $\Omega$. The initial wave scatters from the discontinuities in $c$, producing an infinite sequence of reflections (Figure~\subref*{f:mr-demo-original}).

In imaging, one attempts to recover $c$ or some proxy for it. In many imaging algorithms currently in use, only waves having undergone a single reflection (so-called \emph{primary reflections}) are typically desired, while the remaining \emph{multiple reflections} only complicate the interpretation of the data. As a result, much research in seismic imaging has been directed toward removing or attenuating multiple reflections.

\begin{figure}
\subfloat[Wave field generated by Cauchy data $h_0$]{
	\includegraphics[page=1]{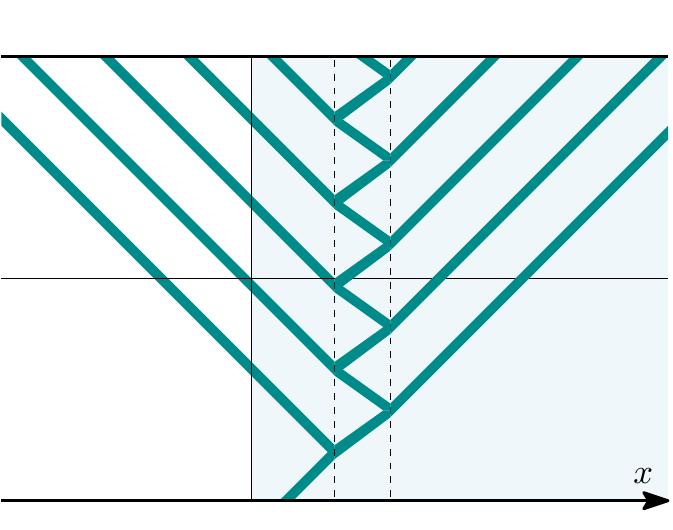}
	\label{f:mr-demo-original}
}
\hfill
\includegraphics[page=4]{Figures/NewTwoStep}
\hfill
\subfloat[Wave field with trailing pulse added to initial data)]{
	\includegraphics[page=2]{Figures/NewTwoStep}
	\label{f:mr-demo-ma}
}

\caption{(a) A domain $\Omega$ (shaded) with unknown wave speed $c$ is probed by exterior Cauchy data $h_0$. Two discontinuities in $c$ (dashed) scatter the incoming wave. (b) An appropriate trailing pulse added to $h_0$ suppresses multiple reflections.}
\label{f:mr-demo}
\end{figure}

For the problem at hand, it is plausible (and can be proven) that by adding a proper control, or \emph{trailing pulse} to the initial data, the multiple reflections may be suppressed, at the cost of a harmless additional outgoing pulse (Figure~\subref*{f:mr-demo-ma}). If $c$ were known inside the domain (cf.~\sref{s:ml-construct}), an appropriate control may be constructed microlocally under some geometric conditions. The issue, of course, is to find the control knowing only the reflection response of $\Omega$. 

Rather than attacking the multiple reflection suppression problem, however, we consider a related problem obtained by focusing on the interior, rather than exterior, of $\Omega$. Returning to Figure~\subref{f:mr-demo-ma}, we note that the wave field rightmost portion of the medium contains a single, purely transmitted wave, which we call the \emph{direct transmission} of the initial data $h_0$. Slightly more precisely, the wave field inside $\Omega$ at time $2T$ is generated exactly by the direct transmission at time $T$. The control has therefore isolated the direct transmission; our problem is to find such a control for a given $h_0$ using only information available outside $\Omega$.

\subsection{Almost direct transmission}				\label{s:adt}

At its heart, the direct transmission is a geometric optics construction, and is valid only in the high-frequency limit where geometric optics holds. Consequently, the directly transmitted wave field can be isolated only microlocally (modulo smooth functions). We will consider the geometric optics viewpoint later, but initially avoid a microlocal approach, as follows. Informally, suppose $h_0$ creates a wave that enters $\Omega$ at time 0, travelling normal to the boundary. At a later time $T$, the directly transmitted wave may be singled out from all others by its distance from the boundary: namely, $T$ (as long as it has not crossed the cut locus). By \emph{distance} we mean the travel time distance, which for $c$ smooth is Riemannian distance in the metric $c^{-2}dx^2$.

With this in mind, given Cauchy data $h_0$ supported just outside $\Omega$ we substitute for the direct transmission the \emph{almost direct transmission}, the part of the wave field of $h_0$ at time $T$ of depth at least $T$. More precisely, let $\Theta$ be a domain containing $\Omega$ and $\supp h_0$; then let $\Theta_T\subset\Theta$ be the set of points in $\Theta$ greater than distance $T$ from the boundary. The almost direct transmission of initial data $h_0$ at time $T$ is the restriction to $\Theta_T$ of its wave field at $t=T$ (Figure~\ref{f:adt}). 

\begin{figure}[tb]
	\centering
	\includegraphics{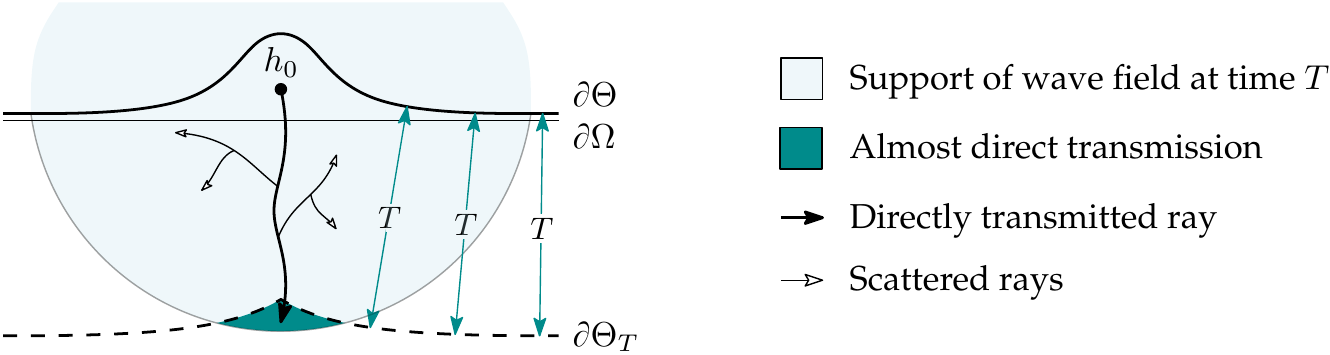}
	\caption{Almost direct transmission of initial data $h_0$ at time $T>0$.}
	\label{f:adt}
\end{figure}

The nonzero volume of $\Theta\setminus\Omega$ means that some multiply reflected rays may still reach $\Theta_T$. Hence, we have in mind taking a limit as $\Theta\to\Omega$ and the support of $h_0$ approaches a point on $\bdy\Omega$. In this limit, the support of the almost direct transmission converges to a point along the normal directly-transmitted ray, for sufficiently small $T$ (at least in the absence of caustics and before reaching the cut locus); see Figure~\ref{f:adt-shrink}.

\begin{figure}[tb]
	\centering
	\vspace*{0.125in}
	\includegraphics{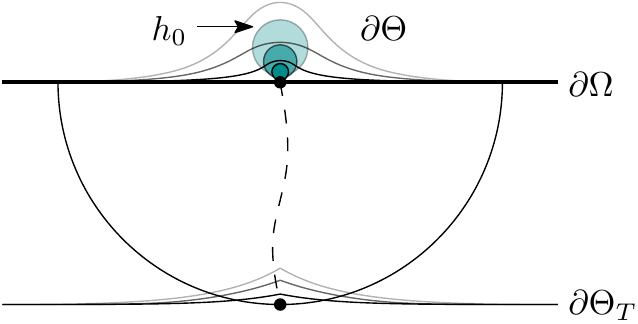}
	\caption{Shrinking the support of the initial data $h_0$ to a point. The dashed line indicates the normal geodesic from that point; the support of the almost direct transmission shrinks to a point on the geodesic.}
	\label{f:adt-shrink}
\end{figure}

\section{Exact scattering control}					\label{s:exact}

We set up the problem and our notation in~\sref{s:exact-setup}, then introduce the scattering control procedure in~\sref{s:exact-maf}, where we study its behavior and convergence properties. The final result, expressed in Corollary~\ref{c:dt-wave field}, is that scattering control recovers the almost direct transmission's wave field outside $\Theta$, modulo harmonic extensions. In~\sref{s:energy}, we apply this to recover the energy (with a harmonic extension) and kinetic energy of this portion of the wave field. Proofs for the results in these sections follow in~\sref{s:exact-proofs}.

\subsection{Setup}								\label{s:exact-setup}

\subsubsection{Unique continuation}					\label{s:uc}

Let $\Omega\subseteq\RR^n$ be a Lipschitz domain, and let $c$ be a wave speed satisfying $c,c^{-1} \in L^\infty(\RR^n)$. 

Initially, the sole extra restriction we impose on $c$ is that it satisfy a certain form of unique continuation. More precisely, assume there is a Lipschitz distance function $d(x,y)$ such that any $u\in C(\RR, H^1(\RR^n))$ satisfying either:
\pagebreak
\begin{itemize}
	\item $u,\d_t u=0$ for $t=0$ and $d(x,x_0)<T$ (finite speed of propagation)
	\item $u=0$ on a neighborhood of $[-T,T]\times\{x_0\}$ (unique continuation)
\end{itemize}
is also zero on the \emph{light diamond}
	\[
		D(x_0,T) = \set{(t,x)}{d(x,x_0)<T-\tabs{t}}\!,
	\]
if $(\d_t^2-c^2\Delta) u=0$ on a neighborhood of $D(x_0,T)$, for any $x_0\in\RR^n$, $T>0$.

While the set of wavespeeds with this property has not been settled in general, several large classes of $c$ are eligible, stemming from the well-known work of Tataru~\cite{Tataru}. Originally known for smooth sound speeds~\cite[Theorem 4]{SU-TATVariable}, Stefanov and Uhlmann later extended this to piecewise smooth speeds with conormal singularities~\cite[Theorem 6.1]{SU-TATBrain}, and Kirpichnikova and Kurylev to a class of piecewise smooth speeds in a certain kind of polyhedral domain~\cite[\textsection5.1]{KK}. The corresponding travel time $d(x,y)$ is the infimum of the lengths of all $C^1$ curves $\gamma(s)$ connecting $x$ and $y$, measured in the metric $c^{-2}dx^2$, such that $\gamma^{-1}(\singsupp c)$ has measure zero.

\subsubsection{Geometric setup}							\label{s:geometric-setup}

Next, let us set up the geometry of our problem.
We will probe $\Omega$ with Cauchy data (an \emph{initial pulse}) concentrated close to $\Omega$, in some Lipschitz domain $\Theta\supset\Omega$. We will add to this initial pulse a Cauchy data control (a \emph{tail}) supported outside $\Theta$, whose role is to remove multiple reflections up to a certain depth, controlled by a time parameter $T\in (0,\frac12\diam\Omega)$. This will require us to consider controls supported in a Lipschitz neighborhood $\Upsilon$ of $\clsr\Theta$ that satisfies $d(\bdy\Upsilon,\clsr\Theta)>2T$ and is otherwise arbitrary.

While we are interested in what occurs inside $\Omega$, the initial pulse region $\Theta$ will actually play a larger role in the analysis. First, define the \emph{depth} $d^*_\Theta(x)$ of a point $x$ inside $\Theta$:
\begin{equation}
	d^*_\Theta(x) = \when{+d(x,\bdy\Theta),		& x \in \Theta,\\
					  -d(x,\bdy\Theta),			& x \notin \Theta.}
													\label{e:exact-depth}
\end{equation}
Larger values of $d^*_\Theta$ are therefore deeper inside $\Theta$. For each $t$, define\footnote{We tacitly assume throughout that $\Theta_t$, $\Theta_t^\star$ are Lipschitz.} the open sets
\begin{nalign}
	\Theta_t^{\phantom\star} &= \set{x\in\Upsilon}{d^*_\Theta(x) > t}\!,\\
	\Theta_t^\star &= \set{x\in\Upsilon}{d^*_\Theta(x) < t}\!.
													\label{e:def-Theta-t-star}
\end{nalign}
As in~\eqref{e:def-Theta-t-star} above, we use a superscript $\star$ to indicate sets and function spaces lying outside, rather than inside, some region.

\subsubsection{Acoustic wave equation}					\label{s:wave-setup}

Let $\tilde{\mathbf C}$ be the space of Cauchy data of interest:
\begin{align}
	\tilde{\mathbf C} = H_0^1(\Upsilon) \oplus L^2(\Upsilon),
\end{align}
considered as a Hilbert space with the \emph{energy inner product}
\begin{equation}
	\big\langle{(f_0,f_1),\,(g_0,g_1)\big\rangle} = \int_{\Upsilon} \left(\gradient f_0(x)\cdot\gradient \conj g_0(x) + c^{-2}f_1(x)\conj g_1(x)\right)\,dx.
\end{equation}
Within $\tilde{\mathbf C}$ define the subspaces of Cauchy data supported inside and outside $\Theta_{t}$:
\begin{nalign}
	\mathbf H_t &= H_0^1(\Theta_{t})\oplus L^2(\Theta_{t}),					& \hspace{1in} \mathbf H &= \mathbf H_0,\\
	\tilde{\mathbf H}_t^{\mathrlap\star} &= H_0^1(\Theta_{t}^\star) \oplus L^2(\Theta_{t}^\star),	& \tilde{\mathbf H}^{\star}\! &= \tilde{\mathbf H}_0^\star.
\end{nalign}
Define the energy and kinetic energy of Cauchy data $h=(h_0,h_1)\in\tilde{\mathbf C}$ in a subset $W\subseteq\RR^n$:
\begin{align}
	\En_W(h) &= \int_W \left(\dabs{\gradient h_0}^2 + c^{-2} \dabs{h_1}^2\right)\,dx,
	&
	\KE_W(h) &= \int_W c^{-2}\dabs{h_1}^2\,dx.
\end{align}
Next, define $F$ to be the solution operator~\cite{LionsMagenes1} for the acoustic wave initial value problem:
\begin{align}
	&F\colon H^1(\RR^n)\oplus L^2(\RR^n)\to C(\RR,H^1(\RR^n)),
	&
	&
	F(h_0,h_1) = u
	\text{\; s.t. }
	\begin{pdebracketed}
	(\partial_t^2-c^2\Delta)u &= 0,\\
				\drestr{u}_{t=0} &= h_0,\\
			\drestr{\d_t u}_{t=0} &= h_1.
	\end{pdebracketed}
	\label{e:wave-ivp}
\end{align}
Let $R_s$ propagate Cauchy data at time $t=0$ to Cauchy data at $t=s$:
\begin{align}
	R_s = \left(F,\d_t F\right)\!\Big|_{t=s}
	\mspace{-8mu}
	\colon H^1(\RR^n)\oplus L^2(\RR^n)\to H^1(\RR^n)\oplus L^2(\RR^n).
\end{align}
Now combine $R_s$ with a time-reversal operator $\nu\colon \tilde{\mathbf C}\to\tilde{\mathbf C}$, defining for a given $T$
\begin{align}
	R &= \nu\circ R_{2T},
	&
	\nu&\colon (f_0,f_1)\mapsto(f_0,-f_1).
\end{align}
In our problem, only waves interacting with $(\Omega,c)$ in time $2T$ are of interest. Consequently, let us ignore Cauchy data not interacting with $\Theta$, as follows.

Let $\mathbf G=\tilde{\mathbf H}^\star\cap\big( R_{2T}(H^1_0(\RR^n\setminus\clsr\Theta)\oplus L^2(\RR^n\setminus\clsr\Theta))\big)$ be the space of Cauchy data in $\tilde{\mathbf C}$ whose wave fields vanish on $\Theta$ at $t=0$ and $t=2T$. Let $\mathbf C$ be its orthogonal complement inside $\tilde{\mathbf C}$, and ${\mathbf H}_t^\star$ its orthogonal complement inside $\tilde{\mathbf H}_t^\star$. With this definition, $R$ maps $\mathbf C$ to itself isometrically.

\subsubsection{Projections inside and outside $\Theta_t$}			\label{s:projections-setup}

The final ingredients needed for the iterative scheme are restrictions of Cauchy data inside and outside $\Theta$. While a hard cutoff is natural, it is not a bounded operator in energy space: a jump at $\bdy\Theta$ will have infinite energy. The natural replacements are Hilbert space projections. More generally, we consider projections inside and outside $\Theta_t$.

Let $\pi_t$, $\pi_t^\star$ be the orthogonal projections of $\mathbf C$ onto $\mathbf H_t$, $\mathbf H_t^\star$ respectively; let $\clsr\pi_t=1-\pi_t^\star$. As usual, write $\clsr\pi=\clsr\pi_0$, $\pi^\star=\pi_0^\star$. The complementary projection $I-\pi_t-\pi^\star_t$ is the orthogonal projection onto $\mathbf I_t$, the orthogonal complement to $\mathbf H_t \oplus\mathbf H_t^\star$ in $\mathbf C$.
It may be described by the following lemma, which is in essence the Dirichlet principle.
\begin{lemma}
	$\mathbf I_t$ consists of all functions of the form $(i_0,0)$, where $i_0\in H_0^1(\Upsilon)$ is harmonic in $\Upsilon\setminus\bdy\Theta_t$.
	\label{l:characterization-of-It}
\end{lemma}
\noindent
Lemma~\ref{l:characterization-of-It} provides two useful pieces of information. First, $\mathbf I=\mathbf I_0$ is independent of $c$. Secondly, we can identify the behavior of the projections $\clsr\pi_t$, $\pi^\star_t$. Inside $\Theta_t$ the projection $\clsr\pi_t h$ equals $h$, while outside $\Theta_t$, it agrees with the $\mathbf I_t$ component of $h$, which is the harmonic extension of $h\restrictto{\bdy\Theta_t}$ to $\Upsilon$ (with zero trace on $\bdy\Upsilon$). Similarly, $\pi^\star_t h$ is zero on $\Theta_t$, and outside $\Theta_t$ equals $h$ with this harmonic extension subtracted.

It will be useful to have a name for the behavior of $\clsr\pi_t h$, and so we define the notion of \emph{stationary harmonicity:}

\begin{defn}
Cauchy data $(h_0,h_1)$ are \emph{stationary harmonic} on $W\subseteq\RR^n$ if $h_0\restrictto W$ is harmonic and $h_1\restrictto W=0$.
\end{defn}

\subsection{Scattering control}				\label{s:exact-maf}

Suppose we have Cauchy data $h_0\in\mathbf H$. We can probe $\Omega$ with $h_0$ and observe $Rh_0$ outside $\Omega$. In particular, the reflected data $\pi^\star R$ can be measured, and from these data, we would like to procure information about $c$ inside $\Omega$. However, multiple scattering as waves travel into and out of $\Omega$ makes $\pi^\star Rh_0$ difficult to interpret. 

In this section, we construct a control in $\mathbf H^\star$ that eliminates multiple scattering in the wave field of $h_0$ up to a depth $T$ inside $\Theta$. More specifically, consider the \emph{almost direct transmission} of $h_0$:

\begin{defn}
	The \emph{almost direct transmission} of $h_0\in\mathbf H$ at time $T$ is the restriction $R_Th_0\restrictto{\Theta_{T}}$.
\end{defn}

Ideally, we would like to recover (indirectly) this restricted wave field. If considered as Cauchy data on the ambient space $\Upsilon$, the almost direct transmission has infinite energy in general due to the sharp cutoff at the boundary of $\Theta_T$. As a workaround, consider the almost direct transmission's minimal-energy extension to $\Upsilon$. This involves a harmonic extension of the first component of Cauchy data:

\begin{defn}
	The \emph{harmonic almost direct transmission} of $h_0$ at time $T$ is
	\begin{equation}
		h\DT = h\DT(h_0,T)=\clsr\pi_T R_T h_0.
		\label{e:adt-def}
	\end{equation}
\end{defn}

By Lemma~\ref{l:characterization-of-It}, $h\DT$ is equal to $R_Th_0$ inside $\Theta_T$; outside $\Theta_T$, its first component is extended harmonically from $\bdy\Theta_T$, while the second component is extended by zero.

\subsubsection{Scattering control series}

Our major tool is a Neumann series, the \emph{scattering control series}
\begin{equation}
	h_\infty = \sum_{i=0}^\infty (\pi^\star R\pi^\star R)^i h_0,
	\label{e:neumann}
\end{equation}
formally solving the \emph{scattering control equation}
\begin{equation}
	(I-\pi^\star R\pi^\star R)h_\infty = h_0.
	\label{e:maf}
\end{equation}
The series in general does not converge in $\mathbf C$; but it does converge in an appropriate weighted space, as we show in Theorem~\ref{t:limit-maf}. Applying $\clsr\pi$ to~\eqref{e:neumann}, we see that $h_\infty$ consists of $h_0$ plus a control in $\mathbf H^\star$. Our first theorem characterizes the behavior of the series. 

\begin{theorem}
	Let $h_0\in\mathbf H$ and $T\in(0,\frac12\diam\Theta)$. Then isolating the deepest part of the wave field of $h_0$ is equivalent to summing the scattering control series:
	\begin{nalign}
			(I-\pi^\star R\pi^\star R)h_\infty = h_0
		&\iff
			R_{-T} \clsr\pi R_{2T}h_\infty = h\DT
			\text{ and }
			h_\infty \in h_0+\mathbf H^\star.
		\label{e:basic-maf-behavior}
	\end{nalign}
	Above, $R_{-T} \clsr\pi R_{2T}h_\infty$ may also be replaced by $R_{-s} \clsr\pi_{T-s} R_{T+s}h_\infty$ for any $s\in[0,T]$.
	
	Such an $h_\infty$, if it exists, is unique in $\mathbf C$. As for the harmonic extension in $h\DT$, it is equal to $\clsr\pi R_{2T}h_\infty$ outside $\Theta$:
	\begin{align}
		h\DT\big|_{\Theta^\star} &= j_0\big|_{\Theta^\star},
		&
		\text{where\quad}\clsr\pi R_{2T}h_\infty = (j_0,j_1),
		\label{e:harmonic-ext-identity}
	\end{align}
	and is bounded:
	\begin{equation}
		\En_{\Theta_T^\star}(h\DT) \leq C\norm{h_0}
		\label{e:harmonic-ext-bound}
	\end{equation}
	for some $C=C(c,T)$ independent of $h_0$.
	\label{t:basic-maf}
\end{theorem}

Equation~\eqref{e:basic-maf-behavior} tells us that the wave field created by $h_\infty$ inside $\Theta$ at $t=2T$ is entirely due to the harmonic almost direct transmission at $t=T$ (Figure~\ref{f:basic-maf}). More generally, the wave field of $h_\infty$ agrees with that of $h\DT$ on its domain of influence. This is not true of $h_0$'s wave field, where other waves, including multiple reflections, will pollute the wave field at time $2T$. It follows that the tail $h_\infty-h_0$ enters $\Omega$ and carries all of the scattered energy of $h_0$ out with it. We will see this from an energy standpoint in Section~\ref{s:energy} and from a microlocal (geometric optics) standpoint in Section~\ref{s:microlocal}. 

\begin{figure}[tb]
	\centering
	\includegraphics{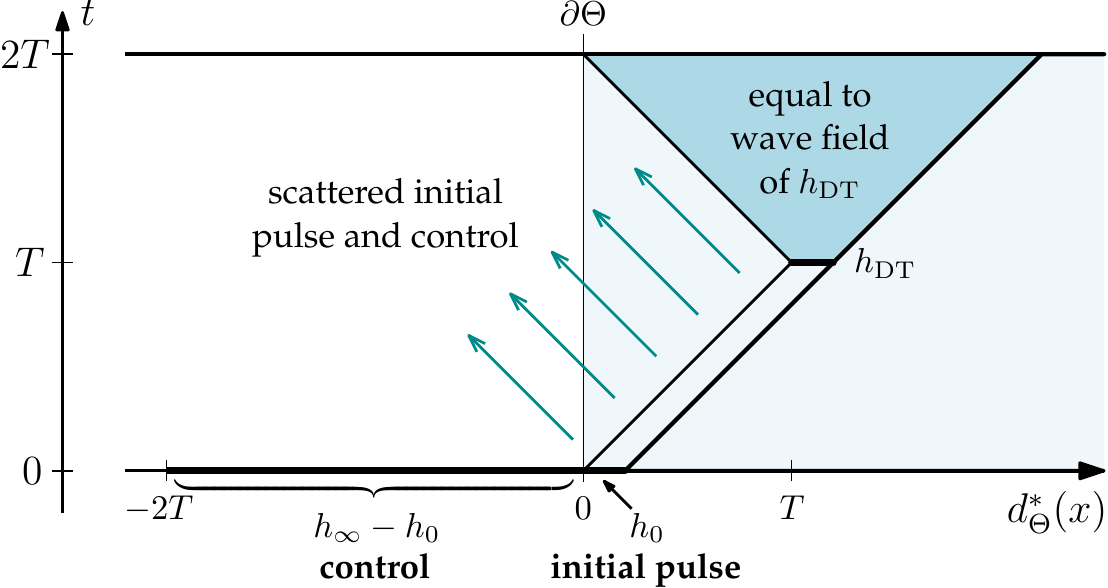}
	\caption{Illustration of the wave field generated by scattering control, as given by Theorem~\ref{t:basic-maf}.}
	\label{f:basic-maf}
\end{figure}

The question now is to study whether the Neumann series~\eqref{e:neumann} converges at all. Since $R$ is an isometry and $\pi^\star$ a projection, we have $\norm{\pi^\star R\pi^\star R}\leq 1$. From our later spectral characterization, we know that $\norm{\pi^\star R h}<\norm h$, strictly, for all $h\in\mathbf H^\star$. This is also true for a completely trivial reason: we eliminated $\mathbf G$ when constructing $\mathbf C$. What hinders convergence is that $\norm{h}-\norm{\pi^\star Rh}$ might be arbitrarily small; in other words, almost all the energy could be reflected off $\Theta$. Note that if the series fails to converge, no other finite energy control in $\mathbf H^\star$ can isolate the harmonic almost direct transmission of $h_0$; see Proposition~\ref{p:only-neumann}.

In the next theorem, we investigate convergence via the spectral theorem. It turns out that the only problem is outside $\Theta$; inside $\Theta$ the partial sums' wave fields at $t=2T$ do converge, and their energies are in fact monotonically decreasing. We will also demonstrate that the Neumann series converges in $\mathbf H$ for a dense set of $h_0$, and identify a larger space in which the Neumann series converges for any $h_0$.

For the statement of the theorem, define $\mathbf J$ to be the following space of Cauchy data, which, roughly speaking, remains completely inside or completely outside $\Theta$ in time $2T$:
\begin{equation}
	\mathbf J=\big(\mathbf H\cap R(\mathbf H)\big) \oplus \big(\mathbf H^\star \cap R(\mathbf H^\star)\big).
\end{equation}
Let $\chi\colon\mathbf C\to\mathbf J$ be the orthogonal projection onto $\mathbf J$.

\begin{theorem}
	With $h_0, T$ as in Theorem~\ref{t:basic-maf}, define the partial sums
	\begin{equation}
		h_k = \sum_{i=0}^k (\pi^\star R\pi^\star R)^i h_0.
		\label{e:partial-sums}
	\end{equation}
	Then the deepest part of the wave field can be (indirectly) recovered from $\{h_k\}$ regardless of convergence of the scattering control series:
	\begin{align}
		\lim_{k\to\infty} R_{-T}\clsr\pi R_{2T}h_k &=  R_T\chi h_0 = h\DT,
		&
		\norm{\clsr\pi Rh_k}&\searrow\norm{h\DT}\!.
		\label{e:maf-interior-limit}
	\end{align}
	The set of $h_0$ for which the scattering control series converges in $\mathbf C$,
	\begin{equation}
		\mathcal Q = \set{h_0\in\mathbf H}{(I-\pi^\star R\pi^\star R)^{-1} h_0\in\mathbf C}\!,
		\label{e:Q-definition}
	\end{equation}
	is dense in $\mathbf H$.
	For all $h_0\in\mathbf H$, the partial sum tails $h_k-h_0$ converge in a weighted space that can be formally written as
	\begin{align}
		&\frac{I}{\sqrt{I-N^2}}(1-\chi)\mathbf C,
		&
		N &= \clsr\pi R\clsr\pi + \pi^\star R\pi^\star.
		\label{e:space-of-maf-convergence}
	\end{align}
	\label{t:limit-maf}
\end{theorem}

As an immediate corollary of~\eqref{e:maf-interior-limit}, we recover in the limit the wave field generated by the harmonic almost direct transmission outside $\Theta$, using only observable data.

\begin{corollary}
	Let $F\DT(t,x) = (Fh\DT)(t-T,x)$ be the harmonic almost direct transmission's wave field. Then
	\begin{align}
		(Fh_k)(t,x) - (F\pi^\star R_{2T}h_k)(t-2T,x)&\to F\DT(t,x)
		&
		\text{ as } k\to\infty,
	\end{align}
	the convergence being $H^1$ in space, uniformly in $t$.
	\label{c:dt-wave field}
\end{corollary}

We end this section with three small propositions. 
The first states that the scattering control equation has no solution if the Neumann series diverges.
\begin{proposition}
	Let $h_0, T$ be as in Theorem~\ref{t:basic-maf}, and suppose $(I-\pi^\star R\pi^\star R)k=h_0$ for some $k\in \mathbf H^*$. Then the scattering control series~\eqref{e:neumann} converges.
	\label{p:only-neumann}
\end{proposition}

The second proposition characterizes the space $\mathbf H^\star$ containing the Cauchy data controls. Essentially, each control is supported in a $2T$-neighborhood of $\Theta$ and its wave field is contained in this neighborhood for $t\in[0,2T]$, up to harmonic functions.

\begin{proposition}
	\nobelowdisplayskip
	The control space $\mathbf H^\star$ consists of Cauchy data supported outside $\Theta$ whose wave fields are stationary harmonic outside a $2T$-neighborhood of $\Theta$ at $t=0,2T$:
	\begin{equation}
		\mathbf H^\star = \set{h\in\tilde{\mathbf C}}{\pi_{-2T}^\star h = \pi_{-2T}^\star R^{}_{2T} h = \clsr\pi h = 0}\!.
		\label{e:H-star-characterization}
	\end{equation}
	\label{p:H-star-characterization}
\end{proposition}

The third proposition shows that our reflection data (the Cauchy solution operator $F$, restricted to the exterior of $\Omega$) is determined by the Dirichlet-to-Neumann map, which is the data usually assumed given in boundary control problems and the inverse problem. As a result, our method requires no additional information, from a theoretical standpoint.

\begin{proposition}
	Let $c_1,c_2$ be $L^\infty$ wave speeds on a $C^1$ domain $\Omega\subseteq\RR^n$. Extend $c_1,c_2$ to $\Omega^\star=\RR^n\setminus\clsr\Omega$ by setting them equal to some $c_0\in C^\infty(\RR^n)$.
	
	Define solution operators $F_1,F_2$ corresponding to $c_1,c_2$ as in~\eqref{e:wave-ivp}, and \emph{Dirichlet-to-Neumann} maps
	\begin{align}
		\Lambda_i&\colon g \mapsto \drestr{\d_\nu u}_{\RR\times\bdy\Omega},
		\text{ where }
		\begin{pdebracketed}
			(\d_t^2-c_i^2\Delta) u &= 0,\\
			\drestr u_{\RR\times\bdy\Omega} &= g,\\
			\drestr u_{t=0} = \drestr{\d_t u}_{t=0} &= 0.
		\end{pdebracketed}
		\label{e:DN-def}
	\end{align}
	If $\Lambda_1=\Lambda_2$, then $\drestr{F_1h}_{\RR\times\Omega^\star}=\drestr{F_2h}_{\RR\times\Omega^\star}$ for all $h\in H^1(\Omega^\star)\oplus L^2(\Omega^\star)$.
	\label{p:DN-determines-Cauchy}
\end{proposition}

\subsection{Recovering internal energy}				\label{s:energy}

As a direct application of the results in~\sref{s:exact-maf}, we show how scattering control can recover the energy of the harmonic almost direct transmission using only data outside $\Omega$, assuming $\supp h_0\subset \Theta\setminus\clsr\Omega$. If the Neumann series converges to some $h_\infty\in\mathbf C$, we can recover the energy directly from $h_\infty$, but if not, Theorem~\ref{t:limit-maf} allows us to recover the same quantities as a convergent limit involving the Neumann series' partial sums. In a forthcoming paper we demonstrate how these energies may be used in inverse boundary value problems for the wave equation that arise in imaging.

\begin{proposition}
	Let $h_0\in\mathbf H$, $T>0$, and suppose $(I-\pi^\star R\pi^\star R)h_\infty = h_0$. Then we can recover the harmonic almost direct transmission's energy from data observable on $\Theta^\star\cup\supp h_0$:
	\begin{align}
		 \En_{\RR^n}(h\DT) &= \En_{\RR^n} \big(h_\infty\big) - \En_{\RR^n} \big(\pi^\star Rh_\infty\big).
		 \label{e:basic-E-recovery}
	\intertext{We can also recover the kinetic energy of the almost direct transmission (with no harmonic extension) from data observable on $\Theta^\star\cup\supp h_0$:}
		\KE_{\Theta_T}(R_Th_0) &= \frac12\tform{h_0,h_0-R \pi^\star Rh_\infty - R h_\infty}.
		\label{e:basic-KE-recovery}
	\end{align}
	\label{p:basic-energy}
\end{proposition}

\begin{proposition}
	Let $h_0\in\mathbf H$ and $T>0$, and $h_k$ as before. We can recover the energy of the harmonic almost direct transmission as a convergent limit involving data observable on $\Theta^\star\cup\supp h_0$:
	\begin{equation}
		\En_{\RR^n}(h\DT) = \lim_{k\to\infty}\left[\En_{\RR^n}(h_k) - \En_{\RR^n}(\pi^\star R h_k)\right].
		\label{e:limit-E-recovery}
	\end{equation}
	Similarly, for the kinetic energy of the almost direct transmission,
	\begin{nalign}
		4\,\KE_{\Theta_{T}}(R_Th_0) &= \lim_{k\to\infty}\Big[\En(h_k)+\En(h_0)-\En(\pi^\star R \pi^\star R h_k) \\
		&\qquad\qquad + 2\tform{\pi^\star Rh_k,\, h_k-R\pi^\star Rh_k} - 2\tform{h_0,\,R \pi^\star Rh_k + R h_k}\Big].
		\label{e:limit-KE-recovery}
	\end{nalign}
	\label{p:limit-energy}
\end{proposition}

\subsection{Proofs}				\label{s:exact-proofs}

\begin{proof}[Proof of Theorem~\ref{t:basic-maf}]
	The proof is mostly a simple application of unique continuation and finite speed of propagation.
	
\paragraph{Equation~\eqref{e:basic-maf-behavior} ($\Rightarrow$)}
	Let $v(t,x)=FR_{-2T}\clsr\pi R_{2T} h_\infty$ be the solution of the wave equation with Cauchy data $\clsr\pi R_{2T}h_\infty$ at $t=2T$. We will often consider Cauchy data at a particular time, and so define $\mathbf v=(v,\d_t v)$.
	
	Applying $\bar\pi$ to the defining equation $(I-\pi^\star R\pi^\star R)h_\infty=h_0$ implies $\clsr\pi h_\infty=h_0$; also $(\pi^\star \mathbf v)(0,\cdot)=0$, since
\begin{nalign}
					0 = \pi^\star h_0	&= \pi^\star (I-\pi^\star R_{-2T}\pi^\star R_{2T}) h_\infty\\
									&= \pi^\star R_{-2T}\clsr\pi R_{2T} h_\infty\\
									&= (\pi^\star \mathbf v)(0,\cdot).
\end{nalign}

Outside of $\Theta$, then, $\mathbf v(0,\cdot)$ and $\mathbf v(2T,\cdot)$ are equal to their projections in $\mathbf I$, and therefore are stationary harmonic. Equivalently, $\partial_t v$ and $\partial_{tt} v$ are zero on $\Theta^\star$ for $t=0,2T$.

Because $c$ is time-independent, $\partial_t v$ is also a (distributional) solution to the wave equation. If $\partial_t v\in C(\RR, H^1(\RR^n))$, then Lemma~\ref{l:uc-top-and-bottom} applied to $\partial_t v$ gives $\partial_t v(T,\cdot)=\partial_{tt} v(T,\cdot)=0$ on $\Theta_T^\star$; it follows that $\mathbf v(T,\cdot)$ is stationary harmonic on $\Theta_T^\star$. For the general case, choose a sequence of mollifiers $\rho_\eps\to\delta$ in $\mathcal E'(\RR)$ and apply Lemma~\ref{l:uc-top-and-bottom} to $\rho_\eps'(t) \ast v$ to obtain the same conclusion.

\begin{figure}[tb]
	\centering
	\includegraphics{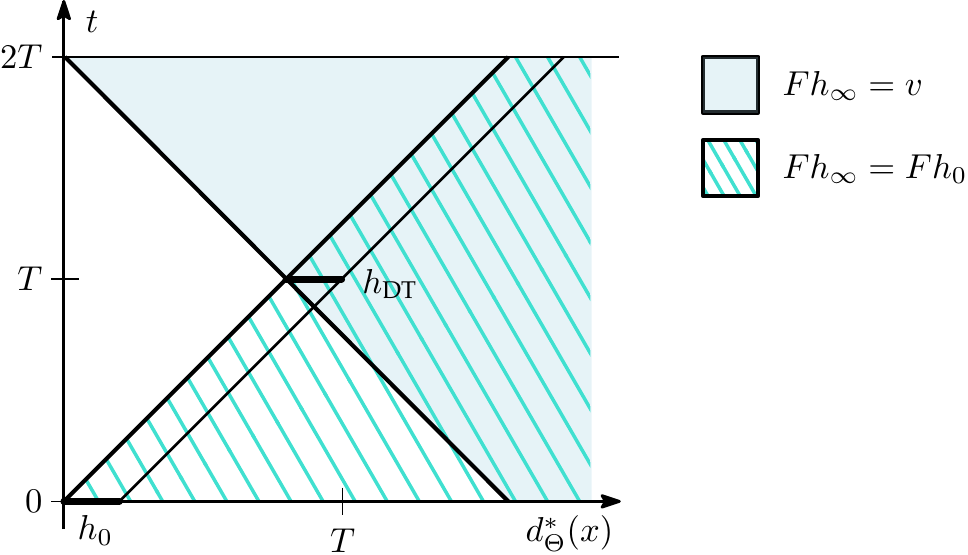}
	\caption{Finite speed of propagation applied twice to wave field $v$.}
	\label{f:double-fsp}
\end{figure}

By finite speed of propagation (FSP), $\bar\pi_{\abs s}R_s\bar\pi=\bar\pi_{\abs s}R_s$ for any $s\in\RR$.
Applying this twice, we find that in $\Theta_T$ at time $T$, the solution $v$ is equal to $h_\infty$'s wave field, which in turn is equal to $h_0$'s wave field (Figure~\ref{f:double-fsp}):
\begin{equation}
	\clsr\pi_T \mathbf v(T,\cdot) =
	\clsr\pi_T R_{-T} \clsr\pi R_{2T} h_\infty
		\eqFSP \clsr\pi_T R_{-T} R_{2T} h_\infty
		= \clsr\pi_T R_T h_\infty
		\eqFSP \clsr\pi_T R_T \bar\pi h_\infty
		= \clsr\pi_T R_Th_0
		\eqdef h\DT.
	\label{e:double-FSP}
\end{equation}
However, since $\mathbf v(T,\cdot)$ is stationary harmonic on $\Theta_T^\star$, we can remove the projection on the left-hand side: $\clsr\pi_T R_{-T} \clsr\pi R_{2T} h_\infty=R_{-T} \clsr\pi R_{2T} h_\infty$. This proves the forward direction of~\eqref{e:basic-maf-behavior}. More generally, it follows that $\clsr\pi_{T-s} R_{T+s} h_\infty=\mathbf v(T+s,\cdot)=R_s h\DT$ for $s\in[0,T]$. Indeed, $\mathbf v(T+s,\cdot)=R_{T+s} h_\infty$ on $\Theta_{T-s}$ by finite speed of propagation, and using Lemma~\ref{l:uc-top-and-bottom} as above implies $\mathbf v(T+s,\cdot)$ is stationary harmonic on $\Theta_{T-s}^\star$ for $s\in[0,T]$.

\paragraph{Equation~\eqref{e:harmonic-ext-identity}}

As above, apply Lemma~\ref{l:uc-top-and-bottom} to $\d_t v$. This implies that $\d_t v\restrictto{[0,2T]\times\Theta^\star}=0$. Hence $v$ is constant in time in $\Theta^\star$. At time $T$, we have $\mathbf v(T,\cdot)=\clsr\pi_T R_Th_0$, and the pressure field $v(T,\cdot)$ is the harmonic extension of the first component of $R_Th_0\restrictto{\bdy\Theta_T}$. At time $2T$, $\mathbf v$ equals $\clsr\pi R_{2T}h_\infty$ on $\Theta^\star$ by construction, proving~\eqref{e:harmonic-ext-identity}.

\paragraph{Equation~\eqref{e:basic-maf-behavior} ($\Leftarrow$)}

Conversely, suppose $R_{-T}\clsr\pi R_{2T} h_\infty=h\DT$.
Let $v(t,x)=(Fh\DT)(t-T,x)$ be the wave field generated by the harmonic almost direct transmission. Since $\mathbf v(T,\cdot)$ is stationary harmonic in $\Theta_T^\star$ we have $(\d_t\mathbf v)(T,\cdot)=0$ there. Applying finite speed of propagation, $(\d_t \mathbf v)(0,\cdot)=0$ on $\Theta^\star$, so $(\pi^\star \mathbf v)(0,\cdot)=0$.

Because $R_{-T}\clsr\pi R_{2T} h_\infty=h\DT$, the solution $v$ is equal to $(F\clsr\pi R_{2T}h_\infty)(t-2T,x)$, the wave field generated by $\clsr\pi R_{2T}h_\infty$. Hence $\pi^\star R_{-2T}\clsr\pi R_{2T}h_\infty = 0$, and we have
\begin{equation}
	(I - \pi^\star R \pi^\star R)h_\infty = (I - \pi^\star R (\pi^\star+\clsr\pi) R)h_\infty = (I - \pi^\star)h_\infty = \clsr\pi h_\infty.
	\label{e:some-maf-solution}
\end{equation}
Therefore $h_\infty$ is a solution of the scattering control equation for some initial pulse $\clsr\pi h_\infty$; by hypothesis, this initial pulse is $h_0$.

\paragraph{Uniqueness of $h_\infty$}
	Since $R$ is unitary and $\pi$ is a projection, any $g\in\mathbf C$ satisfies
\begin{equation}
	\norm{\pi^\star R\pi^\star Rg} \leq \norm{\pi^\star R g} \leq \dnorm g\!.
	\label{e:maf-iteration-no-energy-gain}
\end{equation}
	Now, suppose that $(I-\pi^\star R\pi^\star R)g=0$ for some $g\in\mathbf C$. As $g=\pi^\star R\pi^\star Rg$ no energy can be lost in either application of $\pi^\star$, and both inequalities of~\eqref{e:maf-iteration-no-energy-gain} are in fact equalities. Hence $\clsr\pi g$ and $\clsr\pi R_{2T}g$ must be zero, implying $g\in \mathbf G$. But by construction $\mathbf G\cap\mathbf C=\{0\}$, establishing uniqueness.
	
	Conversely, any $g\in\mathbf G$ satisfies $g=\pi^\star R\pi^\star R g$ by finite speed of propagation, so in fact $\mathbf G = \ker(I-\pi^\star R\pi^\star R)$.

\paragraph{Equation~\eqref{e:harmonic-ext-bound}}

Finally, since $i=h\DT\restrictto{\Theta^\star}=\clsr\pi_T R_Th_0\restrictto{\Theta^\star}$, it follows immediately that
\begin{equation}
	\norm i \leq\norm{\clsr\pi_T R_Th_0}\leq\norm{R_Th_0}=\norm{h_0}.
\end{equation}
The proof is complete.
\end{proof}

\noindent In the proof of Theorem~\ref{t:basic-maf}, we used the following corollary of finite speed of propagation and unique continuation:

\begin{lemma}
	\nobelowdisplayskip
	Let $u\in C(\RR,H^1(\RR^n))$ be a solution of $(\partial_t^2-c^2\Delta)u=0$ such that $u(0,\cdot)=u(2T,\cdot)=\d_t u(0,\cdot)=\d_t u(2T,\cdot)=0$ on $\Theta^\star$. Then $u$ is zero on the set
	\[
		\mathcal D = \set{(t,x)}{d^*_\Theta(x)<T-\tabs{t-T}}\!.
	\]	
	\label{l:uc-top-and-bottom}
\end{lemma}
\begin{proof}
	By finite speed of propagation, $u$ is zero on a neighborhood of $[0,2T]\times \Theta_{-T-\delta}$ for all $\delta>0$, and thus by unique continuation, also zero on the union of open light diamonds centered at points in $[0,2T]\times\bdy\Theta_{-T-\delta}$. This includes $[0,2T]\times \Theta_{-T/2-\delta}$, and repeating the argument, we find that $u=0$ on all open light diamonds centered at points in $[0,2T]\times \Theta_{-T/2^n-\delta}$ for all $n\in\ZZ$ and $\delta>0$. The union of these open light diamonds is $\mathcal D$.
\end{proof}

\begin{proof}[Proof of Theorem~\ref{t:limit-maf}]
The proof is via the spectral theorem, which will also shed further light on the behavior of the Neumann series.

First, note $R=\nu\circ R_{2T}$ is self-adjoint as well as unitary, since $R^*=R_{2T}^*\circ\nu^*=R_{-2T}\circ\nu=\nu\circ R_{2T}$. Divide $R$ into two self-adjoint parts, $N$ and $Z$:
\begin{align}
	N &= \pi^\star R\pi^\star + \clsr\pi R\clsr\pi,
	&
	Z &= \pi^\star R\clsr\pi + \clsr\pi R\pi^\star.
\end{align}
In other words, thinking of $\im\pi^\star=\mathbf H^\star$ and $\im\clsr\pi=\mathbf H\oplus\mathbf I$ as two halves of $\mathbf C$, the operator $N$ describes wave movement within one half, while $Z$ describes movement from one half to the other. For any $f\in\mathbf H$ the identity $f = R^2f = (N^2+Z^2)f + (NZ+ZN)f$ holds. If $f\in\mathbf H^\star$ or $f\in\mathbf H\oplus\mathbf I$, then $(NZ+ZN)f$ is in the opposite half from $f$, so $NZ+ZN=0$, and $N^2+Z^2=I$ when the domain is restricted to either half.

Applying the spectral theorem to $N$, identify $\mathbf C$ with $L^2(X,\mu)$ for some set $X$ and measure $\mu$, upon which $N$ acts as a multiplication operator $n(x)$. As $Z$ and $N$ do not commute, $Z$ has no special form with respect to this spectral representation.

Since $\norm N\leq \norm R=1$, we have $\abs n\leq 1$. Split $X$ into two sets
\begin{nalign}
	X'	&=	n^{-1}(\{-1,1\}),			\\
	X''	&=	n^{-1}((-1,1)) = X\setminus X'.
\end{nalign}
For $h\in L^2(X',\mu)$,
\begin{equation}
	\norm{Nh}=\left(\int_X n^2\abs h^2\,d\mu \right)^{\mathrlap{1/2}} \;\; = \norm{h}=\norm{Rh}\!,
\end{equation}
implying $Zh=0$. Conversely, if $Zh=0$, then $\norm{Nh}=\norm h$, implying $n=\pm 1$ on $\supp h$. In consequence, $L^2(X',\mu)=\ker Z=\mathbf J$, and hence $\chi$ is multiplication by the characteristic function of $X'$.

Returning to the Neumann series, since $(\pi^\star)^2 = \pi^\star$, rewrite $h_k$ as
\begin{nalign}
	h_k - h_0	= \sum_{i=0}^{k-1} (\pi^\star R\pi^\star R\pi^\star )^i
						 (\pi^\star R\pi^\star)(\pi^\star R\bar\pi) h_0
		=  \sum_{i=0}^{k-1} n^{2i+1} Z h_0
		&=  n\frac{1-n^{2k}}{1-n^2} Zh_0.
	\label{e:spectral-tail}
\end{nalign}
Turning to $\clsr\pi R h_k$ now, since $Zn=-nZ$ on $\im\pi^\star\owns n^iZh_0$ and $Z^2=1-n^2$,
\begin{nalign}
	\clsr\pi R h_k
	=
	Z(h_k-h_0)+nh_0
	&=
	Zn\frac{1-n^{2k}}{1-n^2}Zh_0 + nh_0\\
	&=
	-n\frac{1-n^{2k}}{1-n^2}Z^2h_0 + nh_0\\
	&=
	n^{2k+1}h_0.
\end{nalign}
$n^{2k+1}h_0$ converges pointwise, monotonically, as a function in $L^2(X,\mu)$:
\begin{equation}
	(\clsr\pi Rh_k)(x) = n^{2k+1}h_0(x) \to \when{nh_0(x), & \abs{n(x)}=1;\\
						0, & \abs{n(x)}<1.}
	\qquad
	\forall x\in X.
	\label{e:lim-pibar-Rhk}
\end{equation}
The convergence holds not only pointwise but also in $L^2(X,\mu)$ by dominated convergence. Its limit function is exactly $n\chi h_0=R\chi h_0$, the projection of $Rh_0$ onto $\mathbf J$, proving the first limit in~\eqref{e:maf-interior-limit}. Also, as a consequence of the monotonicity, $\norm{\clsr\pi Rh_k}\searrow \norm{R\chi h_0}=\norm{\chi h_0}$.

Hence, while the Neumann series $\{h_k\}$ may diverge, the component of $Rh_k$ in $\mathbf H\oplus\mathbf I$ (and therefore inside $\Theta$) converges and is actually decreasing in energy.

\paragraph{Proof of~\eqref{e:space-of-maf-convergence}}

Starting from~\eqref{e:spectral-tail}, we wish to commute $Z$ and the powers of $n$.
In the weighted space $L^2(X'',(1-n^2)^2\mu)$,
\begin{equation}
	h_k-h_0\to \frac{n}{1-n^2}Zh_0 = \frac{n}{1-n^2}Z(1-\chi) h_0 = -Z\frac{n}{1-n^2}(1-\chi) h_0.
\end{equation}
The factor $(1-\chi)$ is a projection away from the kernel of $Z$, where $(1-n^2)^{-1}$ blows up. We may insert it because $\mathbf J=\ker Z$, and therefore $Z\chi=0$. After doing so, the second equality holds because $(1-\chi)h_0$ lies in the inside half $\mathbf H\oplus\mathbf I$.

Any $j\in\mathbf H$ (or $\mathbf H^\star$) satisfies $\norm j^2 = \norm{Rj}^2=\norm{Zj}^2+\norm{Nj}^2$, so
\begin{equation}
	\dnorm{Zj}^2 = \int_X (1-n^2)\abs j^2\,d\mu = \dnorm{\sqrt{1-n^2}\,j}^2\!.
\end{equation}
Applying this relation to $h_k-h_0$,
\begin{equation}
	\norm{h_k-h_0} = \norm{n\frac{1-n^{2k}}{\sqrt{1-n^2}}(1-\chi) h_0}.
\end{equation}
Therefore, $h_k-h_0$ lies in the weighted space $L^2(X'',(1-n^2)\mu)$, and, by dominated convergence, converges to a function $h_\infty-h_0\in L^2(X'',(1-n^2)\mu)$. Formally, this latter space can be written $(I-N^2)^{-1/2} (1-\chi) \mathbf C$, establishing~\eqref{e:space-of-maf-convergence}.

\paragraph{Density of $\mathcal Q$}

Decompose $X$ as the disjoint union of the family of sets
\begin{nalign}
	X_{-1}	&= n^{-1}(\{-1,0,1\});\\
	X_i		&= n^{-1}((-1+2^{-i-1},-1+2^{-i})\cup(1-2^{-i},1-2^{-i-1})) \qquad i=0,1,\dotsb.
\end{nalign}
Let $h_0^{\smash{(i)}}=h\cdot\mathbf 1_{X_{-1}\sqcup\dotsb\sqcup X_i}$, where $\mathbf 1_A$ denotes the indicator function of $A\subseteq X$. Then $h_0^{\smash{(i)}}\to h_0$ in $L^2(X,\mu)$. Using the fact that $Zn=-nZ$ on $\mathbf H^\star$, as before the $k\mith$ partial sum of the Neumann series for $h_0^{\smash{(i)}}$ is
\begin{equation}
	h^{(i)}_k = h^{(i)}_0 + n\frac{1-n^{2k}}{1-n^2} Zh^{(i)}_0
		= h^{(i)}_0 - Zn\frac{1-n^{2k}}{1-n^2} (1-\chi) h^{(i)}_0.
\end{equation}
Since either $n=\pm1$ (so that $1-\chi=0$) or $\abs n<1-2^{-i-1}$, the multiplier $n\frac{1-n^{2k}}{1-n^2} (1-\chi)$ is bounded in $k$ and the Neumann series converges in $\mathbf C$. Hence $h^{\smash{(i)}}_0\in\mathcal Q$ for all $i$, proving $\mathcal Q$ is dense.

\paragraph{Proof of $R \chi h_0=h\DT$}

When $h_k$ converges in $\mathbf C$, by Theorem~\ref{t:basic-maf} we have
\begin{equation}
\lim_{k\to\infty} R_{-T}\clsr\pi R_{2T}h_k=h\DT.
\end{equation}
The left hand side is equal to $R\chi h_0$; hence for $h_0\in\mathcal Q$,
\begin{equation}
	R\chi h_0=h\DT.
	\label{e:restricted-chi-dt-link}
\end{equation}
By the unitarity of $R$ and~\eqref{e:harmonic-ext-bound}, $h_0\mapsto h\DT$ is a continuous map from $\mathbf H$ to $\mathbf C$. The left-hand side is likewise continuous in $h_0$. So, since $\mathcal Q$ is dense in $\mathbf H$,~\eqref{e:restricted-chi-dt-link} holds for all $h_0\in\mathbf H$. This together with our earlier work establishes~\eqref{e:maf-interior-limit}. By the same argument, $h\DT=\lim_{k\to\infty} R_{-s}\clsr\pi_{T-s} R_{T+s}h_k$ for any $s\in[0,T]$.
\end{proof}

\begin{proof}[Proof of Proposition~\ref{p:basic-energy}]

Equation~\eqref{e:basic-E-recovery} follows directly from~\eqref{e:basic-maf-behavior}:
\begin{nalign}
	\En(h\DT) = \En(R_{-T} \clsr\pi R_{2T} h_\infty) = \En(\clsr\pi R h_\infty) = \En(R h_\infty) - \En(\pi^\star R h_\infty) = \En(h_\infty) - \En(\pi^\star R h_\infty).
\end{nalign}

For \eqref{e:basic-KE-recovery}, let $v(t,x)=(F\clsr\pi R_{2T}h_\infty)(t-2T,x)$, as in the proof of Theorem~\ref{t:basic-maf}.
Subtract its time-reversal to get the solution $w(t,x)=v(t,x)-v(2T-t,x)$, and as before write $\mathbf v=(v,\d_t v)$, $\mathbf w=(w,\d_t w)$. Consider the energy of $\mathbf w$ at $t=T$. Now $w(T,\cdot)=0$ everywhere and $\d_tw = 2\d_t v=0$ on $\Theta_T^\star$ (as shown by the proof of Theorem~\ref{t:basic-maf}), so the only energy of $\mathbf w$ at time $T$ is inside $\Theta_{T}$:
\begin{nalign}
	\En(\mathbf w(T,\cdot)) &= \int_{\RR^n} c^{-2} \dabs{\d_t w(T,\cdot)}^2\,dx = \int_{\RR^n} c^{-2} \dabs{2\d_t v(T,\cdot)}^2\,dx\\
	&\qquad\qquad\quad = 4\,\KE_{\Theta_T}(\mathbf v(T,\cdot)) \eqFSP
	4\,\KE_{\Theta_T}(R_Th_\infty) \eqFSP 4\,\KE_{\Theta_T}(R_Th_0).
	\label{e:KE-w-link}
\end{nalign}
The last two equalities are by finite speed of propagation, as in~\eqref{e:double-FSP}. By conservation of energy,
\begin{align}
	\En(\mathbf w(T,\cdot)) = \En(\mathbf w(2T,\cdot)) &= \En(\clsr\pi R h_\infty - \clsr\pi R\clsr\pi R h_\infty).
	\label{e:w-energy}
\end{align}
Expanding out the energy norm on the right hand side,
\begin{equation}
	4\,\KE_{\Theta_T}(R_Th_0) = \dnorm{\clsr\pi R h_\infty}^2 + \dnorm{\clsr\pi R \clsr\pi Rh_\infty}^2 - 2\tform{\clsr\pi R h_\infty,\clsr\pi R \clsr\pi Rh_\infty}.
\end{equation}
Using $\clsr\pi R\clsr\pi Rh_\infty + \clsr\pi R\pi^\star Rh_\infty=\clsr\pi h_\infty=h_0$, and $\pi^\star R\clsr\pi Rh_\infty=0$,
\begin{nalign}
	\dnorm{\clsr\pi R h_\infty}^2 &= \dnorm{Rh_\infty}^2 - \dnorm{\pi^\star Rh_\infty}^2
	\\
	&= \dnorm{h_\infty}^2 - \dnorm{\pi^\star Rh_\infty}^2;\\
	\dnorm{\clsr\pi R \clsr\pi Rh_\infty}^2 &= \dnorm{h_0 - \clsr\pi R \pi^\star Rh_\infty}^2\\
	&= \dnorm{h_0}^2 + \dnorm{\clsr\pi R\pi^\star Rh_\infty}^2 - 2\form{h_0, \clsr\pi R\pi^\star Rh_\infty}\\
	&= \dnorm{h_0}^2 + \dnorm{\pi^\star Rh_\infty}^2 - \dnorm{\pi^\star R\pi^\star Rh_\infty}^2 - 2\tform{h_0, R\pi^\star Rh_\infty};\\
	\form{\clsr\pi R h_\infty,\clsr\pi R \clsr\pi Rh_\infty}
	&=
	\form{R h_\infty, R \clsr\pi Rh_\infty}
	-
	\form{\pi^\star R h_\infty,\pi^\star R \clsr\pi Rh_\infty}	
	\\
	&=
	\form{h_\infty, \clsr\pi Rh_\infty}\\
	&=
	\tform{h_0, Rh_\infty}.
	\label{e:expand-ke-energy-diff}
\end{nalign}
Recalling $\pi^\star R\pi^\star R h_\infty = h_0-h_\infty$ and simplifying yields~\eqref{e:basic-KE-recovery}.
\end{proof}

\begin{proof}[Proof of Proposition~\ref{p:limit-energy}]

\hfill			

\paragraph{Proof of~\eqref{e:limit-E-recovery}}

The energy recovery formula follows directly from Theorem~\ref{t:limit-maf}:

\begin{nalign}
	\lim_{k\to\infty}\left[\En(h_k)-\En(\pi^\star Rh_k)\right]
		&=
	\lim_{k\to\infty} \dnorm{Rh_k}^2-\dnorm{\pi^\star Rh_k}^2\\
		&=
	\lim_{k\to\infty} \dnorm{\clsr\pi Rh_k}^2\\
		&=
	\dnorm{h\DT}^2.
\end{nalign}

\paragraph{Proof of~\eqref{e:limit-KE-recovery}}

The proof is similar to~\eqref{e:basic-KE-recovery}, but with extra terms.
By~\eqref{e:KE-w-link}--\eqref{e:expand-ke-energy-diff}, $h_\infty$ satisfies
\begin{align}
	4\,\KE_{\Theta_T}(R_Th_0)
		&=
	\En(\clsr\pi R h_\infty - \clsr\pi R\clsr\pi R h_\infty)
	\label{e:ke-expansion-h-infinity-1}
	\\
		&=
	\En(h_\infty)+\En(h_0)-\En(\pi^\star R \pi^\star R h_\infty) - 2\tform{h_0,R \pi^\star Rh_\infty + R h_\infty}.
	\label{e:ke-expansion-h-infinity-2}
\end{align}
For $h_k$, we must modify the second equality as $\pi^\star R\clsr\pi R h_k$ is no longer zero. Instead, write $\pi^\star R\clsr\pi Rh_k$ as $\pi^\star h_k-\pi^\star R\pi^\star Rh_k$ to obtain
\begin{nalign}
	\En(\clsr\pi R h_k - \clsr\pi R\clsr\pi R h_k) &= \En(h_k)+\En(h_0)-\En(\pi^\star R \pi^\star R h_k) \\
		&\qquad\qquad + 2\tform{\pi^\star Rh_k,\pi^\star h_k- \pi^\star R\pi^\star Rh_k}- 2\tform{h_0,R \pi^\star Rh_k + R h_k}.
	\label{e:ke-expansion-h-k}
\end{nalign}
The right-hand side is the quantity in the limit in~\eqref{e:limit-KE-recovery}. As $k\to\infty$, it converges to~\eqref{e:ke-expansion-h-infinity-2} by continuity as long as $h_0\in\mathcal Q$; hence its limit is $4\,\KE_{\Theta_T}(R_Th_0)$. This proves~\eqref{e:limit-KE-recovery} when $h_0\in\mathcal Q$. Then, by continuity and the density of $\mathcal Q$, \eqref{e:limit-KE-recovery} must hold for all $h_0\in\mathbf H$.

Interestingly, to obtain kinetic energy we used initial data
\begin{nalign}
	\lim_{k\to\infty} \left[\clsr\pi R h_k - \clsr\pi R\clsr\pi R h_k\right]= R\chi h_0-\clsr\pi\chi h_0 = (n-1)\chi h_0,
\end{nalign}
equal to $-2$ times the projection of $h_0$ onto $L^2(n^{-1}(\{-1\}),\mu)$.
\end{proof}

\begin{proof}[Proof of Lemma~\ref{l:characterization-of-It}]
	The proof is essentially that of the Dirichlet principle. First, while ${\mathbf H} = {\mathbf H}_t^\star \oplus \mathbf I_t\oplus {\mathbf H}_t$, we note that also (with tildes)
	\begin{equation}
		\tilde{\mathbf H} = \tilde{\mathbf H}_t^\star \oplus \mathbf I_t\oplus {\mathbf H}_t.
	\end{equation}
	This is true simply because $\mathbf I_t$ is orthogonal to $\mathbf G$ and hence to $\tilde{\mathbf H}_t^\star={\mathbf H}_t^\star\oplus \mathbf G$.
	
	Now, for one direction of the proof, consider an arbitrary $i=(i_0,i_1)\in\mathbf I_t$. Since $\Theta_t$ is Lipschitz, its boundary has measure zero, so $L^2(\Upsilon)=L^2(\Theta_{t}^\star)\oplus L^2(\Theta_{t})$. Hence $i_1$ must be zero.
	
	Let $\phi\in \mathbf H_t$ be nonzero and $a>0$. Then $\norm{i+a\phi}^2=\norm i^2+a^2\norm\phi^2>\norm i^2$ by orthogonality. Hence $a=0$ is a local minimum of $\norm{i+a\phi}^2$, and the derivative of this quantity with respect to $a$ is zero at $a=0$:
	
	\begin{nalign}
		0 = \left.\smd a \dnorm{i+a\phi}^2\right|_{a=0} = 2\dform{i,\phi} = 2\int_{\mathrlap{\Upsilon}} \quad \gradient i_0\cdot\gradient\phi_0.
		\label{e:i0-weakly-harmonic}
	\end{nalign}
	Since $i_0$ is weakly harmonic on $\Theta_{t}$, it is strongly harmonic; in the same way it is harmonic on $\Theta_{t}^\star$.
	
	Conversely, if $i_0\in H^1_0(\Upsilon)$ is harmonic on $\Upsilon\setminus\bdy\Theta_t$, it is weakly harmonic, immediately implying $(i_0,0)$ is orthogonal to $\mathbf H_t$ and $\mathbf H_t^\star$.
\end{proof}

\begin{proof}[Proof of Proposition~\ref{p:only-neumann}]
	First, we have the equivalence
	\begin{equation}
		(I-\pi^\star R\pi^\star R)h_\infty=h_0
		\iff
		(I-\pi^\star R\pi^\star R)(h_\infty-h_0)=\pi^\star R\pi^\star Rh_0.
	\end{equation}
	Since $\pi^\star R\pi^\star$ is self-adjoint and $\norm{\pi^\star R\pi^\star}\leq 1$ (cf.\ the proof of Theorem~\ref{t:limit-maf}),
	it suffices to apply the following lemma.
\end{proof}

\begin{lemma}
	Let $A$ be a self-adjoint linear operator on a Hilbert space $X$ with $\norm A\leq 1$. If $x,y\in X$ satisfy $(I-A^2)y=x$, then the Neumann series $\sum_{k=0}^\infty A^{2k}x$ converges to the minimal-norm solution $y=y^*$ to $(I-A^2)y=x$.
	\label{l:only-neumann}
\end{lemma}
\begin{proof}
	By the spectral theorem, $X$ can be identified with $L^2(W,\mu)$ for some set $W$ and measure $\mu$, upon which $A$ acts as a (real-valued) multiplication operator $a(w)$; also $\norm A\leq 1$ implies $\abs a\leq1$ for all $w\in W$. If $i(w)$ denotes the indicator function of $a^{-1}(\pm1)$, then $y=y^*=iy$ is the minimal-norm solution of $(I-A^2)y=x$.
	
	Let $y_n=y_n(w)=\sum_{k=0}^n a^{2k} x$ be the $n\mith$ partial sum of the Neumann series; then $y_n(w)$ converges monotonically away from zero to $yi$ for each $w$. Hence $y_n\to y^*$ in $L^2(W,\mu)$.
\end{proof}

\begin{proof}[Proof of Proposition~\ref{p:H-star-characterization}]
	Our first task is to characterize $\mathbf G$, the space of functions staying outside $\Theta$ in time $2T$. We make a guess $\mathbf G_1$ for $\mathbf G$ and show that the two are equal by unique continuation, using Lemma~\ref{l:uc-top-and-bottom}. After identifying $\mathbf G$, it will be easy to identify $\mathbf H^\star$, its complement in $\tilde{\mathbf H}^\star$.
	
	First, define
	\begin{nalign}
		\mathbf G_0 &= H_0^1(\Theta^\star_{-2T})\oplus L^2(\Theta^\star_{-2T}),\\
		\mathbf G_1 &= G_0 + R_{2T}G_0.
	\end{nalign}
	By finite speed of propagation, $\mathbf G_0,\, R_{2T}\mathbf G_0\subseteq\mathbf G$, so $\mathbf G_1\subseteq \mathbf G$. We want to show that in fact $\mathbf G=\mathbf G_1$. Accordingly, suppose $g\in\mathbf G$ and $g\perp\mathbf G_1$.%

	Having $g\perp\mathbf G_0$ implies $\pi^\star_{-2T}g = 0$; similarly $g\perp R_{2T}\mathbf G_0$ implies $\pi^\star_{-2T}Rg=0$. That is, the wave field of $g$ is stationary harmonic outside a $2T$-neighborhood of $\Theta$ at $t=0,2T$. As in the proof of Theorem~\ref{t:basic-maf}, we can apply Lemma~\ref{l:uc-top-and-bottom} to (a smoothed version of) $\d_t Fg$ to conclude that $R_Tg$ is stationary harmonic outside a $T$-neighborhood of $\Theta$ at time $T$; i.e.,
	\begin{equation}
		\pi^\star_{-T}R_T g=0.
		\label{e:perp-to-G0-implies}
	\end{equation}
	On the other hand, $g\in \mathbf G$ implies that $\clsr\pi g = \clsr\pi Rg=0$; the wave field of $g$ is zero on $\Theta$ at $t=0,2T$. Applying Lemma~\ref{l:uc-top-and-bottom}, we can conclude that the wave field of $g$ is zero on a $T$-neighborhood of $\Theta$ at time $T$; i.e.
	\begin{equation}
		\clsr\pi_{-T} R_T g = 0.
		\label{e:perp-to-RG0-implies}
	\end{equation}
	Hence $R_T g = \pi^\star_{-T}R_Tg + \clsr\pi_{-T}R_Tg = 0$; we conclude that $g=0$, and therefore $\mathbf G = \mathbf G_1$.
	
	Now, we can prove~\eqref{e:H-star-characterization}. $\mathbf H^\star$ is the complement of $\mathbf G$ in $\tilde{\mathbf H}^\star$.
	For Cauchy data $h\in \tilde{\mathbf C}$,
	\begin{align}
		h \in \tilde{\mathbf H}^\star 	&\iff		\clsr\pi h = 0,
	\end{align}
	and since $\mathbf G = \mathbf G_1$, equations~(\ref{e:perp-to-G0-implies}--\ref{e:perp-to-RG0-implies}) imply
	\begin{align}
		h \perp \mathbf G 	&\iff		h\perp\mathbf G_0 \text{ and } h\perp R\mathbf G_0	\iff	\pi^\star_{-2T} h = 0 \text{ and } \pi^\star_{-2T}Rh = 0. 	\mbox\qedhere
	\end{align}
\end{proof}

\begin{proof}[Proof of Proposition~\ref{p:DN-determines-Cauchy}]
	Let $h\in H^1(\Omega^\star)\oplus L^2(\Omega^\star)$, and let $u_1=F_1h$ be the solution with respect to $c_1$. Define $u_2$ to be the solution of the IBVP~\eqref{e:DN-def} with boundary data $\restr{u_1}_{\RR\times\bdy\Omega}$. Since $c_1$ and $c_2$ have identical Dirichlet-to-Neumann maps, it follows that $\restr{\d_\nu u_1}_{\RR\times\bdy\Omega}=\restr{\d_\nu u_2}_{\RR\times\bdy\Omega}$. Therefore, $u_2$ may be extended to $\RR\times\RR^n$ by setting it equal to $u_1$ outside $\Omega$, and both $u_2$ and $\d_\nu u_2$ will be continuous on $\RR\times\bdy\Omega$. Hence $u_2$ satisfies the wave equation with respect to $c_2$ inside and outside $\Omega$, and satisfies the interface conditions at $\bdy\Omega$. Therefore, it is a solution of the $c_2$ wave equation on all of $\RR^n$~\cite[Theorem 2.7.3]{StolkThesis}. By uniqueness of the Cauchy problem, $u_2 = F_2h$, and by definition $u_2=u_1=F_1h$ on $\Omega^\star$.
\end{proof}

\section{Microlocal analysis of scattering control}				\label{s:microlocal}

In this section, we turn from our exact analysis of scattering control to a study of its microlocal (high-frequency limit) behavior, allowing us to study reflections and transmissions of wavefronts naturally. To accomodate the microlocal analysis, we first narrow the setup somewhat, and consider a microlocally-friendly version of the scattering control equation in~\sref{s:ml-maf}. Section~\ref{s:ml-adt} introduces a natural analogue of the almost direct transmission, based on depths of singularities (covectors), rather than points.

Just as before, isolating the microlocal almost direct transmission is sufficient for solving the microlocal scattering control equation (\sref{s:ml-isolating}). If the wave speed $c$ is known, it is not hard, as~\sref{s:ml-construct} shows, to construct solutions assuming some natural geometric conditions. Our main result, Theorem~\ref{t:ml-convergence}, is that the scattering control iteration converges to a similar solution, to leading order in amplitude, under the same conditions. Finally,~\sref{s:ml-uniqueness} discusses uniqueness for the microlocal scattering control equation. Proofs of the key results follow in~\sref{s:ml-proofs}.

\paragraph{Notation}

Throughout, ``$\eqml$'' denotes equality modulo smooth functions or smoothing operators, and $\To^*M=T^*M\setminus 0$ ($M$ a manifold). A \emph{graph FIO} is a Fourier integral operator associated with a canonical graph. Finally, for a set of covectors $W\subseteq T^*M$, let $\mathcal D'_W$, $\mathcal E'_W$ denote the spaces of distributions with wavefront set in $W$.

\subsection{Microlocal scattering control}	\label{s:ml-maf}

In this section, we begin by restricting $\Omega$ and $c$ suitably in order to study reflection and transmission of singularities. We also adjust the scattering control equation slightly, replacing projections with smooth cutoffs, and employing a parametrix for wave propagation.

Let $\Omega\subseteq\RR^n$ be a smooth open submanifold, and $c$ a piecewise smooth%
	\footnote{As usual, ``smooth'' means $C^\infty$ throughout.}
wave speed that is singular only on a set of disjoint, closed%
		\footnote{If $c$ is singular on some non-closed hypersurface $\Gamma_i$, we may be able to ``close up'' $\Gamma_i$ in such a way that it does not intersect the other hypersurfaces.}%
, connected, smooth hypersurfaces $\Gamma_i$ of $\clsr\Omega$, called \emph{interfaces}. Let $\Gamma=\bigcup\Gamma_i$; let $\{\Omega_j\}$ be the connected components of $\RR^n\setminus\Gamma$. Also assume each smooth piece of $c$ extends smoothly to $\RR^n$.

The projections $\clsr\pi$, $\pi^\star$ arose quite naturally in the exact setting, taking the roles of cutoffs inside and outside $\Theta$. Because they introduce singularities along $\bdy\Theta$, it is natural to replace them by smooth cutoffs for a microlocal study. We will also separate the initial data $h_0$ from the cutoff region. To accommodate both aims, choose nested open sets $\Theta',\,\Theta''$ between $\Omega$ and $\Theta$:
\begin{equation}
	\Omega\subseteq\Theta' \subseteq\clsr{\Theta'}\subseteq\Theta''\subseteq\clsr{\Theta''}\subseteq\Theta,
\end{equation} 
and smooth cutoffs $\sigma,\sigma^\star\colon\RR^n\to[0,1]$ such that
\begin{align}
	\sigma(x) &= \when{1, & x\in\Theta''\!,\\
				 0, & x\notin\Theta,}
	&
	&\supp\sigma = \Theta,
	\\
	\sigma^\star &= 1 - \sigma,
	&
	&\supp\sigma^\star = \RR^n\setminus\Theta''.
\end{align}
The sets $\Theta',\,\Theta''$ should be thought of as arbitrarily close to $\Theta$; we will write $\Theta'^\star=\RR^n\setminus\clsr{\Theta'}$.

Finally, a standard parametrix $\tilde R$ accounting for reflections and refractions will frequently replace the exact propagator $R$, discussed at greater length in Appendix~\ref{s:parametrix-construction}. Most importantly, $\tilde R$ includes microlocal cutoffs along glancing rays, so that $Rh_0\eqml\tilde Rh_0$ as long as $\WF(h_0)$ is disjoint from a set of covectors $\mathcal W\subset T^*(\RR^n\setminus\Gamma)$ producing near-glancing broken bicharacteristics.

The object of study is now the
 \emph{microlocal scattering control equation}
\begin{equation}
	(I - \sigma^\star R\sigma^\star R) h_\infty \eqml h_0,
												\label{e:ml-maf}
\end{equation}
and accompanying formal Neumann series
\begin{equation}
	h_\infty \eqml \sum_{i=0}^\infty (\sigma^\star R)^{2i} h_0.
												\label{e:ml-maf-series}
\end{equation}
In general, the operator $(\sigma^\star R)^2$ preserves but does not improve Sobolev regularity, preventing us from assigning any meaning to this infinite sum \emph{a priori}.\footnote{Were $(\sigma^\star R)^2$ to have negative Sobolev order,~\eqref{e:ml-maf-series} may be interpreted as an asymptotic series. This situation occurs, for example, for $c$ with $C^{1,\alpha}$ or weaker singularities~\cite{dHUV}, in the absence of diving rays.}
Instead, we will consider the limiting behavior of its partial sums.

\subsection{Microlocal almost direct transmission}		\label{s:ml-adt}

\begin{figure}[tb]
	\centering
	\hbox{}\hfill
	\subfloat[Wavefront set of solution at time $T$]{\includegraphics[page=1]{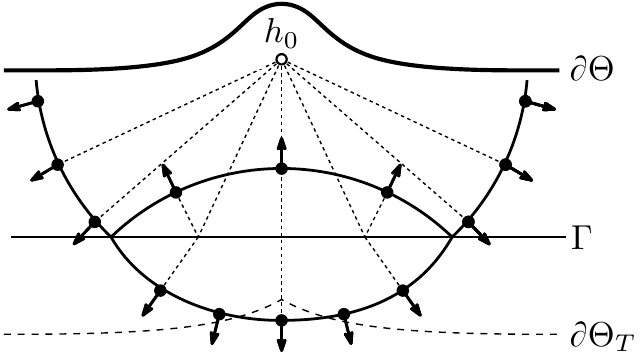}}
	\hfill
	\subfloat[Wavefront set of microlocal almost direct transmission]{\includegraphics[page=2]{Figures/NewMicroADT.pdf}}
	\hfill\hbox{}

	\hbox{}\hfill
	\subfloat[Wavefront set of almost direct transmission]{\includegraphics[page=3]{Figures/NewMicroADT.pdf}}
	\hfill\hbox{}
	
	\caption{Microlocal almost direct transmission. (a) The wavefront set of the solution with point source $h_0$ includes reflected and refracted singularities due to an interface $\Gamma$. (b) The microlocal almost direct transmission does not include the reflected singularities; their depth is less than $T$. (c) Wavefront set of the (non-microlocal) almost direct transmission, for comparison.}

	\label{f:ml-adt}
\end{figure}

The almost direct transmission played a central role in the exact analysis of scattering control. We begin by studying its natural microlocal analogue. Intuitively, the \emph{microlocal almost direct transmission} $h\MDT$ is the microlocal restriction of the solution at time $T$ to singularities in $\To^*\Theta$ whose distance from the surface $\bdy T^*\Theta$ is at least $T$ (Figure~\ref{f:ml-adt}). 
The distance here should be defined as the length of the shortest broken bicharacteristic segment connecting a covector to the boundary (Figure~\ref{f:cotangent-depth}). In general, our $h\MDT$ is not equivalent to the ideal direct transmission, which would contains only transmitted waves, but it may still serve as a useful proxy.

In the remainder of the section, we briefly define distance in the cotangent bundle, then use it to define the microlocal almost direct transmission $h\MDT$.

\paragraph{Distance in the Cotangent Bundle}

Let $V = \RR\times(\RR^n\setminus\Gamma)$.
For brevity, we shall simply say $\gamma\colon(s_-,s_+)\to \To^*V_\pm$ is a \emph{bicharacteristic} if it is a bicharacteristic for $\d_t^2-c^2\Delta$; is \emph{unit speed}, i.e., $dt/ds=1$ on $\gamma$; and is \emph{maximal}, i.e., cannot be extended. Here $s_\pm$ may be infinite.

\begin{figure}
	\centering
	\includegraphics{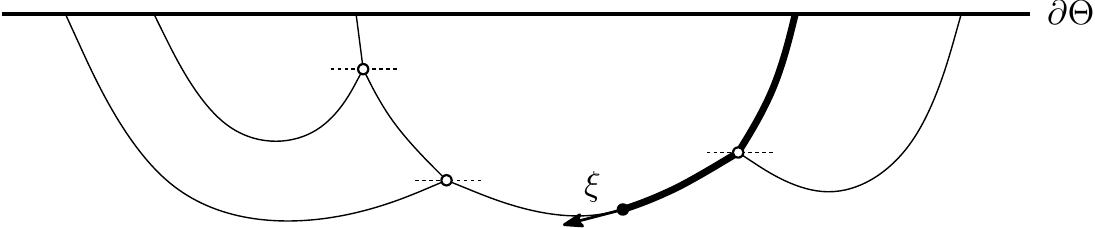}
	
	\caption{Depth of a singularity. The broken bicharacteristic segments joining covector $\xi$ to the boundary are shown, projected to $\RR^n$ (solid); they reflect and refract at interfaces (dotted lines). The depth of $\xi$ in $T^*\Theta$ is defined as the length of the shortest of these paths to the boundary (bold).}
	\label{f:cotangent-depth}
\end{figure}

A \emph{broken bicharacteristic} $\gamma\colon(s_0,s_1)\cup(s_1,s_2)\cup\dotsb\cup(s_{k-1},s_k)\to \To^*V$ is a sequence of bicharacteristics connected by reflections and refractions obeying Snell's law: for $i=1,\dotsc,k-1$,
\begin{align}
	\gamma(s_i^-), \; \gamma(s_i^+) &\in \To^*([0,2T]\times\Gamma),
	&
	(di_\Gamma)^*\gamma(s_i^-)&=(di_\Gamma)^*\gamma(s_i^+),
	\label{e:broken-bichar-interfaces}
\end{align}
where $i_\Gamma\colon\Gamma\hookrightarrow\Omega$ is inclusion. Since any broken bicharacteristic may be parameterized by time, we will often abuse notation and consider $\gamma$ as a map from $t\in\RR$ into $\To^*\mathbf(\RR^n\setminus\Gamma)$.

The \emph{distance} of a covector $\xi\in\To^*(\RR^n\setminus\Gamma)$ from the boundary of $M\subseteq\RR^n$ is
\begin{equation}
	d(\xi,\,\bdy T^*M) = \min \cset{\abs{a-b}}{\gamma(a) = \xi,\ \gamma(b)\in\bdy T^*\mathbf M},
\end{equation}
the minimum taken over broken bicharacteristics $\gamma$. Extend $d(\cdot,\bdy T^*\mathbf M)$ to all $\xi\in\To^*\RR^n$ by lower semicontinuity. In general, $d$ will not be continuous at $\To^*(\RR\times\Gamma)$.

\emph{Depth} is the same as distance, but with a sign indicating whether $\xi$ is inside or outside $M$:
\begin{equation}
	d^*_{T^*M}(\xi) = \when{
		+d(\xi,\,\bdy T^*M), &		\xi\in T^*\mathbf M,\\
		-d(\xi,\,\bdy T^*M), &			\text{otherwise}.%
	}
\end{equation}
\begin{figure}[p]
	\centering
	\includegraphics[page=1]{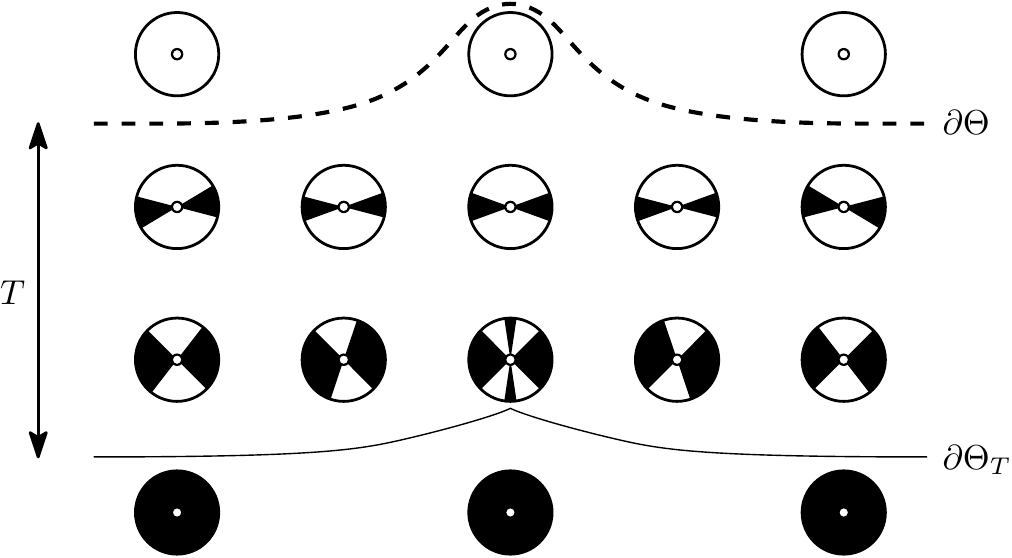}
	
	\caption{Example of a depth sublevel set $(T^*\Theta)_T$, with wave speed $c=1$. Each marked circle describes the unit covectors based at its center point: those inside $(T^*\Theta)_T$ are marked in black, those outside in white. Near the boundary, $(T^*\Theta)_T$ contains only nearly horizontal covectors, while below $\Theta_T$ it contains covectors in all directions, as the distance to the surface in any direction is greater than $T$.}
	\label{f:ml-depth-set}
\end{figure}%

\paragraph{Microlocal Almost Direct Transmission}

Let $(T^*M)_t$ be the set of covectors of depth greater than $t$ in a manifold $M$:
\begin{equation}
	(T^*M)_t = \set{\xi\in T^*M}{d^*_{T^*M}(\xi)>t}.
\end{equation}
Figure~\ref{f:ml-depth-set} illustrates $(T^*M)_t$ in a simple case. Note $(T^*M)_t\supsetneq T^*(M_t)$ in general, where $M_t$ is defined as in~\eqref{e:def-Theta-t-star}.

A \emph{microlocal almost direct transmission} of $h_0$ at time $T$ is a distribution $h\MDT$ satisfying
\begin{align}
	h\MDT  &\eqml R_Th_0 \quad\text{on $(T^*\Theta')_T$}
	&
	\WF(h\MDT) &\subseteq \clsr{(T^*\Theta'')_T}.
	\label{e:ml-adt}
\end{align}
Essentially, $h\MDT$ is any sufficiently sharp microlocal cutoff of $R_Th_0$ outside $(T^*\Theta')_T$. Note that there is a gap $G=\clsr{(T^*\Theta'')_T}\setminus (T^*\Theta')_T$ in which we do not characterize $h\MDT$; the gap is needed in case $\WF(R_Th_0)$ intersects $\bdy(T^*\Theta')_T$, since then the cutoff may not be infinitely sharp. The solutions of~\eqref{e:ml-adt} form an equivalence class modulo $\mathcal D'_G+C^\infty(\RR^n)$, since any two choices of $h\MDT$ differ exactly by a distribution with wavefront set in $G$. With this equivalence class in mind, we denote by $h\MDT$ any solution of~\eqref{e:ml-adt} and refer to it simply as \emph{the} microlocal almost direct transmission. Note that
\begin{equation}
	\begin{array}{r@{\;}c@{\;}l}
		\WF(h\MDT) 	& \subset 	& (T^*\Theta)_T\\
					&		& \quad\rotatebox{90}{$\subset$}\\
		\WF(h\DT) 	& \subset 	& T^*(\Theta_T).\\					
	\end{array}
\end{equation}
\begin{figure}
	\centering
	\hbox{}\hfill
	\subfloat[Depths of singularities in wave field]{\quad\includegraphics[page=1]{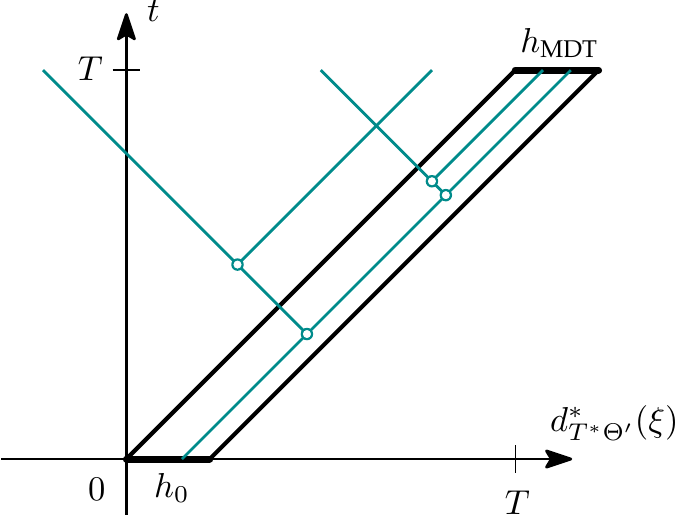}}
	\hfill
	\subfloat[Wave field of $h_0$]{\qquad\includegraphics[page=2]{Figures/MDTDepthDiagram.pdf}\qquad}
	\hfill\hbox{}
	\caption{Microlocal almost direct transmission: $h\MDT$ contains the singularities in $R_Th_0$ of depth at least $T$ in $T^*\Theta'$. (a) Depth diagram; interfaces marked with small circles. (b) Projection onto $\RR^n$; interfaces dotted.}
	\label{f:ml-mdt-depth}
\end{figure}%
It is natural to visualize $h\MDT$ with a \emph{depth diagram} plotting the depths of the wave field's singularities over time (Figure~\ref{f:ml-mdt-depth}). The depth of a singularity traveling along any broken bicharacteristic $\gamma$ is a piecewise linear function of time, with derivative $\pm1$ almost everywhere, so a depth diagram consists of line segments of slope $\pm1$. Note that the depth of $\gamma(t)$ is (up to sign) the shortest distance from $\gamma(t)$ to the surface along \emph{any} broken bicharacteristic, not only along $\gamma$.

\begin{figure}
	\centering
	
	\centering
	\hbox{}\hfill
	\subfloat[Ray configuration with one interface]{\includegraphics{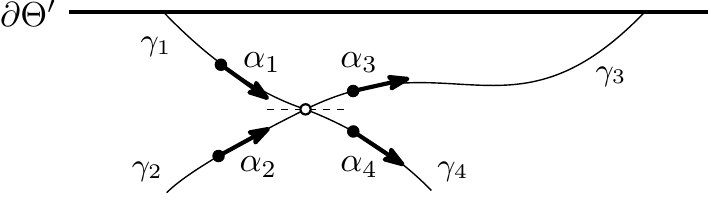}}
	\hfill
	\subfloat[Depth diagram of bicharacteristics $\gamma_i$]{\quad\includegraphics{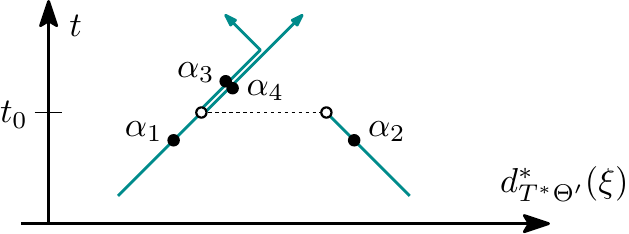}\label{f:discont-depth-rd-diagram}}
	\hfill\hbox{}

	\caption{Depth discontinuity at interfaces. (a) Covectors $\alpha_3$, $\alpha_4$ are closer to the boundary (via $\gamma_1$) than $\alpha_2$, which cannot take this path. (b) Depths of the positive bicharacteristics $\gamma_i$ through these $\alpha_i$, meeting the interface at time $t_0$. A jump occurs at the interface along either broken bicharacteristic through $\alpha_2$.}
	\label{f:discont-depth}
\end{figure}

\begin{remarks}\hfill	
\begin{itemize}
\item
	Along a broken bicharacteristic, $d^*_{T^*\Theta'}$ is often discontinuous at interfaces, as illustrated in Figure~\ref{f:discont-depth}.
	
	To see why, consider a bicharacteristic $\gamma_1$ encountering an interface; let $\gamma_3,\gamma_4$ be the reflected and transmitted bicharacteristics, and let $\gamma_2$ be the opposite incoming bicharacteristic. In general, one of the $\gamma_i$, say $\gamma_1$, provides the shortest route from the interface to the boundary. Singularities along $\gamma_3$ or $\gamma_4$ can reach the boundary along $\gamma_1$, while those along $\gamma_2$ cannot and must take a longer path. Consequently, a jump in depth occurs when passing from $\gamma_2$ to either $\gamma_3$ or $\gamma_4$.

\item
	Along a singly reflected bicharacteristic, depth does not switch from increasing to decreasing at the moment of reflection in general. Instead, depth will change from increasing to decreasing halfway along; compare the broken bicharacteristic $\gamma_1\cup\gamma_3$ in Figure~\ref{f:discont-depth}.

\item
	Depth (and hence $h\MDT$) cannot intrinsically distinguish reflections from transmissions. This is possible only under geometric assumptions ensuring that reflected waves travel toward the boundary, and transmitted waves travel away from it; e.g., $\Theta=\{x_n>0\}$ a half\-space, and $c$ a function of $x_n$ alone.
	
\end{itemize}
\end{remarks}

\subsection{Isolating the microlocal almost direct transmission}			\label{s:ml-isolating}

One of our earlier key facts, expressed in Theorem~\ref{t:basic-maf}, is that solving the (exact) scattering control equation $(I-\pi^\star R\pi^\star R)h_\infty=h_0$ for $h_\infty$ is equivalent to isolating the almost direct transmission: $\clsr\pi R_{2T}h_\infty=R_Th\DT$ (assuming $h_\infty=h_0$ on $\clsr{\Theta}$). In other words, the wave field of $h_\infty$ at $t=2T$ inside the domain $\Theta$ is exactly the almost direct transmission's wave field, undisrupted by any waves from shallower regions.

\begin{figure}
	\centering
	\mbox{}
	\hfill
	\subfloat[Wave field of $h_0$]{%
		\includegraphics[page=1]{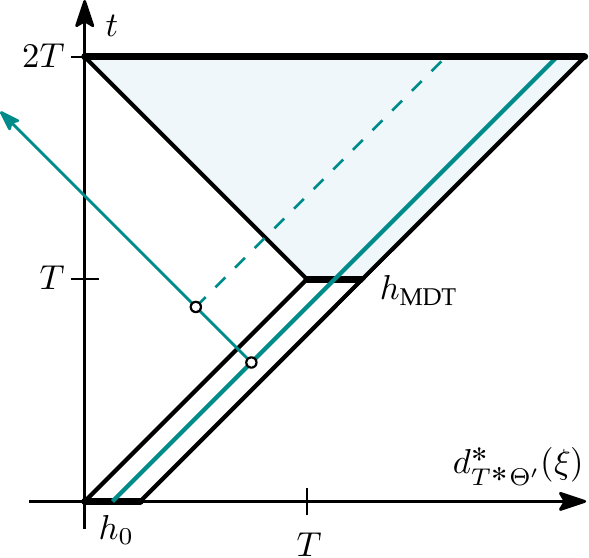}%
		\label{f:isolate-mdt-before}%
	}
	\hfill\hfill
	\subfloat[Wave field of $h_\infty$]{%
		\includegraphics[page=2]{Figures/NewIsolatingMDT.pdf}%
		\label{f:isolate-mdt-after}%
	}
	\hfill\hfill
	\mbox{}
		
	\caption{Isolating $h\MDT$. A singularity from $h_0$ travels inward, reflecting and refracting from two interfaces (indicated by open circles). The multiply-reflected ray (dotted) will enter the domain of influence of $h\MDT$ (shaded). To prevent this, $h_\infty$ must include an appropriate singularity to eliminate the multiply-reflected ray. The horizontal axis is depth in the cotangent bundle.
	}
	\label{f:isolate-mdt}
\end{figure}

Our main goal now is to consider the microlocal version of this equivalence: is solving the microlocal scattering control equation~\eqref{e:ml-maf} equivalent to isolating $h\MDT$? As before, one direction is easy: if a tail $h_\infty$ is found that isolates $h\MDT$ (in the sense that $R_{2T}h_\infty\eqml R_Th\MDT$ on $\Theta$) it is a solution of~\eqref{e:ml-maf}. The idea behind crafting such an $h_\infty$ we have seen already in Figure~\ref{f:mr-demo}: $h_\infty$ should include appropriate extra singularities that ensure singularities in the wave field of $h_0$ at depth less than $T$ do not interfere with $h\MDT$'s wave field. Figure~\ref{f:isolate-mdt} illustrates the situation.

\begin{lemma}
	Let $h_0\in \mathcal E'(\Theta'\setminus\Gamma)\oplus \mathcal E'(\Theta'\setminus\Gamma)$. Suppose $h_\infty\in \mathcal E'(\RR^n\setminus\Gamma)\oplus \mathcal E'(\RR^n\setminus\Gamma)$ isolates the microlocal almost direct transmission, in the sense that
	\begin{equation}
	 	\drestr{h_\infty}_{\Theta}\eqml \drestr{h_0}_{\Theta} \text{\ \ and } \drestr{R_{2T}h_\infty}_{\Theta}\eqml\drestr{R_Th\MDT}_{\Theta}.
	\end{equation}
	Then $h_\infty$ satisfies the microlocal scattering control equation, $(I-\sigma^\star R\sigma^\star R)h_\infty \eqml h_0$. The same holds true with $\tilde R$ replacing $R$.
	
	\label{l:isolate-mdt}
\end{lemma}

\begin{proof}
	Let $v(t,x)=(F\sigma R_{2T}h_\infty)(t-2T,x)$ be the wave field generated by $\sigma R_{2T}h_\infty$, and $\mathbf v=(v,\d_t v)$. Since $\WF(h\MDT)\subseteq \clsr{(T^*\Theta'')_T}$, propagation of singularities limits the wavefront set of $R_Th\MDT$ to $\clsr{T^*\Theta''}$, where the cutoff $\sigma$ is identity. Hence $\mathbf v$ at time $2T$ agrees with $R_Th\MDT$. Moving to time $T$, we have $\mathbf v(T,\cdot)\eqml f\MDT$; by propagation of singularities again, $\WF(\mathbf v(0,\cdot))\subseteq \clsr{T^*\Theta''}$. In particular, $\sigma^\star R\sigma Rh_\infty = \sigma^\star \mathbf v(0,\cdot)$ is smooth. We conclude that
	\begin{equation}
		\sigma^\star R\sigma^\star R h_\infty = \sigma^\star R(1 - \sigma) Rh_\infty \eqml \sigma^\star h_\infty - 0 \eqml h_\infty - h_0.
	\end{equation}
	The same argument holds with the parametrix $\tilde R$ in place of $R$.
\end{proof}

Just like Theorem~\ref{t:basic-maf}, Lemma~\ref{l:isolate-mdt} assures us that solving the microlocal scattering control equation is necessary for producing a tail $h_\infty-h_0$ that isolates $h\MDT$.

The other direction of the problem (does a solution of the microlocal scattering control equation isolate $h\MDT$?) is a more subtle question, taken up in the following sections. Our overarching goal is to show that $h\MDT$, like its non-microlocal version $h\DT$, may be found by the Neumann-type iteration~\eqref{e:ml-maf-series}.
We start by explicitly constructing a Fourier integral operator $A$ that isolates $h\MDT$, given $c$. By Lemma~\ref{l:isolate-mdt} this FIO is a microlocal inverse for $I-\sigma^\star R\sigma^\star R$. Now, Neumann iteration also provides a (formal) microlocal inverse for this operator. The existence of $A$ can be used to show that Neumann iteration isolates $h\MDT$ as well, in a principal symbol sense. This leads to the question of injectivity for $I-\sigma^\star R\sigma^\star R$, explored in greater depth in Section~\ref{s:ml-uniqueness}.

\subsection{Constructive parametrix for $I-\sigma^\star R\sigma^\star R$}						\label{s:ml-construct}

In this section, we lay out conditions on $\Theta$, $c$, $h_0$ under which we can show the existence of an $h_\infty$ isolating $h\MDT$, and thereby $I-\sigma^\star R\sigma^\star R$. The motivation for this relatively straightforward task is that it enables the study the convergence behavior of the microlocal Neumann iteration in the following section.

We start by making a number of definitions; most of which are illustrated in Figure~\ref{f:ml-constructive-terms}.%
\footnote{Note that for simplicity Figure~\ref{f:ml-constructive-terms} is not generic; in light of the remarks in \sref{f:ml-adt}, the behavior of $d^*_{T^*\Theta'}$ is typically much more complicated.}

\begin{defn}\hfill
	\renewcommand\labelenumi{(\alph{enumi})}
	\renewcommand\labelenumii{\roman{enumii}.}
	\begin{enumerate}
		\item The forward and backward \emph{microlocal domains of influence} $\mathcal D^+\MDT$, $\mathcal D^-\MDT$ are defined by:
		\begin{nalign}
			\mathcal D^-\MDT &= \set{(t,\eta)\in[0,T]\times \To^*\RR^n}{d^*_{T^*\Theta'}(\eta)>t}\!,\\
			\mathcal D^+\MDT &= \set{(t,\eta)\in[T,2T]\times \To^*\RR^n}{d^*_{T^*\Theta'}(\eta)>2T-t}\!.
		\end{nalign}
		By propagation of singularities, every $\eta\in\WF(h\MDT)$ is connected to some $\eta'\in\WF(h_0)$ by a broken bicharacteristic inside $\mathcal D^-\MDT$.
		\item A \emph{returning} bicharacteristic $\gamma:(t_-,t_+)\to\To^*(\RR^n\setminus\Gamma)$ is one that leaves $\mathcal D^-\MDT$ before $t=T$. More precisely, $\gamma(t_0)\in\mathcal D^-\MDT$ and $\lim_{t\to t_1}\gamma(t)\notin\mathcal D^-\MDT$ for some $t_0,t_1\in(t_-,t_+]$, $t_0<t_1$.

\begin{figure}
	\centering
	\includegraphics{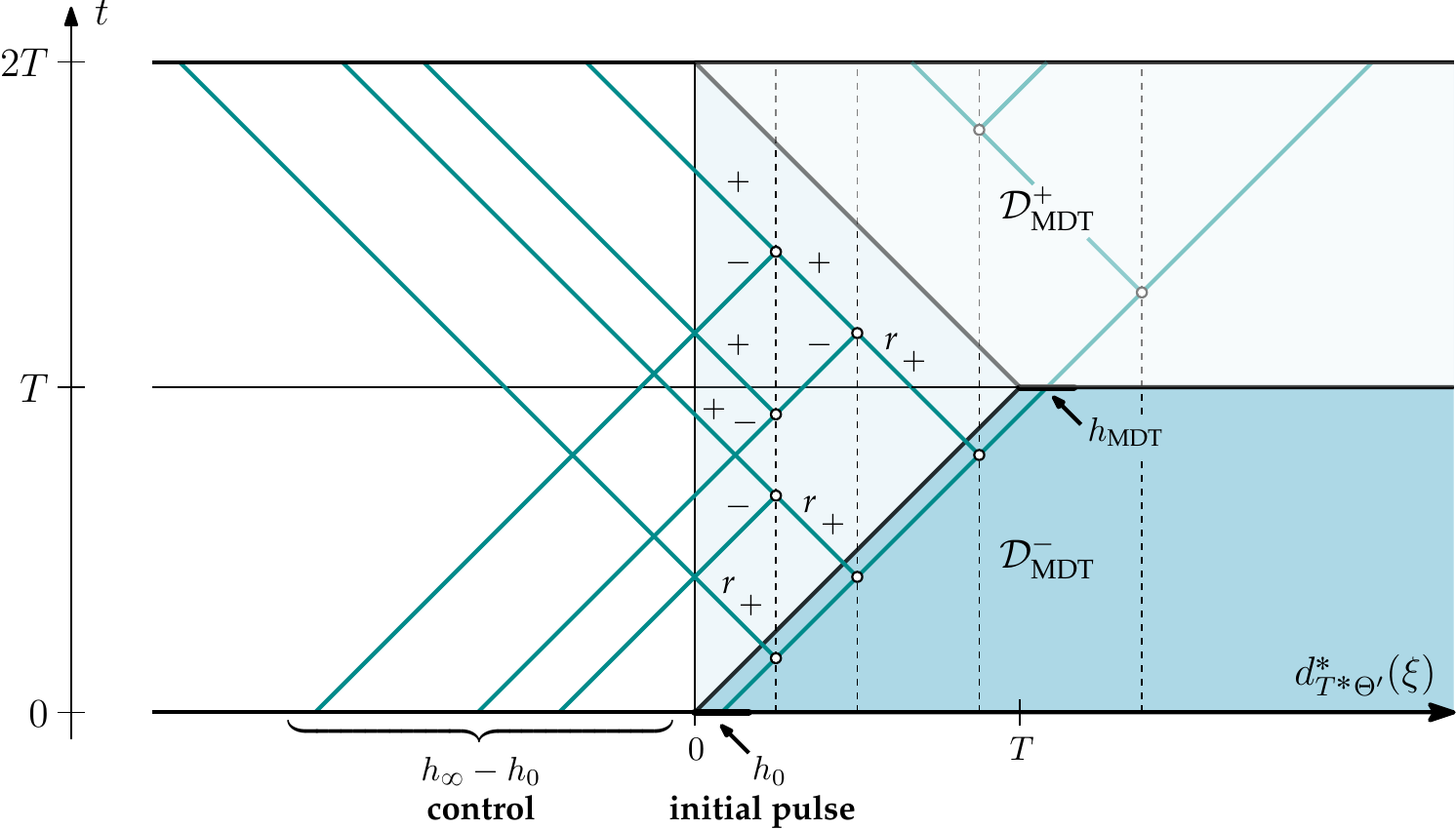}
	
	\caption{Terminology for constructing an inverse of $I-\sigma^\star R\sigma^\star R$. Here $\Theta$ is a halfspace $\{x_n>0\}$ and $c$ is piecewise constant with discontinuities along planes of constant $x_n$ (dashed lines). The wavefront set of the initial pulse $h_0$ is a single ray; to isolate $h\MDT$ three additional singularities are added to $h_\infty$ as indicated.
	Returning, $(+)$-, and $(-)$-escapable bicharacteristics are labeled \textit r, $+$, and~$-$ respectively.}
	\label{f:ml-constructive-terms}
\end{figure}

		\item Bicharacteristics $\gamma_1$, $\gamma_2$ are \emph{connected} if their union $\gamma_1\cup\gamma_2$ is a broken bicharacteristic. A bicharacteristic $\gamma_1$ terminating in an interface may have one (totally reflected), or two (reflected and transmitted) connecting bicharacteristics there. If it has two, there exists an \emph{opposite} bicharacteristic $\gamma_3$ sharing $\gamma_1$'s connecting bicharacteristics.
		\item A bicharacteristic $\gamma\colon(t_-,t_+)\to\To^*(\RR^n\setminus\Gamma)$ is \emph{$(\pm)$-escapable} if either:
		\begin{enumerate}
			\item it has \emph{escaped}: $\gamma$ is defined at $t=T\pm T$ and $\gamma(T\pm T)\notin T^*\Theta$,
		\end{enumerate}
		or recursively, after only finitely many recursions, either
		\begin{enumerate}
			\stepcounter{enumii}
			\item all of its connecting bicharacteristics at $t_\pm$ are $(\pm)$-escapable;
			\item one of its connecting bicharacteristics at $t_\pm$ is $(\pm)$-escapable, and the opposite bicharacteristic is $(\mp)$-escapable.
		\end{enumerate}
		In the final case, if the $(\pm)$-escapable connecting bicharacteristic is a reflection, we also require $c$ to be discontinuous at $\lim_{t\to t_\pm}\gamma(t)$ to ensure the reflection operator has nonzero principal symbol there.
	\end{enumerate}
\end{defn}

Roughly speaking, we may ensure a singularity traveling along a $(+)$-escapable bicharacteristic never creates a singularity in $\mathcal D^+\MDT$ by choosing $h_\infty$ appropriately. Similarly, we may \emph{produce} a singularity along a $(-)$-escapable bicharacteristic without introducing any extra singularities inside $\mathcal D^+\MDT$.

Now, if every returning bicharacteristic in $\WF(Fh_0)$ is $(+)$-escapable, we can find an $h_\infty$ isolating $h\MDT$ with an FIO construction, leading to a microlocal inverse of $I-\sigma^\star R\sigma^\star R$. Accordingly, let $\mathcal S\subset T^*\Theta'$ be the set of $\xi\notin\mathcal W$ such that every returning bicharacteristic belonging to a broken bicharacteristic through $\xi$ is $(+)$-escapable%
\footnote{Recall from~\sref{s:ml-maf} that $\mathcal W$ is the set of covectors for which the parametrix $\tilde R$ is valid.}%
. We then have the following result:

\begin{proposition}
	There is an FIO $A\colon \mathcal E'(\Theta')\oplus \mathcal E'(\Theta')\to\mathcal D'(\RR^n)\oplus\mathcal D'(\RR^n)$ of order 0 satisfying
	\begin{nalign}
		(I-\sigma^\star R\sigma^\star R) A &\eqml I
		\quad\text{on $\mathcal D'_{\mathcal S}$}.
	\end{nalign}
	Furthermore, $R_{2T}Ah_0\eqml R_T h\MDT$ for any $\WF(h_0)\subset\mathcal S$.
	
	\label{p:constructive-parametrix}
\end{proposition}

Note that, because any broken ray intersects only finitely many interfaces in the time interval $t\in[0,2T]$, the condition of being $(\pm)$-escapable is open, and in particular $\mathcal S$ is open.

\subsection{Convergence of microlocal Neumann iteration}					\label{s:ml-convergence}

With the microlocal inverse $A$ constructed for $I-\sigma^\star R\sigma^\star R$ (knowing $c$), we may now examine the behavior of Neumann iteration (which does not require knowing $c$). Recalling~\eqref{e:ml-maf-series}, define the Neumann iteration operators
\begin{equation}
	N_k = \sum_{i=0}^k (\sigma^*\tilde R)^{2i}.
	\label{e:ml-neumann-partial-sum-ops}
\end{equation}
In this section we present our main microlocal theorem: the operators $N_k$ isolate $h\MDT$ in a particular leading order sense as $k\to\infty$. Throughout, as in~\eqref{e:ml-neumann-partial-sum-ops} we substitute for $R$ the parametrix $\tilde R$ having cutoffs near glancing rays.

Since $\lim N_k$ has no microlocal interpretation in general we will instead consider the convergence of the partial sum operators' principal symbols. Technically, of course, these symbols belong to separate spaces, since each $N_k$ is associated with a different Lagrangian in general. Hence, we first define a suitable symbol space containing the principal symbols of $A$ and $N_k$, and any reasonable FIO parametrix of~\eqref{e:ml-maf}. We then introduce a natural $\ell^2$ norm, which acts as a \emph{microlocal energy norm}, on restrictions of the symbol space, and state the convergence theorem.

To describe the principal symbols of $A$ and $N_k$, we split them into finite sums of \PsiDO{}s composed with fixed unitary FIO, then record the \PsiDO{}s' principal symbols; this is a kind of polar decomposition. As is well-known (see appendix~\ref{s:parametrix-construction}), after a standard microlocal splitting of the wave equation into positive and negative wave speeds, $\tilde R$ is a sum of graph FIO $R_s$, one for each finite sequence $s\in\{\text R,\text T\}^j$, $j\geq 0$ of reflections and transmissions. For each $s$, let $C_s$ be the canonical transformation of $R_s$; form the set of all possible compositions
\begin{equation}
\sC=\set{C_{s^{(1)}}\circ\dotsb\circ C_{s^{(m)}}\!}{m\geq 0}.
\end{equation}
and enumerate this resulting set with a single index $i$:
\begin{equation}
\sC=\set{\sC_i}{i\in\mathcal I}.
\end{equation}
Hence, each composition of reflections, transmissions, and time-reversals leads to a canonical transformation $\sC_i$; in general, a single $\sC_i$ might be represented by (infinitely many) different compositions $C_{s^{(1)}}\circ\dotsb\circ C_{s^{(m)}}$. We term an FIO \emph{$\sC$-compatible} if it is associated with a finite union of $\sC_i$.

Next, fix a set of elliptic FIO $(J_i)_{i\in\mathcal I}$ associated with the $\sC_i$ that are microlocally unitary, that is, $J_i^*J_i\eqml I$.
Any $\sC$-compatible FIO $\mathcal Z$ may now be written in the form $\mathcal Z = \sum_{i\in\mathcal I} P_iJ_i$ for appropriate \PsiDO{}s $P_i$. Define the \emph{principal symbol of $\mathcal Z$ with respect to $(J_i)_{i\in\mathcal I}$} to be the tuple of principal symbols of the $P_i$, restricted to the cosphere bundle:
\begin{equation}
	\sigma_0=\sigma_0(\mathcal Z) = \big(\sigma_0(P_i)\big)_{i\in\mathcal I} \in C^\infty\big(S^*(\bRR^n\setminus\bGamma)\times \mathcal I\big),
\end{equation}
The boldface $\bRR^n\setminus\bGamma$ denotes a doubled space containing two copies of $\RR^n\setminus\Gamma$; due to the microlocal splitting this is a natural space for Cauchy data. For convenience, we consider the tuple $\sigma_0$ as a function on a single domain having one copy of $S^*(\bRR^n\setminus\bGamma)$ for each $i\in\mathcal I$. Note that a full symbol for $\mathcal Z$ (not needed here) could be defined analogously.

Now, for $\eta\in S^*(\bRR^n\setminus\bGamma)$ define
\begin{equation}
	\mathcal G_{\eta} = \set{(\sC_i(\eta),i)}{i\in\mathcal I,\,\eta\in\mathcal D(\sC_i)} \subset S^*(\bRR^n\setminus\bGamma)\times \mathcal I,	
\end{equation}
where $\mathcal D(\sC_i)$ is the domain of $\sC_i$. That is, $\mathcal G_{\eta}$ contains all covectors reachable from $\eta$, together with a knowledge of the paths $i$ taken for each.

Consider the restriction of a principal symbol $\sigma_0(\mathcal Z)$ to the space $\mathcal G_\eta$. Here, $\sigma_0(\mathcal Z)$ may be viewed both as an element of $\mathcal G_\eta$ and the unique linear operator on $\mathcal G_\eta$ defined by left-composition:
\begin{equation}
	\sigma_0(\mathcal Z)\colon \sigma_0(\mathcal Z')\big|_{\mathcal G_\eta} \mapsto \sigma_0(\mathcal Z\mathcal Z')\big|_{\mathcal G_\eta},
\end{equation}
for $\sC$-compatible FIOs $\mathcal Z'$. The composition $\mathcal Z\mathcal Z'$ is well-defined as an FIO since all operators involved are sums of graph FIO.

The key idea is that the $\ell^2$ norm on $\mathcal G_\eta$ provides a natural microlocal energy operator norm for $\mathcal Z$. In particular (see Lemma~\ref{l:ml-energy-conservation} in~\sref{s:ml-proofs}), just as $\norm R=1$ w.r.t.\ the exact operator norm, so composition with $\tilde r$ has operator norm 1 on the $\ell^2(\mathcal G_\eta)$ principal symbol space, in the absence of glancing ray cutoffs.
Combining this norm with existence of an $\ell^2$-bounded microlocal inverse of $I-\sigma^\star\tilde R\sigma^\star\tilde R$, we can prove principal symbol convergence for Neumann iteration. In the limit, furthermore, the wave field produced by Neumann iteration at $t=2T$ inside $\Theta'$ agrees with that produced by the given microlocal inverse, modulo $C^\infty$.

\begin{theorem}
	Suppose $\tilde{\mathcal S} \subset\To^*(\bRR^n\setminus\bGamma)$ is a conic set on which $I-\sigma^\star \tilde R\sigma^\star \tilde R$ has a $\sC$-compatible right parametrix $\tilde A$ on $\tilde{\mathcal S}$; that is, $(I-\sigma^\star \tilde R\sigma^\star \tilde R)\tilde A\eqml I$ on $\tilde{\mathcal S}$. 
	Assume that $\sigma_0(\tilde A)$ restricts to a bounded operator on $\ell^2(\mathcal G_\eta)$ for each $\eta\in\tilde{\mathcal S}\cap S^*(\bRR^n\setminus\bGamma)$.
	
	Then, for every $\eta\in\tilde{\mathcal S}\cap S^*(\bRR^n\setminus\bGamma)$, the Neumann series principal symbols $\sigma_0(N_k)$ converge to some $n_\infty\in\ell^2(\mathcal G_{\eta})$. Furthermore, $\sigma_0(\tilde RN_k)\to \sigma_0(\tilde R\tilde A)$ in $\ell^2(\mathcal G_\eta\cap S^*\bTheta')$.
		
	\label{t:ml-convergence}
\end{theorem}

Of course, we have in mind for $\tilde A$ the concrete parametrix $A$ of Proposition~\ref{p:constructive-parametrix}. This parametrix is $\sC$-compatible~(cf.~\sref{s:ml-construct-proof}); it also has finitely many graph FIO components, so it is a bounded operator on $\ell^2(\mathcal G_\eta)$. Taking $\tilde A=A$ we have the following direct corollary of Proposition~\ref{p:constructive-parametrix} and Theorem~\ref{t:ml-convergence}:

\begin{corollary}
	For every $\eta\in\mathcal S\cap S^*(\bRR^n\setminus\bGamma)$, the Neumann series principal symbols $\sigma_0(N_k)$ converge in $\ell^2(\mathcal G_{\eta})$. Furthermore, $\sigma_0(RN_k)\to \sigma_0(RA)$ in $\ell^2(\mathcal G_\eta\cap S^*\bTheta')$.
		
	\label{c:ml-convergence-mdt}
\end{corollary}

According to Proposition~\ref{p:constructive-parametrix}, we have $R_{2T}Ah_0\eqml R_T h\MDT$ on $T^*\bTheta'$. Hence, the corollary implies that to leading order, the same is true of the $N_k$ as $k\to\infty$; they also isolate $h\MDT$.

Note that Theorem~\ref{t:ml-convergence} does not claim that the principal symbol limit $n_\infty$ is itself the principal symbol of some FIO. In particular, the support of $n_\infty$ on some fiber $\mathcal G_\eta$ may be infinite, that is, $n_\infty$ maps $\eta$ to infinitely many singularities. In this case it is not obvious that $n_\infty$ corresponds to any FIO. Conversely, if $n_\infty$ is smooth and its restriction to every $\mathcal G_\eta$ has finite support, an FIO $N_\infty$ with principal symbol $n_\infty$ is easily constructed.

\subsection{Microlocal uniqueness}						\label{s:ml-uniqueness}

The previous two sections treated the solution of $(I-\sigma^\star R\sigma^\star R)h_\infty\eqml h_0$, both constructively and iteratively. In this section we turn to the question of uniqueness; i.e.~the solutions of $g\eqml \sigma^\star R\sigma^\star Rg$. As we will see, the microlocal scattering control equation displays two distinct kinds of nonuniqueness: a normal type, due to diving rays and total reflections, and a pathological type, involving an infinite-energy sequence of reinforcing singularities.

The first type is analogous to the nonuniqueness seen in the exact setting. In the exact case, the kernel $\mathbf G$ of $I-\pi^\star R\pi^\star R$ consists only of initial data whose wave fields are supported outside $\Theta$, due to unique continuation. In other words, no waves can enter $\Theta$, completely reflect, and leave in finite time $2T$. Microlocally, however, there is a much richer space of completely reflecting wave fields, including totally reflecting and diving rays. Note that these rays do not affect $\trestr{h_\infty}_{\Theta'}$ and in particular do not interfere with the wave field of $h\MDT$, up to smoothing.

The second type of nonuniqueness is unique to the microlocal setting. In this case, the wave field produced by initial data $g$ does include singularities inside $\bTheta'$ at time $2T$, which $\sigma^\star$ cuts off. The (microlocal) energy lost in this cutoff must be replenished by a second singularity in the initial data, which in turn must be replenished a third, and so on, necessitating an infinite chain of singularities. Since $Rg$ is not smooth in $\bTheta'$, the converse of Lemma~\ref{l:isolate-mdt} fails.

In the following examples, we illustrate these two nonuniqueness types at length.

\begin{figure}
\centering

\mbox{}
\subfloat[An element in the microlocal kernel of $(I-\sigma^\star R\sigma^\star R)$]{%
	\qquad\qquad\includegraphics[page=1]{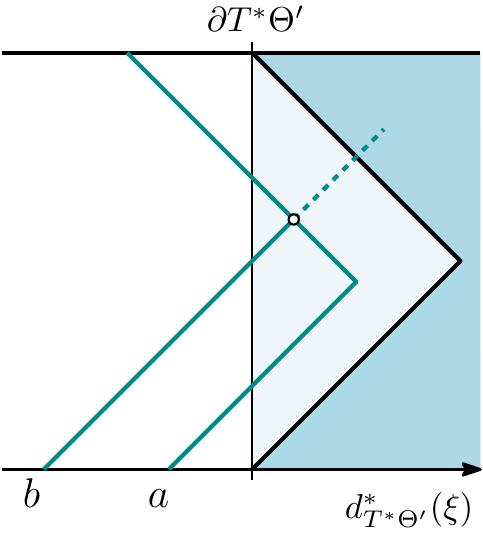}\qquad\qquad%
	\label{f:regular-ml-nonuniqueness-kernel}%
}
\hfill
\quad\includegraphics[page=3]{Figures/NewNormalNonuniqueness.pdf}%
\hfill
\subfloat[Nonuniqueness for microlocal scattering control]{%
	\qquad\includegraphics[page=2]{Figures/NewNormalNonuniqueness.pdf}\qquad%
	\label{f:regular-ml-nonuniqueness-h0}%
}
\mbox{}

\caption{Regular nonuniqueness for microlocal scattering control; interfaces are marked with discs. (a) An appropriate combination of singularities at $a$ and $b$ is smooth on the dashed bicharacteristic and reflects from $\Theta$. (b) A singularity from $h_0$ can be cancelled at either $a$ or $b$.}

\label{f:regular-ml-nonuniqueness}
\end{figure}

\begin{example}
Figure~\subref*{f:regular-ml-nonuniqueness-kernel} presents an element of the microlocal kernel of $(I-\sigma^\star R\sigma^\star R)$, with a diving or totally reflecting ray and one interface. If $g$ has singularities at $a$ and $b$ satisfying an appropriate pseudodifferential relation, its wave field will be smooth along the dashed ray. Thus the cutoffs $\sigma^\star$ have no effect, and $\sigma^\star R\sigma^\star R g\eqml RRg=g$, implying $(I-\sigma^\star R\sigma^\star R)g\eqml 0$.

Figure~\subref*{f:regular-ml-nonuniqueness-h0} illustrates how this lack of injectivity leads to multiple solutions $h_\infty$. Here, a stray ray from the direct transmission can be cancelled by an appropriate singularity at either $a$ or $b$, or a linear combination of them. The proof of Theorem~\ref{t:ml-convergence} shows that Neumann iteration converges in principal symbol to a solution operator having ``least microlocal energy'' in the sense of a weighted $\ell^2$ norm on its principal symbol.

\label{x:ml-nonuniqueness-1}
\end{example}

\begin{example}

\begin{figure}
\centering

\includegraphics[page=1]{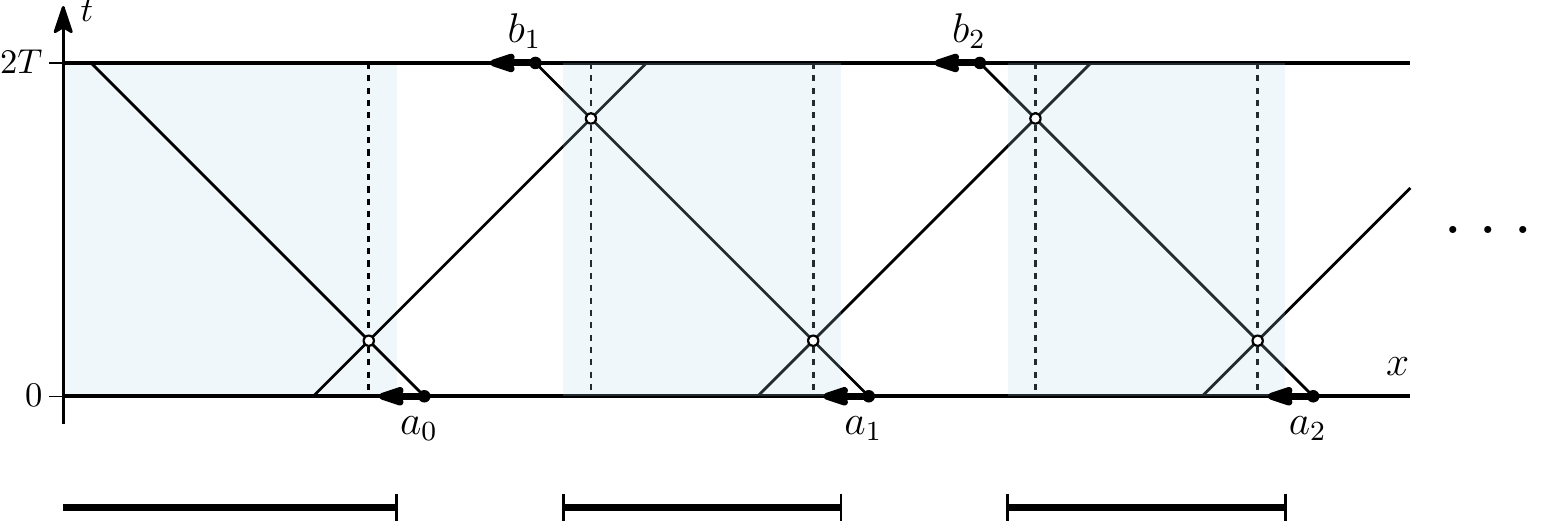}

\caption[One-dimensional example of the second type of uniqueness]{One-dimensional example of pathological nonuniqueness. $\Theta$ is a union of infinitely many intervals; dotted lines are interfaces. The pattern continues indefinitely as $x\to +\infty$.}

\label{f:pathological-ml-nonuniqueness-1D}
\end{figure}

\begin{figure}
\centering

\includegraphics{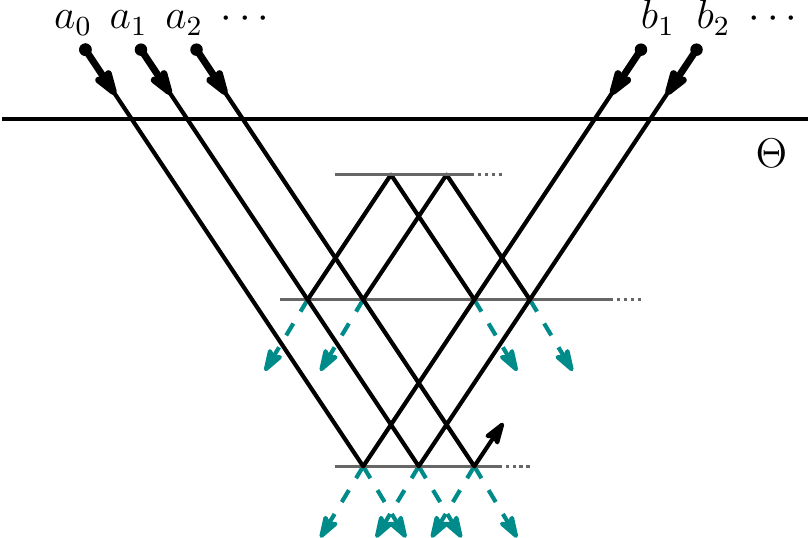}

\caption{Two-dimensional version of Figure~\ref{f:pathological-ml-nonuniqueness-1D}. Thin lines represent interfaces; dashed rays never reach the surface. Total internal reflection occurs at the upper interface.}

\label{f:pathological-ml-nonuniqueness-2D}
\end{figure}

Figure~\ref{f:pathological-ml-nonuniqueness-1D} shows a one-dimensional setup exhibiting the second type of nonuniqueness. While this example is contrived, Figure~\ref{f:pathological-ml-nonuniqueness-2D} shows how an equivalent and more realistic higher-dimensional version may be constructed. (Both examples involve non-compact domains, and we conjecture noncompactness is required for this type of nonuniqueness.)

Here $\Theta$ consists of an infinite series of disconnected open intervals $(-\infty,w_0)\cup(v_1,w_1)\cup(v_2,w_2)\cup\dotsb$. On each finite interval $c$ has two jump discontinuities; assume $\Theta'$ is sufficiently close to $\Theta$ to contain these singularities.
Two sequences of unit covectors $\{a_i\}_{i=0}^\infty,\,\{b_i\}_{i=1}^\infty\subset S^*\bTheta^\star\setminus\mathcal W$ are chosen so that the canonical relation of $\sigma^\star\tilde R$ sends $a_i$ to $\{b_i,b_{i+1}\}$ and $b_i$ to $\{a_{i-1},a_i\}$. 

We now construct a $g$ in the microlocal kernel of $I-\sigma^\star R\sigma^\star R$ with an infinite sequence of singularities at $a_0,a_1,a_2,\dotsc$. First, note that the canonical relation of $\sigma^\star \tilde R\sigma^\star \tilde R$ sends $a_i$ $(i>1)$ to $\{a_{i-1},a_i,a_{i+1}\}$. Suppose now that we choose some initial data $g$ with a singularity at $a_0$. After applying $\sigma^\star R\sigma^\star R$, some portion of this singularity's amplitude will be lost due to the $\sigma^\star$ cutoffs. We may, however, restore the lost amplitude by adding an appropriate singularity to $g$ at $a_1$. In turn, some of this new singularity's amplitude will be lost under $\sigma^\star R\sigma^\star R$, which we make up for with an appropriate singularity at $a_2$, and so on.

Rigorously, decompose $\sigma^\star \tilde R\sigma^\star\tilde R$ near each $a_i$ as the sum of three graph FIO $A_{-1}$, $A_0$, $A_1$ whose canonical graphs map $a_i$ to $a_{i-1}$, $a_i$, and $a_{i+1}$ respectively. Modify $A_0$, say, by a smooth operator so that $\sigma^\star R\sigma^\star R=A_{-1}+A_0+A_1$ exactly. It can be shown (cf.~\eqref{e:geometric-ps-refl-trans}) that the $A_k$ are elliptic. 

Now, choosing any $g_0\in L^2(\Theta^\star)$ with $\WF(g_0)=\RR^+ a_0$, we look for $g_i$, $i=1,2,\dotsc$ with wavefront sets at $\RR^+a_i$ such that the sum $g=\sum g_i$ satisfies $(I-\sigma^\star R\sigma^\star R)g\eqml 0$. This leads to the infinite matrix equation
\begin{equation}
	\setlength\arraycolsep{3pt}
	\left(
	I -
	\begin{bmatrix*}[l]
		A_0				& A_{-1}	&		&  \\
		A_1				& A_0	& A_{-1}	& \phantom\ddots\\
		\phantom{A_{-1}}	& A_1	& A_0	& \ddots\\
						&		& \ddots	& \ddots
	\end{bmatrix*}
	\right)
	\begin{bmatrix}
		g_0				\\
		g_1 \vphantom\vdots \\
		g_2 \vphantom\vdots \\
		\vdots
	\end{bmatrix}
	\eqml
	0.
	\label{e:path-eg-matrix-eqn}
\end{equation}

By ellipticity,~\eqref{e:path-eg-matrix-eqn} has a solution, namely $g_{i+1} \eqml (A_{-1})^{-1} \big((I-A_0)g_i+A_1g_{i-1}\big)$. To construct an associated $g$, we use the fact that the $\{a_i\}$ are discrete in $S^*(\bTheta^\star)$ (which implies $\Theta$ is unbounded). 

Each $g_i$ is locally $L^2$, so after multiplying by a smooth cutoff near the base point of $a_i$, we may assume $g_i\in L^2$. Applying radial cutoffs in the Fourier domain, we may assume that $\norm{g_i}_{L^2}\leq 2^{-i}$, so $g=\sum g_i$ converges in $L^2$. Defining $g_{-1}=0$, consider
\begin{equation}
	(I - \sigma^\star R\sigma^\star R)g = \sum_{i=0}^\infty - A_{1}g_{i-1} + (I-A_0)g_i - A_{-1}g_{i+1}.
	\label{e:error-path-nonuniqueness}
\end{equation}
Each summand is smooth by construction, and compactly supported near the base point of $a_i$. Because the $\{a_i\}$ are discrete, we can ensure only finitely many summands of~\eqref{e:error-path-nonuniqueness} are nonzero at any given point. Hence the entire sum is smooth, showing $g$ is in the microlocal kernel of $I-\sigma^\star R\sigma^\star R$. As expected, $Rg$ is not smooth in $\Theta'$; it is not hard to see it must be singular at every $b_i$. Hence, solving $(I-\sigma^\star R\sigma^\star R)h_\infty\eqml h_0$ is not sufficient for isolating $h\MDT$.

\label{x:ml-nonuniqueness-2}
\end{example}

\paragraph{Uniqueness and Isolating $h\MDT$}

We now close the circle, and return to the question of whether solving $(I-\sigma^\star R\sigma^\star R)h_\infty\eqml h_0$ is equivalent to isolating $h\MDT$. Of our two types of nonuniqueness, only the second interferes with isolating $h\MDT$. We may rule it out, to leading order, by assuming the same kind of microlocal energy boundedness seen earlier in Theorem~\ref{t:ml-convergence}: namely, $\ell^2$ boundedness of the parametrix's principal symbol. Assuming this condition, we reach a partial converse of Lemma~\ref{l:isolate-mdt}: a solution of the microlocal scattering control equation isolates $h\MDT$ to leading order as long as this is possible. We frame our proposition as a uniqueness result.

\begin{proposition}
	\nobelowdisplayskip
	Suppose $B_1,B_2$ are $\sC$-compatible microlocal right inverses for $I-\sigma^\star \tilde R\sigma^\star \tilde R$ on a conic subset $\tilde{\mathcal S}\subset\To^*(\bRR^n\setminus\bGamma)$. If their principal symbols restrict to elements of $\ell^2(\mathcal G_\eta)$ for all $\eta\in\tilde{\mathcal S}$,
	\begin{align}
		\trestr{\tilde RB_1h_0}_{\bTheta'}&\eqml\trestr{\tilde RB_2 h_0}_{\bTheta'}
		\bmod{H^{s+1}(\bRR^n\setminus\bGamma)}
		& \text{for all } h_0\in H^s(\bRR^n\setminus\bGamma)\cap\mathcal D'_{\tilde{\mathcal S}}.
	\end{align}
	\label{p:ml-uniqueness}
\end{proposition}

In particular, as long as there is some ``finite microlocal energy'' parametrix isolating $h\MDT$ on a conic set $\tilde{\mathcal S}\subset\To^*(\bRR^n\setminus\bGamma)$, all other finite microlocal energy parametrices on $\tilde{\mathcal S}$ also isolate $h\MDT$.

\subsection{Proofs}						\label{s:ml-proofs}

\subsubsection{Microlocal convergence (\sref{s:ml-convergence})}

	The major task in proving Theorem~\ref{t:ml-convergence} is to show that composition with $\tilde R$ has operator norm at most 1 on $\ell^2(\mathcal G_{\eta})$ for any $\eta$ --- a microlocal version of energy conservation. We begin with its proof.
	
	To present the energy conservation lemma, note that composition with $\tilde R$ is linear and well-defined on $\sC$-compatible FIO. It therefore induces a linear operator $\tilde r$ on their principal symbols in the space $C^\infty\left(S^*(\bRR^n\setminus\bGamma)\times\mathcal I\right)$. Since $\mathcal G_\eta$ is closed under the canonical relation of $\tilde R$, operator $\tilde r$ restricts to a linear operator on $\ell^2(\mathcal G_\eta)$ for any $\eta\in S^*(\bRR^n\setminus\bGamma)$. 
	
\begin{lemma}[Microlocal Energy Conservation]
	Let $\eta\in S^*(\bRR^n\setminus\bGamma)$. Then $\norm{\tilde r}\leq 1$ with respect to the operator norm on $\ell^2(\mathcal G_\eta)$.
	\label{l:ml-energy-conservation}
\end{lemma}
\begin{proof}	
	First, assume that there are no cutoffs in the parametrix $\tilde R$ due to glancing rays originating in $\mathcal G_\eta$.	In this case, $\tilde R^2\eqml R^2=I$, so $\tilde r^2=I$ likewise. If $\tilde r$ were self-adjoint, it would follow that $\norm{\tilde r}_{\ell^2}=1$. Certainly $\tilde R$ is microlocally self-adjoint, since $\tilde R^*\eqml R^*=R\eqml \tilde R$. This property does not immediately carry over to $\tilde r$ due to the presence of Maslov factors; fortunately, it is still possible to show $\tilde r$ is self-adjoint.
	
	Let $(\alpha,i),\,(\beta,j)\in \mathcal G_{\eta}$, and let $e_{\alpha,i},\,e_{\beta,j}\in \ell^2(\mathcal G_\eta)$ be the vectors having 1 in the $(\alpha,i)$ or $(\beta,j)$ position respectively and zeros elsewhere. It suffices to show that
	\begin{equation}
		\form{\tilde r e_{\alpha,i},\,e_{\beta,j}} = \overline{\form{\tilde r e_{\beta,j},\,e_{\alpha,i}}}.
		\label{e:tilde-r-sa-comparison}
	\end{equation}
	To compute each side, we choose \PsiDO{}s $P,P'\in\Psi^0$ with $\sigma_0(P)=\sigma_0(P')=1$ near $\alpha,\,\beta$ respectively. Decompose
	\begin{align}
		\tilde RPJ_i &\eqml \sum_{j\in\mathcal I} Q_j J_j,		&
		\tilde RP'J_j &\eqml \sum_{i\in\mathcal I} Q'_iJ_i.
		\label{e:tilde-r-sa-decomp}
	\end{align}
	The left- and right-hand sides of~\eqref{e:tilde-r-sa-comparison} then become $\overline{\sigma_0(Q_j)(\beta)}$ and $\sigma_0(Q'_i)(\alpha)$.
	
	If there is no $C_s$ carrying $(\alpha,i)$ to $(\beta,j)$ (that is, $C_s(\alpha)=\beta$ and $C_s\circ\sC_i=\sC_j$ on their common domain of definition), there is also no $C_{s'}$ carrying $(\beta,j)$ to $(\alpha,i)$, and vice versa. In this case, both sides of~\eqref{e:tilde-r-sa-comparison} are zero. Otherwise, there are unique $C_s$ and $C_{s'}$ satisfying the above; let $R_s$ and $R_{s'}$ be the microlocal restrictions of $\tilde R$ to each of these canonical relations near $\alpha$ and $\beta$ respectively. We may replace $\tilde R$ in the first and second equations of~\eqref{e:tilde-r-sa-decomp} by $R_s$ and $R_{s'}$, respectively. Furthermore, $R_{s'}\eqml R_s^*$ since $\tilde R$ is microlocally self-adjoint and $C_{s'}=(C_{s})^{-1}$.
	
	Now we apply singular symbol calculus (see~\cite{C}) to both sides of the first equation of~\eqref{e:tilde-r-sa-decomp} and evaluate at $\beta$ and $\alpha$. Let lowercase letters ($r_s$, $j_i$, etc.) denote singular principal symbols (of $R_s$, $J_i$, etc.). This yields
	\begin{nalign}
		r_s(\beta)j_i(\eta)i^{\kappa(\ud\sC_i(V_\eta),\,V_\alpha,\,\ud C_s^{-1}(V_\beta))/2} &= q_j(\beta) j_j(\eta),	\\
		r_{s'}(\alpha)j_j(\eta)i^{\kappa(\ud\sC_j(V_\eta),\,V_\beta,\,\ud C_s(V_\alpha))/2} &= q'_i(\alpha) j_i(\eta),
	\end{nalign}
	where $V_\gamma$ denotes the vertical subspace in $T_\gamma T^*\mathbf (\bRR^n\setminus\bGamma)$, and $\kappa$ is the Kashiwara index~\cite{LV,S}. Solving for $\overline{q_j(\beta)}$ and $q'_i(\alpha)$ we obtain
	\begin{nalign}
		\form{\tilde re_{\alpha,i},\,e_{\beta,j}} &= \overline{q_j(\beta)} = \overline{r_s(\beta)}\frac{\;\overline{j_i(\eta)}\;}{\overline{j_j(\eta)}} i^{-\kappa(\ud\sC_i(V_\eta),\,V_\alpha,\,\ud C_s^{-1}(V_\beta))/2},			\\
		\overline{\form{\tilde r e_{\beta,j},\,e_{\alpha,i}}} &= q'_i(\alpha) = r_{s'}(\alpha)\frac{j_j(\eta)}{j_i(\eta)} i^{\kappa(\ud\sC_j(V_\eta),\,V_\beta,\,\ud C_s(V_\alpha))/2}.
	\end{nalign}
	Comparing terms, $\overline{r_s(\beta)}=r_{s'}(\alpha)$ since $R_{s'}=R_s^*$, and similarly $\overline{j_i(\eta)}/\overline{j_j(\eta)}=j_j(\eta)/j_i(\eta)$, because $J_i$ being unitary implies $\abs{j_i}=1$. As for the Kashiwara indices, since $\kappa$ is coordinate-invariant and alternating,
	\begin{nalign}
		\kappa(\ud\sC_i(V_\eta),V_\alpha,\ud C_s^{-1}(V_\beta))
												&= \kappa(\ud\sC_j(V_\eta),\ud C_s(V_\alpha),V_\beta)			\\
												&= -\kappa(\ud\sC_j(V_\eta),V_\beta,\ud C_s(V_\alpha)).
	\end{nalign}
	The conclusion is that $\tilde r$ is self-adjoint, and therefore $\norm{\tilde r}=1$, since $\norm{\tilde r^2}=\norm{I}=1$.
	
	In the presence of near-glancing rays in $\mathcal G_\eta$, the parametrix constructed in appendix~\ref{s:parametrix-construction} includes pseudodifferential cutoffs away from glancing rays (in constructing $\varphi^+$ and $\JBS$). In a neighborhood of any $\alpha\in\mathcal G_\eta$ for which some broken ray is at least partially cut off, $\tilde R$ is microlocally equivalent to a composition of propagators and pseudodifferential cutoffs
	\begin{equation}
		\tilde R \eqml \upsilon\circ \tilde R_{t_m} \circ P_{m-1} \circ \tilde R_{t_{m-1}} \circ \dotsb \circ P_{1} \circ \tilde R_{t_1},
	\end{equation}
	where $t_1+\dotsb+t_m=2T$ and $P_1,\dotsc,P_{m-1}\in\Psi^0$ have principal symbols of magnitude at most 1, and none of the intermediate propagators $\tilde R_{t_k}$ involve glancing ray cut offs when $\tilde R$ is restricted to the neighborhood of $\alpha$.
	
	For each $k=0,\dotsc,m$, we let $\sC^{(k)}=\{\smash{C^{\smash{(k)}}_s}\circ \sC_i\}$ be the set of compositions of $\sC_i$'s with canonical graphs $C^{\smash{(k)}}_s$ defined as in~\sref{s:ml-convergence} but with $2T$ replaced by $t_1+\dotsb+t_k$. Naturally, $\sC^{(0)}\!=\sC^{(m)}\!=\sC$. Choose sets of corresponding unitary operators $\{J^{\smash{(k)}}_i\}$ as before for each $k$. Then composition by each $\tilde R_{t_k}$ sends $\sC^{(k)}$- to $\sC^{(k+1)}$-compatible FIO, and as before induces a map between their principal symbol spaces; the argument above shows it is an isometry with respect to the $\ell^2$ norms.
	 
	 Composition with the pseudodifferential cutoffs $P_k$ acts by pointwise multiplication by $p_k$ on these $\ell^2$ spaces, and hence has operator norm at most 1. Since $\sC^{(m)}=\sC$, operator $\tilde r$ is given by the composition of all these operators $\tilde r_{t_m}\circ p_{m-1}\circ \tilde r_{t_{m-1}}\circ\dotsb$, and thus $\norm{\tilde r}\leq 1$.
\end{proof}

\begin{proof}[Proof of Theorem~\ref{t:ml-convergence}]	
	We begin with the first statement of the theorem: convergence of the $N_k$'s principal symbols in $\ell^2(\mathcal G_{\eta})$.
	
	 Since composition with $\sigma^\star$ multiplies principal symbols pointwise by $\sigma^\star$, it is a linear operator on $\ell^2(\mathcal G_\eta)$ with norm at most 1. Therefore $\sigma^\star\tilde r\sigma^\star\tilde r$, the operation of principal symbol composition with $\sigma^\star\tilde R\sigma^\star\tilde R$, has norm at most 1 as an operator on $\ell^2(\mathcal G_\eta)$. 

	Let $n_k$, $\tilde a$, and $i$ denote the principal symbols of $N_k$, $\tilde A$, and the identity with respect to the $J_i$. We will see that $\tilde a$'s existence implies the convergence of $n_k$ by the spectral theorem, applied to a symmetrization of $\sigma^\star\tilde r$.
	
	Restricting to $\mathcal G_\eta$, suppose
	\begin{equation}
		(I-\sigma^\star \tilde r\sigma^\star\tilde r)u=i \text{\qquad for some } u\in \ell^2(\mathcal G_\eta).
		\label{e:posited-u}
	\end{equation}
	Then $u=i+v$ for some $v$ in the range of $\sigma^\star$. In particular, $v$ is supported in $\mathcal G_\eta\cap T^*\bTheta'^\star$. Solving~\eqref{e:posited-u} for $w=v/\sqrt{\sigma^\star}$ gives
	\begin{align}
		(I-\sqrt{\sigma^\star} \tilde r\sigma^\star\tilde r \sqrt{\sigma^\star})\frac{v}{\sqrt{\sigma^\star}}	&=\sqrt{\sigma^\star}\tilde r\sigma^\star\tilde ri.
		\label{e:posited-w}
	\end{align}
	As the process is reversible, $u$ is a solution of~\eqref{e:posited-u} if and only if $w=(u-i)/\sqrt{\sigma^\star}$ solves~\eqref{e:posited-w} in the weighted space $\ell^2(\mathcal G_\eta\cap T^*\bTheta'^\star,\sigma^\star)$. Now, if there is any solution to~\eqref{e:posited-w}, applying Lemma~\ref{l:only-neumann} to the self-adjoint operator $\sqrt{\sigma^\star}\tilde r \sqrt{\sigma^\star}$ shows that the Neumann series
	\begin{equation}
		w_0 = \sum_{k=0}^\infty \big[\,\sqrt{\sigma^\star}\tilde r\sigma^\star \tilde r\sqrt{\sigma^\star}\,\big]^k \sqrt{\sigma^\star}\tilde r\sigma^\star\tilde ri
	\end{equation}
	converges in $\ell^2(\mathcal G_\eta\cap T^*\bTheta'^\star,\sigma^\star)$ to the minimal-norm solution of~\eqref{e:posited-w}. The corresponding $u_0=i+\sqrt{\sigma^\star} w_0\in\ell^2(\mathcal G_\eta)$ is exactly $\lim n_k$.
	
	In particular, $u=\tilde a$ is a solution of~\eqref{e:posited-u} and it is in $\ell^2(\mathcal G_\eta)$ since its support in $\mathcal G_\eta$ is finite. Hence, the Neumann series partial sum principal symbols converge in $\ell^2(\mathcal G_\eta)$. They may not converge to $\tilde a$, as $I-\sigma^\star \tilde r\sigma^\star \tilde r$ may have a nontrivial nullspace.

	Consider this nullspace. Suppose $(I-\sigma^\star\tilde r\sigma^\star\tilde r)g=0$ for some $g\in\ell^2(\mathcal G_\eta)$, so that $g = \sigma^\star\tilde r\sigma^\star\tilde r g$. But since the operator norms of $\sigma^\star$ and $\tilde r$ are at most 1, we must have 
	\begin{equation}
		\norm{g} = \norm{\tilde r g} = \norm{\sigma^\star\tilde r g} = \norm{\tilde r\sigma^\star \tilde r g} = \norm{\sigma^\star\tilde r\sigma^\star\tilde rg}. 
	\end{equation}
	The second equality implies that $\tilde r g$ is supported in $T^*\bTheta'^\star$. Taking $g=\tilde a-\lim n_k$, we conclude $\tilde r a$ and $\tilde r\circ \lim n_k$ are equivalent in $T^*\bTheta'^\star$, finishing the proof.
\end{proof}

\subsubsection{Constructive parametrix (\sref{s:ml-construct})}       \label{s:ml-construct-proof}

\begin{proof}[Proof of Proposition~\ref{p:constructive-parametrix}]
	The proof is purely technical, specifying a recursive procedure for constructing a set of incoming singularities that ensure that only the directly-transmitted singularity reaches $D^+\MDT$. The notation of Appendix~\ref{s:parametrix-construction} will be used throughout.
	
	Our key constructions will be order-0 FIO $\Xi\supi_\pm,\Xi\supo_\pm\colon C^\infty(\RR\times\bdy Z)\to \mathcal D'(\mathbf Z)$ producing tails outside $\Theta$ for $(\pm)$-escapable bicharacteristics. Following~\sref{s:ml-construct}, the $\Xi_+\supio$-constructed tail for a singularity on a $(+)$-escapable bicharacteristic ensures this singularity escapes $\Theta$ at time $2T$, without generating any singularities in $h\MDT$'s microlocal forward domain of influence, $D^+\MDT$. The $\Xi\supio_-$-constructed tail generates a given singularity on a $(-)$-escapable bicharacteristic, again without causing any singularities to enter $D^+\MDT$. The $\Xi_\pm\supo$ are defined on outgoing boundary data while the $\Xi_\pm\supi$ are defined on incoming data, microlocally near the final, resp., initial covectors of $(\pm)$-escapable bicharacteristics.
	
	Let $\gamma\colon(t_-,t_+)\to T^*\mathbf Z$ be a $(\pm)$-escapable bicharacteristic. Denote by $\beta\supo$ the pullback to the boundary of its final point: $\beta\supo=(di_\Gamma)^*\gamma(t_\pm)$, where by abuse of notation we consider $\gamma(t_\pm)$ as a space-time covector, in $\To^*(\RR\times\mathbf Z)$. Define $\beta\supi=(di_\Gamma)^*\gamma(t_\mp)$ similarly. We now define $\Xi\supio_\pm$ microlocally near $\beta\supio$, starting with the incoming maps $\Xi\supi_\pm$.
\begin{itemize}
	\item
	\emph{If $t_\pm\in(0,2T)$:} We simply follow the bicharacteristic and apply $\Xi\supo_\pm$ at the other end. In the $(+)$ case define $\Xi\supi_+\eqml \Xi\supo_+\JBB$ near $\beta\supi$. In the $(-)$ case, define $\Xi_-\eqml \Xi_-  \JBB^- M$ near $\beta\supi$, where $\JBB^-=\upsilon\JBB\upsilon$ is like $\JBB$ but propagating backward in time. 
	
	\item
	\emph{If $\gamma$ escapes, $t_\pm\notin[0,2T]$:} This is the terminal case. In the $(+)$ case, there is nothing to do: define $\Xi_+\eqml 0$ near $\beta\supi$. For the $(-)$ case, define $\Xi_-\eqml \JCB^{-1}$ near $\beta\supi$ to obtain the necessary Cauchy data.
\end{itemize}
	
\noindent	We now turn to $\Xi\supo_\pm$, considering each case in the definition of $(\pm)$-escapability.
	
\begin{itemize}
	\item
	\emph{If $\gamma$ escapes:} This case never arises: $\Xi\supi_\pm$ is not defined in terms of $\Xi\supo_\pm$ for such $\gamma$.

	\item	
	\emph{If all outgoing bicharacteristics are $(\pm)$-escapable:} Recursively apply $\Xi\supi_\pm$ to the reflected and transmitted (if any) bicharacteristics, defining $\Xi\supo_\pm\eqml \Xi\supi_\pm M$ near $\beta\supo$.
	
	\item
	\emph{If one outgoing bicharacteristic is $(\pm)$-escapable, and the opposite incoming ray is $(\mp)$-escapable:} This is the core case. In the $(+)$ case, near $\beta\supo$ let
	\begin{equation}
		\Xi\supo_+\eqml
		\when{
			-\Xi\supi_-M\subR^{-1} M^{}\subT + \Xi\supi_+ (M^{}\subR-M^{}\subT M\subR^{-1} M^{}\subT), 	& \qquad\text{case (R),}\\
			-\Xi\supi_-M\subT^{-1} M^{}\subR + \Xi\supi_+ (M^{}\subT-M^{}\subR M\subT^{-1} M^{}\subR), 	& \qquad\text{case (T),}
		}
	\end{equation}
	according to whether the reflected (R) or transmitted (T) outgoing ray is $(+)$-escapable. The inverses are all microlocal. The $(-)$ case is slightly different: near $\beta\supo$,
	\begin{equation}
		\Xi\supo_-\eqml
		\when{
			\Xi\supi_-M\subR^{-1} + \Xi\supi_+ M^{}\subT  M\subR^{-1}, 	& \qquad\text{case (R),}\\
			\Xi\supi_-M\subT^{-1} + \Xi\supi_+ M^{}\subR  M\subT^{-1}, 	& \qquad\text{case (T).}
		}
	\end{equation}
	For case (R), the requirement in the definition that $c$ be discontinuous at $\beta\supio$ implies that $M\subR$'s principal symbol is nonzero there (cf.~\eqref{e:geometric-ps-refl-trans}), guaranteeing the existence of a parametrix $M\subR^{-1}$ near $\beta\supio$. For case (T), $M\subT$ always has positive principal symbol, regardless of $c$.
\end{itemize}

	While $\Xi\supio_\pm$ is defined recursively, by definition only finitely many recursions are needed to reach the non-recursive case where $\gamma$ escapes. Since all the cases are open conditions on $\beta$, operators $\Xi\supio_\pm$ are well-defined (assuming that in regions where both the second and third cases hold, we decide between them consistently). Furthermore, the $\Xi\supio_\pm$ are order-0 FIO, since they are microlocally sums of compositions of order-0 FIO associated with invertible canonical graphs.
	
	We now use $\Xi\supio_\pm$ to define a parametrix $A$.
	Given $\eta\in \mathcal S\subset \To^*\bTheta'$, consider the escaping bicharacteristics starting at $\eta$. Each is associated with a distinct sequence of reflections and transmissions $s=(s_1,\dotsc,s_k)\in \{R,T\}^k$ for some $k$, and a corresponding propagation operator
	\begin{equation}
		\mathcal P_s = \JBB M_{s_k}\dotsb\JBB M_{s_2}\JBB M_{s_1}\JCB.
	\end{equation}
Let $\mathfrak S$ be the set of escaping bicharacteristic sequences $s$, and define
	\begin{equation}
		A_\eta = I + \Xi\supo_+\sum_{s\in\mathfrak S} \mathcal P_s,
		\label{e:def-A-eta}
	\end{equation}
	Then define $A$ by patching together the $A_\eta$ with a microlocal partition of unity. As $\Xi\supio_\pm$\!, $\mathcal P_s$ are FIO of order 0, so is $A$.
	
	We now check that $A$ isolates $h\MDT$ and is therefore a microlocal right inverse for $I-\sigma^\star R\sigma^\star R$ by Lemma~\ref{l:isolate-mdt}. Let $h_0$ be microsupported in a sufficiently small neighborhood of $\eta\in\mathcal S$ and let $h_\infty = Ah_0$. Define the outgoing boundary parametrix
	\begin{equation}
		\mathfrak B = \JBS \sum_{k=0}^\infty (M\JBB)^k.
	\end{equation}
	With $\mathcal P_s$, $\mathfrak S$ as before, define $\mathfrak S^\perp$ to be the set of sequences $s$ for which no $s'\in\mathfrak S$ is a prefix. Then $\tilde Fh_\infty$ splits into three components:
	\begin{equation}
		\tilde Fh_\infty = \tilde F(h_\infty-h_0) + \mathfrak B M \sum_{s\in\mathfrak S}\mathcal P_sh_0 + \sum_{s\in\mathfrak S^\perp} \tilde F_s.
	\end{equation}
	For $t\in[T,2T]$, the last term is the wave field of $h\MDT$; accordingly, it suffices to prove that the sum of first two terms are smooth in $D^+\MDT$. Rewrite
	\begin{equation}
		\tilde F(h_\infty-h_0) + \mathfrak BM \sum_{s\in\mathfrak S}\mathcal P_sh_0
		=
		\sum_{s\in\mathfrak S} (\tilde F\Xi\supo_+ + \mathfrak BM) \mathcal P_s h_0.
	\end{equation}
	By construction, $\tilde F\Xi\supo_+ + \mathfrak BM$ is smoothing at the terminal end of $(+)$-escapable bicharacteristics, and in particular on $\WF(\mathcal P_s h_0)$ for each $s\in\mathfrak S$, as desired. Hence $\tilde R_{2T} h_0\eqml\tilde R_T h\MDT$. Applying Lemma~\ref{l:isolate-mdt}, we conclude $(I-\sigma^\star\tilde R\sigma^\star\tilde R)Ah_0\eqml h_0$. The same result holds for all $h_0\in\mathcal D'_{\mathcal S}$ by a microlocal partition of unity.
\end{proof}

\subsubsection{Uniqueness (\sref{s:ml-uniqueness})}

\begin{proof}[Proof of Proposition~\ref{p:ml-uniqueness}]
	Let $b_1$, $b_2$, $i$ be the principal symbols of $B_1$, $B_2$, and the identity. Letting $\sigma^\star$ and $\tilde r$ denote the operators on the space of principal symbols induced by multiplication with $\sigma^\star$ and composition with $\tilde R$, respectively, $(I-\sigma^\star\tilde r\sigma^\star\tilde r)(b_1-b_2)=0$. As in the proof of Theorem~\ref{t:ml-convergence}, it follows that $\tilde r(b_1-b_2)$ is supported in $T^*\bTheta'^\star$.
\end{proof}

\section{Comparison of the exact and microlocal analyses}					\label{s:compare}

Both the exact analysis of Section~\ref{s:exact} and the microlocal analysis of Section~\ref{s:microlocal} prove that scattering control isolates a certain portion of the wave field of $h_0$ at $t=T$, while effectively erasing the rest. Our two analyses, however, predict the isolation of two \emph{different} portions of the wave field. Surprising at first glance, this disparity provides further insight on scattering control, which we explore in this section.

While the arguments are quite general, we consider for simplicity two particular examples that illustrate the fundamental differences between dimensions $n=1$ and $n>1$. In the one-dimensional example, the microlocal and exact analyses align as $h\DT$ and $h\MDT$ are essentially equal; the result is unconditional convergence of the Neumann iteration, both exactly and microlocally. In higher dimensions, however, $h\DT$ and $h\MDT$ can be quite different, causing a loss of convergence in finite energy space.

\subsection{Convergence in $n=1$ dimension}

Let $\Omega=(\eps,\infty)$ and $\Theta=(0,\infty)$ for fixed $\eps>0$; let $\Theta',\Theta''$ be arbitrary. Let $c$ be piecewise smooth on $\RR$, and equal to 1 on $\Omega^\star$. 
In general, the distance of a point from $\bdy\Theta$ is the minimum distance of a singularity at that point from $\bdy\Theta$:
\begin{equation}
	d(x,\bdy\Theta) = \min_{\xi\in \To^*_x\RR\vphantom{\mathring T}}d(\xi,\bdy T^*\Theta).
	\label{e:1D-cotangent-distance}
\end{equation}
In one dimension, this means $d^*_{T^*\Theta}(\xi)=d^*_\Theta(x)$ if $\xi\in\To^*_x\RR$. Hence, $h\DT$ and $h\MDT$ are essentially equivalent, differing only in their respective usage of harmonic extensions and smooth cutoffs. We now discuss the microlocal and exact behaviors that arise in scattering control.

On the microlocal side,~\eqref{e:1D-cotangent-distance} implies every returning bicharacteristic is trivially $(+)$-escapable, as no glancing or totally reflected waves arise. Consequently, the constructive parametrix $A$ may be defined everywhere in $\To^*\bTheta'$, and hence by Theorem~\ref{t:ml-convergence} microlocal Neumann iteration always converges in principal symbol.

On the exact side, the \emph{exact} Neumann series converges to a finite energy solution $h_\infty$ of $(I-\pi^\star R\pi^\star R)h_\infty=h_0$, thanks again to microlocal analysis. To see why, first separate the initial data into rightward- and leftward-traveling waves (possible since $c=1$ there). The rightward-traveling portion has a directly transmitted component inside $\Theta$, which is its image under an elliptic graph FIO. Due to the ellipticity this directly transmitted wave carries a positive fraction of the initial energy, by \Garding's inequality and unique continuation (compare Stefanov and Uhlmann's work~\cite{SU-TATBrain}). Leftward-traveling waves, meanwhile, may be safely ignored, since $c$ is constant for $x<0$. The full proof requires some care, and we defer it to~\sref{s:compare-proofs}.

\begin{proposition}
	Let $\Omega$, $\Theta$, $c$ be as above, and $\eps<2T$. Then $\norm{\pi^\star R\pi^\star R}<1$ on $H^1(\Omega^\star)\oplus L^2(\Omega^\star)$; in particular $\sum_{k=0}^\infty (\pi^\star R)^{2k}h_0$ always converges.
	\label{p:1D-convergence}
\end{proposition}

\subsection{Convergence in $n>1$ dimensions}

\begin{figure}
	\centering
	\includegraphics{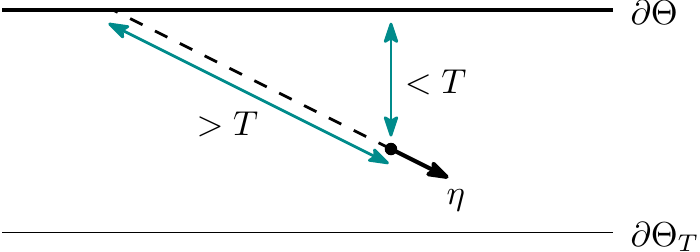}
	\caption{A singularity in $h\MDT$ but not $h\DT$. Its distance along the slanted bicharacteristic is greater than $T$, but its base point is less than distance $T$ from the boundary. Hence $\eta\in (T^*\Theta)_T$ but $\eta\notin T^*(\Theta_T)$.}
	\label{f:slant-singularity}
\end{figure}

Consider a halfspace $\Theta=\{x_n\geq 0\}$, and let $c(x)=1$. Any $\eta=(x',x_n,\xi',\xi_n)\in \To^*\bTheta$ with $x_n>T$ then belongs to $T^*(\bTheta_T)$. However, if $\xi'\neq 0$, then $d^*_{T^*\bTheta}(\eta)>x_n$ and $\eta\notin(T^*\bTheta)_T$ if $T$ is sufficiently close to $x_n$ (Figure~\ref{f:slant-singularity}). This discrepancy, which of course occurs for general $\Theta$, $c$ when $n>1$, implies that $h\DT$ is fundamentally smaller than $h\MDT$. Furthermore, it prevents the exact Neumann series from converging (in finite energy space) for any $h_0$ producing singularities in the gap $(T^*\bTheta)_T\setminus \clsr{T^*(\bTheta_T)}$, as we now show.

Suppose $\eta\in\WF(R_Th_0)\cap\big((T^*\bTheta)_T\setminus \clsr{T^*(\bTheta_T)}\big)$, and $\gamma$ is the bicharacteristic passing through $\eta$ at $t=T$. If there were a finite energy solution $h_\infty\in\mathbf C$ of the scattering control equation~\eqref{e:maf}, the proof of Theorem~\ref{t:basic-maf} implies (via unique continuation) that the wave field $v(t,x)=(F\clsr\pi Rh_\infty)(2T-t,x)$ is stationary harmonic at $t=T$ on $\bTheta_T^\star$, and in particular smooth at $\eta$. Propagation of singularities makes this impossible, since $\gamma([0,2T])$ lies completely inside $\Theta$. Hence no $h_\infty\in\mathbf C$ exists, and the Neumann series for $h_0$ must diverge, implying that $\norm{\pi^\star R\pi^\star R}=1$.

Using this argument, a divergent Neumann series may be constructed whenever $(T^*\bTheta)_T \neq T^*(\bTheta_T)$. Hence we expect $\norm{\pi^\star R\pi^\star R}=1$ in general for $n>1$ dimensions, in opposition to Proposition~\ref{p:1D-convergence} in 1D. It is worth noting that in numerical tests the Neumann iteration appears to follow its microlocally predicted behavior (isolation of $h\MDT$) more closely than its exact behavior (isolation of $h\DT$).

\subsection{Proof of convergence in one dimension}					\label{s:compare-proofs}

\begin{proof}[Proof of Proposition~\ref{p:1D-convergence}]
	This proof is inspired in large part by a proof of Stefanov and Uhlmann~\cite[Prop. 5.1]{SU-TATBrain}.
	Let $x(t)$ be the inverse function of the travel time $t=\int_0^x c(x')^{-1}\,\ud x'=d^*_\Theta(x)$; then $\Theta_t=(x(t),\infty)$. Choose $\delta>0$ small enough that $\abs{t_1-t_2}>\delta/2$ for any distinct $x(t_1),x(t_2)\in\singsupp c$. 
		
	In $(-\infty,\eps)$ take the factorization $\d_t^2-\Delta=(\d_t+i\d_x)(\d_t-i\d_x)$ associated with d'Alembert solutions $u(t,x)=f(x-t)+g(x+t)$. Identifying $h_0$ with $(f,g)\in H^1\times H^1$,
	\begin{equation}
		\dnorm{h_0}^2 = \int_0^\eps c^{-2}\dabs{g'-f'}^2 + \dabs{f'+g'}^2\,dx
				 = 2\big(\dnorm{f'}_{L^2}^2+\dnorm{g'}_{L^2}^2\big).
	\end{equation}
	The leftward-traveling component $g$ is trivially handled, since it is preserved by $R\pi^\star R$: indeed, if $f= 0$, then $\supp Rh_0\subset (-2T,-2T+\eps)$, and $\pi^\star R\pi^\star Rh_0=\pi^\star R^2h_0=0$. Hence we restrict attention to rightward-traveling initial data $h_0=(f,0)$.
	
	Intuitively, the energy of the direct transmission of $f$, that is, its image under the graph FIO components of $R$ involving only transmissions, should be bounded away from zero by \Garding's inequality since these components are elliptic.
	
	To start, assume $\supp h_0$ is contained in an interval $(a,b)$ of width $b-a\leq\delta$, so that no multiply-reflected rays enter the direct transmission region $I=(x(a+2T),x(b+2T))$. Furthermore, assume $c$ is constant on $I$, so that $Rh_0$ again divides into leftward- and rightward-travelling components $F,G$. 
	
	On $I$ we have $Rh_0 \eqml (R\DT^++R\DT^-) h_0$, where $R\DT^\pm$ are elliptic graph FIO (one for each family of bicharacteristics) associated with propagation along purely transmitted broken bicharacteristics; see Appendix~\ref{s:parametrix-construction}. Let $\pi_\pm=\fiv 2(I\pm iH)$ be the projections onto positive and negative frequencies (where $H$ is the Hilbert transform), and define the elliptic FIO $R\DT=R\DT^+\pi_+ + R\DT^-\pi_-$. Now on $I$ we have $F'\eqml \psi\d_x R\DT \d_x^{-1}f'$. Applying \Garding's inequality to the normal operator of $\d_x R\DT\d_x^{-1}$, with an appropriate spatial cutoff,
	\begin{nalign}
		\norm{h_0} = \sqrt 2\dnorm{f'}_{L^2} &\leq C_1\sqrt 2\dnorm{F'}_{L^2(I)} + \dnorm {Kf'}_{L^2}\\
				& = C_1\left(\En_{I}( Rh_0)\right)^{1/2} + \tnorm {\tilde K h_0}\\
				& \leq C_1\norm{\clsr\pi Rh_0} + \tnorm {\tilde K h_0},
		\label{e:1D-garding}
	\end{nalign}
	where $K,\tilde K$ are compact operators. In fact, $h_0=(f,0)\perp\ker \clsr\pi R$, so the compact error term $\norm{\tilde Kh_0}$ may be eliminated. To see this, by unique continuation $h_1=(f_1,g_1)\in\ker\clsr\pi R$ implies $Fh_1=0$ along $\RR\times\bdy\Omega$ and $[\eps,2T]\times\bdy\Theta$. Since $Fh_1=f_1(x-t)+g_1(x+t)$ outside $\Omega$, we conclude $f_1=0$. Conversely, $\clsr\pi(0,g_1)=0$ so that $\ker\clsr\pi R=\{(0,g_1)\}\perp h_0$.
	
	Hence on the subspace $g=0$, for some constant $C_2>0$,
	\begin{equation}
		\norm{\pi^\star R\pi^\star R} \leq \norm{\pi^\star R} \leq 1 - \frac 1{C_2}.
	\end{equation}
	and as $\pi^\star R\pi^\star R(f,g)=\pi^\star R\pi^\star R(f,0)$ this proves the result for all $h_0$.
	
	The same is true even if $c$ is not constant on $I$, since without affecting $\pi^\star R\pi^\star R$ we may modify $c$ so as to be constant on some deeper interval $(x(2T'),\infty)$, $T'>T+\epsilon/2$, and deduce an estimate analogous to~\eqref{e:1D-garding}, but at the later time $t=2T'$. By finite speed of propagation and conservation of energy, we can move the estimate back to $t=2T$ to establish~\eqref{e:1D-garding}.
	
	Finally, if $\eps>\delta$, it is possible that the direct transmission of a shallower part of $h_0$ may be cancelled by that of a deeper part of $h_0$, derailing the \Garding{} estimate. However, if this occurs the shallower and deeper parts of $h_0$ must be related by an elliptic FIO; therefore, the shallower part's energy is controlled by the deeper part's direct transmission.
	
	To make a simpler version of this idea rigorous, cover $(-2T,\eps)$ with intervals of width $\delta$:
	\begin{align}
		I_j&=((j-1)\delta,j\delta), 
		&
		j&=\lfloor -2T/\delta\rfloor,\dotsc,\lceil \eps/\delta\rceil=k.
	\end{align}
	Choose $f_j\in H^1\Loc$ with $f_j'=\mathbf 1_{I_j}f'$, where $\mathbf 1_{I_j}$ denotes the characteristic function. For each $j$, we have an estimate of the form~\eqref{e:1D-garding} with $h_0=(0,f_j)$. Let $E_j=\sqrt 2\norm{f_j'}_{L^2}$ be the energy of $f_j$.
	Now, let $j_0$ be the smallest $j$ for which $E_j\geq 2C_2^{-1}\sum_{i>j}E_i$; this is true of $j=k$ so such a $j_0$ always exists. By finite speed of propagation, the energy of $Rh_0$ in $I'' = (x(2T+(j_0-1)\delta), x(2T+j_0\delta))$ depends only on $f_i$ with $i\geq j_0$. But the direct transmission of $f_{j_0}$ contributes at least energy $2\sum_{i>j_0} E_i$, so by conservation of energy and \Garding's inequality
	\begin{equation}
		\big\lVert{f_{j_0}'}\big\rVert_{L^2} \lesssim \En_{I''}(Rh_0)+\tnorm{\tilde K h_0}.
	\end{equation}
	However, we may bound all of $f'$ in terms of $f_{j_0}'$. For, if $j>j_0$ certainly $\tnorm{f_j'}\lesssim \tnorm{f_{j_0}'}$; for $j<j_0$, this is also true as $E_j\not\geq 2C_2^{-1}E_{j_0}$. Hence
	\begin{equation}
		\dnorm{f'}_{L^2} < C_3\En_{I''}(Rh_0)+\tnorm{\tilde K h_0},
	\end{equation}
	with a constant $C_3=C_3(C_2,\eps,\delta,T)$. The remainder of the proof follows as before.
\end{proof}

\section{Connecting scattering control to the Marchenko equation}
\label{s:marchenko}

In this section, we illustrate the connection between Marchenko's integral equation and scattering control by first generalizing Rose's focusing algorithm~\cite{Rose02} to higher dimensions. This will show how one can eliminate multiple scattering in higher dimensions to eventually obtain a focused wave. We will start by summarizing Rose's approach in one space dimension to eliminate multiple scattering and obtain a focused wave. We will then explain the drawbacks to his approach, and provide our results that generalize his one-sided autofocusing results to higher dimensions. In addition, the one dimensional case will provide an accurate illustration of the microlocal solution $A$ constructed in Proposition~\ref{p:constructive-parametrix}. This will provide a clear distinction between the scattering control process and Rose's focusing algorithm where the advantages of scattering control are readily apparent. Lastly, we will connect our results with the 1D Marchenko equation used to solve the inverse scattering problem.
\subsection{Rose's one-sided autofocusing} \label{sec Rose autofocusing}
In \cite{Rose02}, Rose tries to focus an acoustic wave (working in $\RR_t \times \RR_x$) inside a medium occupying $\{ x > 0 \}.$ On the left side, $\{ x < 0 \}$, the wave speed is known, say $1$ for simplicity. Inside $x<0$, the total wave field $u$ may directly be decomposed into its incoming and outgoing components:
$$ u(x,t) = u_{\text{in}}(x,t) + u_{\text{out}}(x,t).$$
One is given the reflection response operator that we denote $\R(t)$ which relates the incoming and outgoing waves at the boundary $\{x=0\}$. By linearity, one has exactly
$$ u_{\text{out}}(x=0,t) = \int \R(t-t')\,u_{\text{in}}(0,t')\,\ud t'.$$
The goal of Rose is to determine a boundary control $u_{\text{in}}(x=0,t)$ such that the total wave field $u$ will be a distribution with support equal to $\{ x = x_f\}$ at time $t=0$ for some focusing point $x_f > 0$ one is interested in. Letting $t_f$ denote the focusing time, i.e.~$t_f = d_c(0,x_f)$, Rose uses the ansatz $u_{\text{in}}(x=0,t) = \delta(t+t_f) + \Omega\tail(t;t_f)$, and then finds an equation that $\Omega\tail$ must solve in order to obtain focusing.

Rose shows that $\Omega\tail$ must solve (see \cite[Equation (8)]{Rose02})
\beq{ \label{eq: Rose equation}
\Omega\tail(-t;t_f) + \R(\Omega\tail(-t;t_f)) = -\R(\delta(-t+t_f)) \text{ for }t<t_f,
}\eeq
where the action of $\R$ applied to a test function $\phi$ is
\beq{
\R\phi = \int_{-\infty}^{\infty}\R(t+t')\phi(t')\,\ud t'.
}\eeq
Equation (\ref{eq: Rose equation}) for $\Omega\tail(-t;t_f)$ is the Marchenko equation encountered in 1D potential scattering, which we will describe in more detail later. Also, if one denotes $r_0 = \delta(t-t_f)$ and $\tilde{K}\tail = \Omega\tail(-t;t_f)$,
then this equation reads
$$ \tilde{K}\tail + \R \tilde{K}\tail = -\R r_0 \text{ for } t< t_f,$$

Note that this approach relies heavily on the directional decomposition of a wave field into incoming and outgoing waves. In higher dimensions, such a decomposition may only be done microlocally, and as such, the reflection response operator $R_{\text{Rose}}$ would only be defined microlocally (see \cite{Stolk04} for a detailed account on doing this direction decomposition). The seismic literature has avoided this issue by ignoring the presence of evanescent and glancing waves, so a rigorous mathematical proof to obtain exact focusing in the presence of conormal singularities in higher dimensions has never been done. The whole point of using Cauchy data rather than boundary data is to avoid such microlocal considerations and obtain an iteration method in an exact sense.

Thus, based on the above equations, if we wanted to generalize this to higher dimensions in an exact sense using our Cauchy data setup, one may naively guess that the appropriate equation should be
\[ K\tail + \pi^{\star}R K\tail = -\pi^{\star}Rr_0\]
for $r_0, K\tail \in \mathbf{C}$, with $r_0$ having support in $\Theta$ and $K\tail$ having support outside $\Theta.$ Notice that no directional wave decomposition is necessary to write down this equation. This in fact turns out to be the correct equation, and we provide a rigorous analysis in the next section.

\subsection{Elimination of multiple scattering via a generalized Marchenko equation using Cauchy data}

We prove here a generalization to arbitrary dimension of Rose's equation (\ref{eq: Rose equation}) that allows one to eliminate multiple scattering of the pressure wave field. This is the key step that will allow one to focus a pressure field or velocity field at a given time. However, to avoid difficult microlocal issues with directional wave decompositions, we prove a theorem using Cauchy data rather than boundary data. Afterwards, we relate how this connects to Rose's algorithm for focusing discussed in the previous section as well as the classical Marchenko equation, which use boundary control rather than Cauchy data.

We now state the following general theorem about eliminating multiple scattering above a certain depth level $T$ (given in travel time coordinates) inside the medium, i.e.~within $\Theta^\star_T$.

\begin{theorem} \label{thm: focusing}
 Let $u$ be the solution to the wave equation with Cauchy data $r_{\infty} = r_0 + K\tail \in \mathbf{C}$, where $r_0$ has support in $\Theta$, and $K\tail$ has support outside $\Theta$. Let $T>0$.
\begin{enumerate}
\item[(i)](Necessity)
If $u(T)$ has support in $\Theta_T$, then necessarily $K\tail$ satisfies the following equation
\beq{ \label{eq: generalized Rose equation}
K\tail + \pi^\star RK\tail = -\pi^\star Rr_0
}\eeq

\item[(ii)](Partial converse) Suppose $K\tail$ satisfies
$$ K\tail + \pi^\star RK\tail = -\pi^\star Rr_0.$$
Then $\Pi^\star_T u(T) = 0$ and $u(T)\restrictto{\Theta_T} = R_Tr_0 \restrictto{\Theta_T}.$

\item[(iii)](Uniqueness of the tail) Any two tails may only differ by Cauchy data that is totally internally reflected, and does not penetrate $\Theta$ in time $2T$. That is, if
    $K\tail + \pi^\star RK\tail = 0,$
    then $K\tail =0$ in $\mathbf{C}$.
\item[(iv)](Almost Solvability) The set of $r_0 \in \mathbf H$ for which one has a convergent Neumann series solution for $K\tail$,
    $$ \mathfrak{Q} := \{ r_0 \in \mathbf H : (I+\pi^\star R)^{-1}r_0 \in \mathbf C \}$$
    is dense in $\mathbf H$.
\end{enumerate}
\end{theorem}
(Note that $\Pi_T^\star$ denotes the orthogonal projection from $H^1(\Theta_T^*)$ onto $H_0^1(\Theta_T^*)$.)

\begin{rem}\label{rem: control of multiple scattering}
The main content of this theorem is that once $r_0$ is given, then one has a formula to construct $K\tail$ that controls the multiple scattering inside $\Theta^\star_T$ at time $T$. The construction of $K\tail$ gives \emph{no information} on what happens inside $\Theta_T$ at time $T$ since $K\tail$ does not affect this region. What happens inside $\Theta_T$ is entirely determined by $r_0$. Thus, for the purposes of focusing, one needs to construct $r_0$ beforehand such that the associated pressure field restricted to $\Theta_T$ at time $T$ will have a singular support at a single point. In Wapenaar et al.~\cite{Wap}, the authors assume they have an approximate velocity profile to construct an approximation to the direct transmission (denoted $\mathcal{T}^{\text{inv}}_d$ in equation (16) there), which is analogous to the $r_0$ we have here. They then construct a tail (denoted by $M$) analogous to our $K\tail$ to control the multiple scattering.
\end{rem}

\begin{rem}
Notice that this theorem never mentions a focusing point but rather an inside region $\Theta_T$. This is because in order to make the theorem more general, we did not specify any support conditions for $r_0$. Typically however, one sends an incident pulse $r_0$ that is supported close to but outside $\Omega$, which is meant to be the direct transmission. Then the domain of influence of $r_0$ inside $\Theta_T$ at time $T$ is only a small region in a neighborhood of $\partial \Theta_T$ containing the desired point of focus (see Figure \ref{f:adt}). We relate the above theorem to focusing via a corollary at the end of this section.
\end{rem}

\begin{rem}
As mentioned in \cite{Rose02} as well, this result only describes how to control multiple scattering of the pressure field, but says nothing about the velocity field at time $T$; hence energy is not controlled and the wave field may still have a large kinetic energy even at time $T$. Also, after the time $t=T$, the Cauchy data inside $\Theta_T^\star$ generate waves that may and generally do enter the inner layer $\Theta_T$ even before time $t=2T$. The main advantage of scattering control is that it controls both the pressure \textit{and} velocity field so that for $T \leq t \leq 2T$, the wave generated by the time $T$ Cauchy data inside $\Theta^*_T$ will not penetrate the domain of influence of the direct transmission $\bar{\pi}_TR_Tr_{0}$.
\end{rem}

\begin{proof}
We start with (i).
Suppose we found a wave field $u$ such that $u(T)$ has support in $\Theta_T$, and Cauchy data $r_{\infty} = r_0 + K\tail$ as in the statement of the theorem. Let us denote
$$ w(t) = u(T+t) + u(T-t).$$
Observe that
$$ w(0) =0 \text{ outside }\Theta_T, \qquad \text{and} \qquad
w_t(0) = 0.
$$
By finite propagation speed, one also has $w(t,x) = 0$ when $d(x,\Theta_T) > t$. Notice that all points in $\Theta^\star$ are at least distance $T$ away from $\Theta_T$ so one has
$$ \pi^\star \mathbf{w}(T) = 0$$
This precisely means that
$$ u(2T) = - u(0) \text{ on }\Theta^\star$$
and
$$ - u_t(2T) = -u_t(0) \text{ on } \Theta^\star.$$

Written in operator form, this amounts to
$$ \pi^\star\nu \circ R_{2T}r_{\infty} = -\pi^\star r_{\infty},$$
where we recall that $R_s$ does not just propagate $s$ units of time, but also give the Cauchy data at time $t=s$.
Plugging in $r_{\infty} = r_0 + K\tail$ above gives
\begin{align}\label{eq: Marchenko for the tail}
 &\pi^\star R(r_0 + K\tail) = -\pi^\star (r_0 + K\tail)
\nonumber \\
&\Leftrightarrow \pi^\star R r_0 + \pi^\star RK\tail = -\pi^\star r_0 -\pi^\star K\tail = -K\tail
\nonumber \\
 &\Leftrightarrow K\tail + \pi^\star RK\tail = -\pi^\star Rr_0.
\end{align}

\paragraph{Proof of (ii)}
First, if one adds $r_0$ to both side of (\ref{eq: Marchenko for the tail}), and brings $-\pi^*R r_0$ to the the left hand side, one obtains
\beq{ \label{eq: Marchenko for the full Cauchy}
(I + \pi^*R)r_{\infty} = r_0.
}\eeq
Again denote $u(t) = (Fr_{\infty})(t),$ and let $w(t)$ be a superposition of $u(t)$ and its time reversal; that is
$$ w(t) = (Fr_{\infty})(t) + (Fr_{\infty})(2T-t).$$
Then using (\ref{eq: Marchenko for the full Cauchy}) and recalling that $r_0$ vanishes outside of $\Theta$, we have
$$ \mathbf{w}(0) = r_{\infty} + R r_{\infty}  \text{ is harmonic in }\Theta^\star.$$
Similarly,
$$ \mathbf{w}(2T) = R_{2T}r_{\infty} + \nu \circ r_{\infty} = \nu \circ ( R r_{\infty} +  r_{\infty}) \text{ is harmonic in }\Theta^\star.$$
Note that $w_t(2T) = 0 = w_t(0)$ in $\Theta^\star$.
Since $w$ also solves that wave equation, then $\p^2_tw$ vanishes wherever $w$ is harmonic. By translation invariance of the wave operator, $\p_t w$ (the mollification argument to make this precise is exactly as in the proof of (\ref{e:basic-maf-behavior})) also solves the wave equation while also having Cauchy data at times $t=0$ and $t=2T$ vanishing in $\Theta^*$.
By Lemma 3, $\p_t\mathbf{ w}(T) = 0$ inside $\Theta^\star_T$. Looking at just the first component of $\mathbf{w}(T)$ this says exactly that $u(T)$ is harmonic in $\Theta^\star_T$, which is equivalent to $\Pi_T^\star u(T)=0$. The second statement in the theorem follows from finite propagation speed, as $K\tail$ is supported in $\Theta^\star$.

\paragraph{Proof of (iii)}
Suppose that $K\tail + \pi^\star RK\tail = 0$. Since $\pi^\star$ is a projection and $R$ is unitary, one has
$$ \norm{\pi^\star RK\tail} \leq \norm{K\tail}.$$
However, since $K\tail = - \pi^\star RK\tail$, then the inequality above must in fact be an equality and so $\norm{\pi^\star RK\tail} = \norm{K\tail}$. Since $R$ is unitary, one has
$$\norm{K\tail}^2 = \norm{RK\tail}^2 = \norm{\pi^*RK\tail}^2 + \norm{\bar{\pi}RK\tail}^2 = \norm{K\tail}^2 + \norm{\bar{\pi}RK\tail}^2.$$
Thus, $\bar{\pi}RK\tail = 0$ and so
$ K\tail = -\pi^\star RK\tail = -RK\tail$, implying that $K\tail \in \mathbf{G}.$

\paragraph{Proof of (iv)}
Denote $K_l = \sum_{j=0}^l (-\pi^\star R)^j(-\pi^\star R r_0).$
The proof follows almost verbatim as the proof showing the density of the set $\mathcal Q$ defined in (\ref{e:Q-definition}).
\end{proof}

In order to make Remark \ref{rem: control of multiple scattering} more transparent on how this theorem relates to focusing, we add the following corollary. First, we conjecture that following the methods of boundary control in \cite{DKO}, one may extract certain travel times between points on the boundary to points in the interior and use that to create an $r_0$ supported outside $\Omega$, such that at a time $T$, the first component of $R_T(r_0)\restrictto{\Omega_T}$ has singular support equal to a single point. Thus we believe that it will be possible to satisfy the assumption in the following corollary using boundary control methods.

\begin{cor}
Suppose $r_0 \in \mathbf{C}$, a time $t=T$, and $\Theta\supset \Omega$ are such that $\supp(r_0) \subset \Theta$ and the singular support of $F(r_0)(T)\restrictto{\Theta_T}$ is nontrivial, contained inside $B_{\epsilon}(x_f)$ for some small $\epsilon>0$. Then if $K\tail$ solves (\ref{eq: Marchenko for the tail}), then the singular support of $u(T)$ is nontrivial and contained in $B_{\epsilon}(x_f)$.
\end{cor}

The corollary is stated using the energy spaces employed throughout the paper. However, we believe it can be refined to encompass general distributions and in particular a point singular support so that one has a focusing wave in the usual sense.
\begin{rem}
We emphasize again that despite the attractiveness of the corollary, it only gives focusing of the pressure field and says nothing about the velocity field. Thus, once one goes past time $t=T$, one has lost all control and one has no information on the wave field at such times, which is usually quite complex since $K\tail$ needs to be quite complicated in order to control the multiple scattering that allows focusing. Thus, the scattering control procedure is much more useful in this regard.
\end{rem}

We close this section with an analogous theorem to Theorem \ref{thm: focusing} which controls the multiple scattering of the velocity field instead. The proof is almost identical excepting sign changes so we omit it.
\begin{theorem} (Multiple scattering control of velocity field)\label{thm: focusing velocity}
 Let $u$ be the solution to the wave equation with Cauchy data $r_{\infty} = r_0 + K\tail \in \mathbf{C}$, where $r_0$ has support in $\Theta$, and $K\tail$ has support outside $\Theta$. Let $T>0$.
\begin{enumerate}
\item[(i)](Necessity)
If $u_t(T)$ has support in $\Theta_T$, then necessarily $K\tail$ satisfies the following equation
\beq{
K\tail - \pi^\star RK\tail = -\pi^\star Rr_0
\nonumber
}\eeq

\item[(ii)](Partial converse) Suppose $K\tail$ satisfies
$$ K\tail - \pi^\star RK\tail = -\pi^\star Rr_0.$$
Then $u_t(T) \restrictto{\Theta^\star_T}= 0$ and $\mathbf{u}(T) \restrictto{\Theta_T} = R_Tr_0 \restrictto{\Theta_T}.$

\item[(iii)](Uniqueness of the tail) Any two tails may only differ by Cauchy data that is totally internally reflected, and does not penetrate $\Theta$ in time $2T$. That is, if
    $K\tail - \pi^\star RK\tail = 0,$
    then $K\tail =0$ in $\mathbf{C}$.
\item[(iv)](Almost Solvability) The set of $r_0 \in \mathbf H$ for which one has a convergent Neumann series solution for $K\tail$,
    $$ \mathfrak{Q} := \{ r_0 \in \mathbf H : (I-\pi^\star R)^{-1}r_0 \in \mathbf C \}$$
    is dense in $\mathbf H$.
\end{enumerate}
\end{theorem}

\begin{rem}
We note that an almost identical proof used to recover kinetic energy of the almost direct transmission in Proposition \ref{p:basic-energy} and \ref{p:limit-energy} may be used here to recover this energy from $K\tail$ instead.
\end{rem}

At this point, one might be led to believe that information may be lost or gained by using our Cauchy data setup versus the boundary setup that is done in Rose. This is actually not the case, and we show in the next section that in one dimension, where one does not worry about glancing rays, both formulations are completely
equivalent.

\subsection{Equivalence between Cauchy and boundary formulations in one dimension}

For simplicity, we assume here that $\Omega$ occupies $x>0$ and $\Theta$ is exactly the half-space $\{ x>-\epsilon\}$ for some $\epsilon > 0$. Without loss of generality, we assume that the wave speed is constantly equal to $1$ outside $\Omega$, i.e.~$c\restrictto{\Omega^\star} =1$. Then any wave field inside $\Omega^\star$ is of the form
\beq{
u \restrictto{\Omega^\star} = f(t-x) + g(x+t)
}\eeq
We assume that $\supp(f(s))\subset \{ -T < s < T+\epsilon \}$ ($T$ is the focusing time; i.e.~we are focusing at a point $x_T$ which is distance $T$ away from $0$ using the metric determined by $c$) and that the left going wave $g$ is activated only after the right going wave $f$ hits the boundary $\{ x =0 \}.$ Precisely, this means that
$$ \supp(g(s)) \subset \{ s > -T \}.$$
As described in the last section, one has
\beq{\label{eq: CtoB g is related to f}
 g(t) = \R \ast f = \int_{-\infty}^{\infty} \R(t-t')f(t')dt'.
 }\eeq
This is well-defined in an exact sense precisely since there are no glancing rays in 1 space dimension. See for example \cite{AR02} for details.

To avoid dealing with harmonic extensions, as they do not add anything essential, we will assume that $R$ applied to any of our Cauchy data has $0$ trace on $\partial \Theta^*= \{ x=0\}$. This merely ensures that
$$ \pi^*R = \mathbf{1}_{\Theta^*}R = \mathbf{1}_{\{{x<-\epsilon}\}}R$$
when applied to such Cauchy data.

Next, observe that since $g(s) = 0 $ when $ s \leq 0$ and using the support condition of $f$, our Cauchy data (initially given at $t=-T$ as opposed to $t=0$) and its time-$2T$ propagation is
\begin{nalign}
\tilde{\f}(x):=\u(-T) &= \begin{pmatrix}f(-T-x)\\f'(-T-x)\end{pmatrix},\\
\pi^*R(\u(-T)) = \pi^*R\tilde{\f} &= \mathbf{1}_{\{ x< -\epsilon \}}\begin{pmatrix}g(T+x)\\-g'(T+x)\end{pmatrix}.
\end{nalign}
Then by $(\ref{eq: CtoB g is related to f})$ we have
\beq{
(\pi^*R\tilde{\f})(t-T) = \mathbf{1}_{\{{t<T-\epsilon}\}}\nu\g(t) = \mathbf{1}_{\{{t<T-\epsilon}\}}\nu(\R\star \f)(t),
}\eeq
where we get an equation for $g'(t)$ by differentiating (\ref{eq: CtoB g is related to f}), and we use the notation $\f, \g$ to represent a column vector of $f,g$ and their derivative. Let us denote $\JCB$ as the \emph{Cauchy-to-boundary} map, which maps Cauchy data at time $t=-T$ to boundary data on $\{ x =0 \}$. In this simple setting, it is well-defined as a map $\JCB: D'(\RR_x) \to D'(\RR_t)$ explicitly defined on smooth functions as
$$ \JCB v(t) = v(t-T)$$
with an obvious extension to elements in $\mathbf{C}$.
Since $\tilde{\f} = \JCB^{-1}\f(-\cdot)$
and $\R \star \phi(-\cdot) = \R \phi(\cdot)$,
we have a nice relationship between $R_{2T}$ and $\R$ given by
\beq{
\JCB R_{2T}\JCB^{-1}(\f(-\cdot)) = \R(\f(-\cdot)) \qquad \text{for }t<T.
}\eeq
\begin{prop}(Equivalence of Rose and Cauchy-Marchenko in one dimension)
 Let $f(t) =K\tail(t)+r_0(t)$ denote the incoming boundary data, and $\tilde{\f}(x) = \JCB^{-1}(\textbf{K}\tail(t)+\textbf{r}_0(t))
 := \tilde{K}\tail(x) + \tilde{r}_0(x)$ be the corresponding Cauchy data at time $-T$ with all the assumptions described earlier. Then, $\tilde{K}\tail$ satisfies the Cauchy-Marchenko equation with $\tilde{r}_0$ iff $K\tail$ satisfies the Rose equation with $r_0$; that is,
$$ \tilde{K}\tail(x) + \pi^*R\tilde{K}\tail(x) = -\pi^*R\tilde{r}_0$$
$$ \Leftrightarrow $$
$$
K\tail(-t) + \R(K\tail(-\cdot))
 = -\R(r_0(-\cdot)) \text{ for } t<T-\epsilon
 $$
\end{prop}
\begin{proof}
 Suppose we start with the Cauchy-Marchenko equation in the form (\ref{eq: Marchenko for the full Cauchy}) (translating everything by time $T$ and using the notation of boldface letters to represent a vector consisting of the funcion and its time derivative):
\begin{align}
 &\u(-T) + \pi^\star R(\u(-T)) = \bar{\pi}\u(-T)
\nonumber \\
\label{eq: CtoB first piece}
 &\Leftrightarrow \tilde{\f}(x) + \pi^\star R\tilde{\f}(x) = \bar{\pi}\tilde{\f}(x)
 \\
 &\Leftrightarrow
 \JCB \tilde{\f} + \JCB \pi^\star R\tilde{\f} = \JCB  \bar{\pi}\tilde{\f}
 \label{eq: CtoB 2nd piece} \\
 &\Leftrightarrow
 \f(-t) + \mathbf{1}_{\{t<T-\epsilon \}}\nu(\R \star \f)(t)
 = \mathbf{r}_0(-t)
\nonumber
\end{align}

This is essentially the right equation for Rose, but we rewrite it in the more familiar form:
\begin{align}
&\f(-t) + \mathbf{1}_{\{t<T-\epsilon \}}\nu(\R \star \f)(t)
 = \mathbf{r}_0(-t)
\nonumber \\
 &\Leftrightarrow
 \mathbf{K}\tail(-t) + \nu(\R \star \mathbf{K}\tail)(t)
 = -\nu(\R \star \mathbf{r}_0)(t) \text{ for } t<T-\epsilon
 \nonumber \\
 &\Leftrightarrow
 \mathbf{K}\tail(-t) + \nu \R(\mathbf{K}\tail(-\cdot))
 = -\nu \R(\mathbf{r}_0(-\cdot)) \text{ for } t<T-\epsilon,
 \nonumber
 \\
 &\Leftrightarrow
 \when{
 K\tail(-t) + \R(K\tail(-\cdot))
 = -\R(r_0(-\cdot))
\nonumber \\
 \smd t\left[
 K\tail(-t) + \R(K\tail(-\cdot))\right]
 = -\smd t\R(r_0(-\cdot))
 }
  \text{ for } t<T-\epsilon
 \nonumber \\
 &\Leftrightarrow
K\tail(-t) + \R(K\tail(-\cdot))
 = -\R(r_0(-\cdot)) \text{ for } t<T-\epsilon
 \nonumber
 \end{align}
 where the first equality is obtained be subtracted $\mathbf{r}_0(-t)$ from both sides of the first equation and writing $f= r_0+K\tail$. \end{proof}

 \begin{rem}
 The above result helps explain the truncation that Rose does in \cite{Rose02} to obtain his autofocusing algorithm.
 The Corollary essentially shows that $K\tail(t)$ must satisfy
 $$ \mathbf{1}_{\{t<T-\epsilon\}}K\tail(-t) + \mathbf{1}_{\{t<T-\epsilon\}}\R(K\tail(-\cdot)) = - \mathbf{1}_{\{t<T-\epsilon\}}\R(r_0(-\cdot))$$
One naturally assumes that the tail come after the direct transmission $r_0$, which means $K\tail(t)$ is supported in $t > -T+\epsilon$ and hence
 $ \mathbf{1}_{\{t<T-\epsilon\}}K\tail(-t) = K\tail(-t).$
 Thus, the Neumann series becomes
  \begin{align*}
  K\tail(-t) &= -\mathbf{1}_{\{t<T-\epsilon\}}\R(r_0(-\cdot))
 + (\mathbf{1}_{\{t<T-\epsilon\}}\R)^2(r_0(-\cdot))
\\
 &\qquad - (\mathbf{1}_{\{t<T-\epsilon\}}\R)^3(r_0(-\cdot))
 + \dots
\end{align*}
 and we may clearly see the truncation happening at each step of the algorithm. The truncation is essential since we just proved the equivalence of Rose's algorithm to our Cauchy scheme, and we already proved that our equation (\ref{eq: Marchenko for the tail}) is necessary and sufficient to control multiple scattering. The proof shows that the truncation essentially comes from (\ref{eq: Marchenko for the tail}) only holding within a certain region in space (i.e.~$\Theta^\star$ in that theorem) that was determined by finite speed of propagation and unique continuation. In one dimension and after using the Cauchy-to-Boundary map, this spatial region corresponds to the time-truncation appearing in Rose.
 \end{rem}
We will describe in the following sections the connection between the equations of the previous theorems, the Marchenko equation, and scattering control.

 \subsection{Connection to the Marchenko equation}
Burridge~\cite{Bur80} considers the 1-dimensional inverse scattering problem for the plasma wave operator $\Box_q = \Box + q(x)$ where $q = 0$ in $x<0$. (recall that in 1 dimension, the acoustic wave equation may be put into this form by a change of variables as in \cite{Bur80}). Since it is not relevant for this part, we will avoid describing the function spaces where all of our distributions here belong. One is interested in solutions to $ \Box_q u = 0$ with certain boundary conditions at $x=0$ that allow for only left-going solutions inside $x<0$ (see \cite[Section 3]{Bur80} for details).
 It is shown in \cite{Bur80} that there is a special Green's function  solution of the form
 $G = \delta(t-x) + K(x,t)$ such that $\supp(K) \subset \{ |t| \leq x, \ x \geq 0 \}$ and one may recover $q$ from knowing $K$.

 The given data are the reflected waves due to a right-going incidence wave in the region $x <0$. Analytically, there is a causal Green's function:
 $$G_1(x,t) = \delta(t-x) + K_1(x,t)$$
 with $\supp(K_1) \subset \{ t \geq |x|, \ t >0\}$. One is given the data $\M(t) = K_1(x=0,t)$ (interpreted as a generalized trace), and the goal is to recover $K$ from $\R$. Then it is shown in \cite[Section 3]{Bur80} that for each fixed $x$, $K$ must satisfy the following integral equation known as the \emph{Marchenko} equation:
 \beq{
  K(x,t) + \int_{-x}^xK(x,\tau) \M(t+ \tau) d\tau =-\M(t+x)  \qquad \text{ for } t<x.
 }\eeq

To relate this to (\ref{eq: generalized Rose equation}), change variables to travel time coordinates
$$ z = \int_0^x c(x')^{-1}\,\ud x'.$$
Comparing with (\ref{eq: Rose equation}), we see that $t_f = z(x_f)$ and $K(z,t) = \Omega\tail(-t;z)$ solves the Marchenko equation above with $\R$ as the given data in place of $\M$. The connection to (\ref{eq: generalized Rose equation}) is now readily apparent from the previous subsections.

\subsection{Connection to scattering control}
Notice that the proof of multiple scattering control in Theorem \ref{thm: focusing} and its corollary essentially utilizes the operators $I+ \pi^\star R$ and $I-\pi^\star R$ to control scattering from the pressure field and the velocity field respectively.
Our scattering control series is a middle ground that allows one to control scattering in both the pressure field and the velocity field such that after time $t=2T$, the exterior data coming from the direct transmission is distinguished. Indeed, the scattering control operator is precisely
$$ I - \pi^*R\pi^*R = (I-\pi^*R)(I+\pi^*R),$$
whose Neumann series solutions involve exactly the even terms in the Neumann series of $I-\pi^\star R$.
Figure~\ref{f:rose-and-sc} depicts the differences between Rose's autofocusing and scattering control in a simple one-dimensional example.

\begin{figure}
\label{f: 1d comparisons}
\centering
 \subfloat[A pulse with Rose's tail]{
          \includegraphics[page=3]{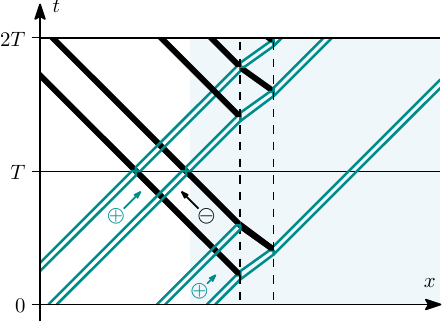}
\label{f: Roses tail}
}
\subfloat[A pulse with scattering control tail]{
\includegraphics[page=2]{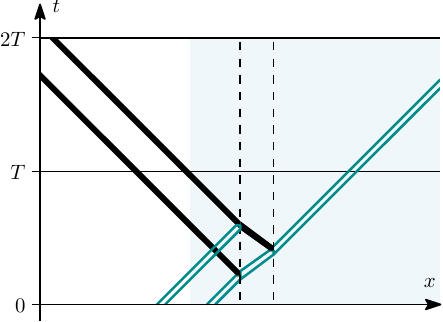}
\label{f: pulse with ma tail}
}
\caption{
 These figures correspond to the incident pulse in Figure \ref{f:mr-demo-original}. In Rose's setup, the tail has extra waves to ensure the pressure field is quiescent exactly at $t=T$ except for the direct transmission.
In (a), the tail (constructed by the formula in (\ref{eq: generalized Rose equation})) consists of three (positive amplitude) waves being sent in after the (positive amplitude) incident pulse. The first wave cancels a returning wave which would create further scattering between the interfaces. The other two waves in the tail cancel the backscattered (negative amplitude) waves at $t=T$, and only there. Thus, at $t=T$, the singular support of the pressure field is precisely one point determined by the direct transmission. Part (b) shows the tail constructed using the scattering control algorithm. For scattering control, we only care about the returning bicharacteristics, so the tail consists of only one wave to eliminate the one returning wave. Thus, for $t\in [T,2T]$ the total wave field only consists of the direct transmission and two waves that will never go deeper into the medium.}
\label{f:rose-and-sc}
\end{figure}

\appendix

\section{Wave equation parametrix with reflection and transmission}		\label{s:parametrix-construction}

We briefly review how a parametrix for the acoustic wave equation initial value problem with piecewise smooth wave speed may be constructed in terms of reflections and transmissions, neglecting glancing rays. This is now-classical FIO theory, drawing from the work of many authors, including Chazarain~\cite{Chazarain}, Hansen~\cite{Hansen}, and Taylor~\cite{Taylor75}. As nothing novel is developed here, we do not include proofs; our goal is simply to provide a bookkeeping system for use in the paper.

Recalling~\sref{s:ml-maf}, consider $c(x)$ piecewise smooth with singular support contained in disjoint closed smooth hypersurfaces $\Gamma_i$, with $\Gamma=\bigcup\Gamma_i$. The interfaces separate $\RR^n\setminus\Gamma$ into disjoint components $\Omega_j$. In order to distinguish the sides of each hypersurface $\Gamma_i$, consider an \emph{exploded space} $Z$ in which the connected components of $\RR^n\setminus\Gamma$ are separate. It may be defined in terms of its closure, as a disjoint union 
\begin{align*}
	\clsr Z &= \bigsqcup_{\smash[b]j} \clsr\Omega_j,
	&
	Z &= \bigcup_{\smash[b]j} \Omega_j\subset\clsr Z.
\end{align*}
In this way, $\bdy Z$ contains two copies of each $\Gamma_i$, one for each adjoining $\Omega_j$. 

Before proceeding further, we perform a standard microlocal splitting in order to separate forward- and backward-moving singularities.
Recall that $\d_t^2-c^2\Delta$ factors microlocally into half-wave operators $(\d_t+iQ)(\d_t-iQ)$.
The full solution operator $F$ is then equivalent microlocally to a sum of solution operators $F^\pm$ corresponding to $\d_t\pm iQ$, with initial data related by a microlocally invertible matrix \PsiDO{} $P$:
\begin{align}
	F(f_0,f_1) &\eqml F^+g_++F^-g_-,
	&
	\begin{bmatrix} g_+\\g_-\end{bmatrix}\eqml P\begin{bmatrix} f_0\\f_1\end{bmatrix}.
\end{align}
The Cauchy data $(g_+,g_-)$ may be interpreted as a single distribution $g$ on a doubled space $\mathbf Z=Z_+\sqcup Z_-$ containing two copies of $Z$.

We now describe a parametrix $\tilde R$ for $R=\nu\circ R_{2T}$ as a sum of graph FIO on $\mathbf Z$ built from sequences of reflections and transmissions, along with operators propagating data from one boundary to another, or propagating the initial data to boundary data. The key feature of the propagators is that waves reaching the boundary of a subdomain $\Omega_j$ simply leave $\Omega_j$ rather than reflecting. To handle reflections and refractions, we record the outgoing boundary data left by waves escaping $\Omega_j$ and convert them to appropriate incoming boundary data on each side of the interface, which generate reflected and refracted waves.

\paragraph{Cauchy Propagators: $\JCS$, $\JCSp$, $\JCB$}

We first develop a reflectionless solution operator $\JCS$ for the Cauchy problem on $\mathbf Z$. To begin, extend each restriction $c_j= \restr{c}_{\Omega_j}$ to a smooth function on $\RR^n$. Let $E^\pm_j$ be the half-wave Lax parametrix associated to $\d_t\pm iQ$, $Q=(-c_j^2\Delta)^{\smash{1/2}}$. Each $\eta\in \To^*\Omega_{\pm,j}$ is associated with a unique $c_j$-bicharacteristic $\gamma_\eta(t)$ in $\To^*\RR^n$ passing through $\eta$ at $t=0$, which may escape and possibly re-enter $\Omega_{\pm,j}$ as $t\to\pm\infty$.

To prevent re-entry of wavefronts, we introduce a pseudodifferential cutoff $\varphi(t,\xi)$, omitting some details for brevity. Let $t_{\mathrm e\pm}$, $t_{\mathrm r\pm}$ denote the first positive and negative escape and re-entry times; let $\varphi(t,\gamma_\eta(t))$ be identically one on $[t_{\mathrm e-},t_{\mathrm e+}]$ and supported in $(t_{\mathrm r-}, t_{\mathrm r+})$. Modify $\varphi$ on a small neighborhood of $\RR\times\To^*\bdy\Omega_{\pm,j}$ (the glancing rays) to ensure it is smooth. Finally, let $\JCS$ be the restriction of $\varphi(t,D_x)\circ E^\pm_j$ to $\RR\times\Omega_{\pm,j}$; this is the desired reflectionless propagator.

We also require a variant $\JCSp$ of $\JCS$ in which waves travel only forward in time. For this replace $\varphi$ with some $\varphi^+$ supported in $(t_{\mathrm e-}, t_{\mathrm r+})$ and equal to 1 on $[0,t_{\mathrm e+}]$. Restricting $\JCSp$ to the boundary, we obtain the \emph{Cauchy-to-boundary} map $\JCB=\restr{\JCSp}^{\phantom+}_{\RR\times\bdy\mathbf Z}$.

It can be shown (cf.~\cite{Chazarain}) that $\JCS,\JCSp\in I^{-1/4}(\mathbf Z\shortrightarrow \RR\times \mathbf Z)$, and $\JCB\in I^0(\mathbf Z\shortrightarrow\RR\times\bdy\mathbf Z)$. As desired, $\JCS$ and $\JCSp$ are parametrices: $(\d_t\pm iQ)\JCS h,(\d_t\pm iQ)\JCSp h\eqml 0$ for $\WF(h)$ lying in a set $\mathcal V\subset\To^*\mathbf Z$ whose bicharacteristics are sufficiently far from glancing. By a direct argument with oscillatory integral representations, it can also be shown that $\JCB$ is elliptic at covectors in $\mathcal V$ whose bicharacteristics intersect $\bdy\mathbf Z$. The near-glancing covector set $\mathcal W$ of~\sref{s:microlocal} is then $\To^*\mathbf Z\setminus\mathcal V$.

\paragraph{Boundary Propagators}

Outgoing solutions from boundary data $f\in\mathcal D'(\RR\times\mathbf Z)$ may be obtained by microlocally converting boundary data to Cauchy data, then applying $\JCS$. The boundary-to-Cauchy conversion can be achieved by applying a microlocal inverse of $\JCB$, conjugated by the time-reflecting map $S_s\colon t\mapsto s-t$ for an appropriate $s$. More precisely, near any covector $\beta=(t,x';\tau,\xi')\in\bdy\Omega_{\pm,j}$ in the hyperbolic region $\abs\tau>c_j\abs{\xi'}$ there exists a unique bicharacteristic $\gamma$ passing through\footnote{That is, $(di)^*\gamma(t)=\beta$, where $i\colon \bdy \mathbf Z\hookrightarrow \clsr {\mathbf Z}$.} $\beta$ and lying inside $\bOmega_{\pm,j}$ in some time interval $[s,t)$, $s<t$. Then $\JBS$ may be defined as $S_s\JCS \JCB^{-1}S_s$ microlocally near $\beta$.

On the elliptic region $\abs\tau<c_j\abs{\xi'}$ define $\JBS$ as a parametrix for the elliptic boundary value problem; see e.g.~\cite[\textsection4.8]{SU-TATBrain}.
Applying a microlocal partition of unity, we obtain a global definition of $\JBS$ away from a neighborhood of the glancing region $\abs\tau=c_j\abs{\xi'}$. It can be proven that $\JBS\in I^{-1/4}(\RR\times\bdy\mathbf Z\shortrightarrow \RR\times \mathbf Z)$. Its restriction to the boundary $r_\bdy\circ\JBS$ consists of a pseudodifferential operator equal to the identity on $\mathcal W$ and an elliptic graph FIO $\JBB\in I^0(\RR\times\bdy\mathbf Z\shortrightarrow\RR\times\bdy\mathbf Z)$ describing waves traveling from one boundary to another.

\paragraph{Reflection and Transmission}

It is well known that transmitted and reflected waves arise from requiring a weak solution to be $C^1$ at interfaces. Given incoming boundary data $f\in\mathcal E'(\RR\times\bdy\mathbf Z)$ (an image of $\JCB$ or $\JBB$) microsupported near $\beta$, we seek data $f\subR,\,f\subT$ satisfying the $C^1$ constraints
\begin{nalign}
	f + f\subR &\eqml \iota f\subT,\\
	\d_\nu (\upsilon\JBS\upsilon f + \JBS f\subR)\big|_{\RR\times\bdy\mathbf Z} &\eqml \iota \d_\nu\JBS f\subT\big|_{\RR\times\bdy\mathbf Z}.
	\label{e:C1-continuity}
\end{nalign}
Here, $\upsilon$ is time-reversal, so $\upsilon\JBS\upsilon$ is the outgoing solution that generated $f$. The map $\iota\colon\bdy\mathbf Z\to\bdy\mathbf Z$ reverses the copies of each boundary component within $\bdy\mathbf Z$, and $\d_\nu$ denotes the normal derivative. The second equation in~\eqref{e:C1-continuity} simplifies to a pseudodifferential equation
\begin{equation}
	N\subI f + N\subR f\subR \eqml N\subT f\subT
	\label{e:C1-continuity-2}
\end{equation}
with operators $N\subI$, $N\subR$, $N\subT\in\Psi^1(\RR\times\bdy Z)$ that may be explicitly computed. The system~(\ref{e:C1-continuity}--\ref{e:C1-continuity-2}) may be microlocally inverted to recover $f\subR=M\subR f$, $f\subT=M\subT f$ in terms of pseudodifferential reflection and transmission operators $M\subR,\,\iota M\subT\in\Psi^0(\RR\times\bdy\mathbf Z)$. Let $M=M\subR+M\subT$.

The principal symbols of $M\subR$ and $\iota M\subT$ have well-known geometric interpretations. In the doubly hyperbolic region where $\abs\tau<c\abs{\xi'}$ on both sides of the interface,
\begin{align}
	\sigma_0(M\subR) &= \frac{\cot\theta\subR-\cot\theta\subT}{\cot\theta\subR+\cot\theta\subT},
	&
	\sigma_0(\iota M\subT) &= \frac{2\cot\theta\subR}{\cot\theta\subR+\cot\theta\subT},
	\label{e:geometric-ps-refl-trans}
\end{align}
where $\theta\subR$, $\theta\subT$ are the angles between the normal and the associated reflected and transmitted bicharacteristics. Here $\cot\theta\subR=\big(c\subR^{-2}\tau^2-\abs{\xi'}^2\big){}^{1/2}/\abs{\xi'}$, where $c\subR$ is the wave speed at $\beta$ on the reflected side, and similarly for $\theta_{\textrm T}$. From~\eqref{e:geometric-ps-refl-trans} we deduce $M\subT$ is elliptic in the doubly-hyperbolic region, while $M\subR$ is elliptic as long as $c$ is discontinuous at the interface. Note that while the principal symbol of $\iota M\subT$ may exceed 1, this does not violate energy conservation since $M\subT$ operates on boundary rather than Cauchy data.

\paragraph{Parametrix}

With all the necessary components defined, we now set
\begin{nalign}
	\tilde F &= \JCS + \JBS M\sum_{k=0}^\infty (\JBB M)^k\JCB,\\
	\tilde R &= r_{2T}\circ\tilde F,
	\label{e:parametrices}
\end{nalign}
where $r_{2T}$ is restriction to $t=2T$, plus time-reversal. Again omitting the proof, it can be shown that $\tilde F\eqml F$ and $\tilde R\eqml R$ away from glancing rays; that is, for initial data $h_0$ such that every broken bicharacteristic originating in $\WF(h_0)$ is sufficiently far from glancing. Recalling that $M=M\subR+M\subT$, we may write $\tilde R$ as a sum of graph FIO indexed by sequences of reflections and transmissions:
\begin{nalign}
	\tilde R &= \smash{\sum_{\mathclap{\substack{s\in\{R,T\}^k\\k\geq0}}} \,\tilde R_s,}
	\qquad\qquad&
	\tilde R_{()} &= r_{2T}  \JCS,\\
	&&\tilde R_{(s_1,\dotsc,s_k)} &= r_{2T}  \JBS  M_{s_k}\JBB\dotsb M_{s_2}\JBB M_{s_1} \JCB.
	\label{e:propagator-graph-components}
\end{nalign}
The solution operator $\tilde F$ likewise decomposes into analogous components $\tilde F_s$.

\paragraph{Comparison with Layered Media Parametrices}

The above construction is in fact the natural generalization from the flat interface case of a layered media. Indeed, suppose our space $\Theta$ is only a small perturbation of the flat layered media case (see \cite{Hij87} for notation and analysis in the flat case). This ensures that bicharacteristic segments starting from $\Gamma_i$ hit $\Gamma_{i-1}$ or $\Gamma_{i+1}$ first before hitting another interface (here, $\Omega_i$ lies below $\Gamma_i$ and above $\Gamma_{i+1}$). The full wave field may be microlocally decomposed into upgoing and downgoing components at each interface $\Gamma_i$ denoted $u^{(i)-}$, resp. $u^{(i)+}$ as described in \cite[proof of Theorem 3.1]{StDeHoop02}. Then localizing the construction of the boundary-to-boundary maps $\JBB$, we obtain $\JBB^{i,i+1}$ (resp. $\JBB^{i,i-1}$),
 which propagate $u^{i,+}$ (resp. $u^{i,-}$) to interface $\Gamma_{i+1}$ (resp. $\Gamma_{i-1}$).

Next, there are reflection and transmission operators, denoted $R^{i,j},T^{i,j} \in \Psi^0(\RR \times \Gamma_i)$ which are essentially the $M_R, M_T$ operators from before but microlocally restricted to a particular ``side'' of a particular interface. The indexing is such that $R^{i,j}$ denotes the reflection coefficient of a wave inside $\Omega_j$ reflecting off of $\Gamma_i$. While $T^{i,j}$ denotes the transmission coefficient for a wave from $\Omega_i$ into $\Omega_j$ where the constructions are made exactly as in the previous section.
Under this simplified geometry, the outgoing waves at interface $\Gamma_i$ are given by 
\[
u^{(i)+} = T^{i-1,i}\JBB^{i-1,i}u^{(i-1)+} +R^{i,i}\JBB^{i+1,i}u^{(i+1)-}
\]
and
\[
u^{(i)-} = R^{i,i-1}\JBB^{i-1,i}u^{(i-1)+} +T^{i,i-1}\JBB^{i+1,i}u^{(i+1)-}.
\]

This is all for $i\geq 2$, while for $i=1$ we must take into account the source term $\phi \in \mathcal{D}'(\Gamma_1)$ (assuming this is the only source) and only those incoming waves from $\Gamma_2$:
\begin{align*}
u^{(1)+} &= R^{1,1}\JBB^{2,1}u^{(2)-} + \phi_{\text{source}}^+ \\
u^{(1)-} &= T^{1,0}\JBB^{2,1}u^{(2)-} + \phi_{\text{source}}^-.
\end{align*}
Denote $u^{\pm} = [u^{(1)\pm}, \dots, u^{(r)\pm}]^T$. Thus, as done in \cite{Cist73}, we may combine, the $R,T$ operators and the corresponding $\JBB$ occurring in the above formulas into one operator (for example, $R^{i,i}\JBB^{i+1,i}$ becomes a single operator). Then we form $T^{\pm}$ and $R^{\pm}$, each a $r \times r$ matrix of FIO's, to obtain the following recursive formula:
\[
\col{ u^+ \\ u^-} = \bmat T^+ & R^+ \\R^- & T^- \emat \col{u^+\\u^-} + \col{(\phi^+_{\text{source}},0,\dots,0)^T\\(\phi^-_{\text{source}},0,\dots,0)^T}.
\]
Hence, it is fitting to denote $S_{sc} = \begin{bsmallmatrix} T^+ & R^+ \\R^- & T^- \end{bsmallmatrix}$ as the scattering ``matrix'', which corresponds to $\JBB M$ appearing in (\ref{e:parametrices}). To connect this construction to (\ref{e:parametrices}), start with Cauchy data $\phi_{\text{Cauchy}} \in \mathbf{C}$ with microsupport close to a single covector, whose corresponding geodesic hits $\Gamma_1$ transversely. Then the solution restricted to $\Gamma_1$ near this first intersection is microlocally equal to
$$\phi_{\Gamma_1} = \phi_{\text{incoming}} + \phi_{\text{source}},$$
 where $\phi_{\text{incoming}}=\JCB \phi_{\text{Cauchy}}$ and $\phi_{\text{source}}^+ = T^{0,1}\JCB\phi_{\text{Cauchy}}$ and
$\phi_{\text{source}}^- = R^{1,0}\JCB\phi_{\text{Cauchy}}$. So the upgoing and downgoing parts of the solution at the interfaces are given by
\[
\col{u^+\\u^-} = \col{(\phi^+_{\text{incoming}},0,\dots,0)^T\\(\phi^-_{\text{incoming}},0,\dots,0)^T} + \sum_{k=0}^{\infty} S^k_{sc} \col{(\phi^+_{\text{source}},0,\dots,0)^T\\(\phi^-_{\text{source}},0,\dots,0)^T}.
\]
After applying the boundary to solution operator, we obtain a formula exactly analogous to (\ref{e:parametrices}),
and one can use the scattering matrix to track the principal symbols of the wave field in each $\Omega_i$ separately.

\paragraph{Funding Acknowledgements:} P.~C.~and V.~K.~were supported by the Simons Foundation                 
under the MATH $+$ X program. M.~V.~dH.~was partially supported by the Simons Foundation                 
under the MATH $+$ X program, the National Science Foundation under                 
grant DMS-1559587, and by the members of the Geo-Mathematical Group at              
Rice University. G.~U.~is Walker Family Endowed Professor of Mathematics at the University of Washington, and was partially supported by the National Science Foundation, a Si-Yuan Professorship at Hong Kong University of Science and Technology, and a FiDiPro Professorship at the Academy of Finland.

\bibliographystyle{siaminitials}
\bibliography{ScatteringControl}

\def\cprime{$'$}
\begin{thebibliography}{10}

\bibitem{AR02}
{\sc T.~Aktosun and J.~H. Rose}, {\em Wave focusing on the line}, J. Math.
  Phys., 43 (2002), pp.~3717--3745.

\bibitem{Belishev97}
{\sc M.~I. Belishev}, {\em Boundary control in reconstruction of manifolds and
  metrics (the {BC} method)}, Inverse Probl., 13 (1997), pp.~R1--R45.

\bibitem{BKLS}
{\sc K.~Bingham, Y.~Kurylev, M.~Lassas, and S.~Siltanen}, {\em Iterative
  time-reversal control for inverse problems}, Inverse Probl. Imaging, 2
  (2008), pp.~63--81.

\bibitem{Bur80}
{\sc R.~Burridge}, {\em The {G}el\cprime fand-{L}evitan, the {M}archenko, and
  the {G}opinath-{S}ondhi integral equations of inverse scattering theory,
  regarded in the context of inverse impulse-response problems}, Wave Motion, 2
  (1980), pp.~305--323.

\bibitem{C}
{\sc P.~Caday}, {\em Computing {F}ourier integral operators with caustics},
  Inverse Probl., 32 (2016), p.~125001.

\bibitem{Chazarain}
{\sc J.~Chazarain}, {\em Param{\'e}trix du probl{\`e}me mixte pour
  l'{\'e}quation des ondes {\`a} l'int{\'e}rieur d'un domaine convexe pour les
  bicaract{\'e}ristiques}, in Journ{\'e}es \'{E}quations aux {D}{\'e}riv{\'e}es
  {P}artielles de {R}ennes (1975), Soc. Math. France, Paris, 1976,
  pp.~165--181. Ast{\'e}risque, No. 34--35.

\bibitem{Cist73}
{\sc A.~Cisternas, O.~Betancourt, and A.~Leiva}, {\em Body waves in a ``real
  {E}arth.'' {P}art {I}}, Bull. Seismol. Soc. Am., 63 (1973), pp.~145--156.

\bibitem{DKO}
{\sc M.~V. de~Hoop, P.~Kepley, and L.~Oksanen}, {\em On the construction of
  virtual interior point source travel time distances from the hyperbolic
  {N}eumann-to-{D}irichlet map}, SIAM J. Appl. Math., 76 (2016), pp.~805--825.

\bibitem{dHUV}
{\sc M.~V. de~Hoop, G.~Uhlmann, and A.~Vasy}, {\em Diffraction from conormal
  singularities}, Ann. Sci. {\'E}c. Norm. Sup{\'e}r. (4), 48 (2015),
  pp.~351--408.

\bibitem{Hansen}
{\sc S.~Hansen}, {\em Singularities of transmission problems}, Math. Ann., 268
  (1984), pp.~233--253.

\bibitem{KK}
{\sc A.~Kirpichnikova and Y.~Kurylev}, {\em Inverse boundary spectral problem
  for {R}iemannian polyhedra}, Math. Ann., 354 (2012), pp.~1003--1028.

\bibitem{LV}
{\sc G.~Lion and M.~Vergne}, {\em The {W}eil representation, {M}aslov index and
  theta series}, vol.~6 of Progress in Mathematics, Birkh{\"a}user, Boston,
  Mass., 1980.

\bibitem{LionsMagenes1}
{\sc J.-L. Lions and E.~Magenes}, {\em Non-homogeneous boundary value problems
  and applications. {V}ol. {I}}, Springer-Verlag, New York-Heidelberg, 1972.
\newblock Translated from the French by P. Kenneth, Die Grundlehren der
  mathematischen Wissenschaften, Band 181.

\bibitem{Rose02}
{\sc J.~H. Rose}, {\em `{S}ingle-sided' autofocusing of sound in layered
  materials}, Inverse Probl., 18 (2002), pp.~1923--1934.
\newblock Special section on electromagnetic and ultrasonic nondestructive
  evaluation.

\bibitem{S}
{\sc Y.~Safarov}, {\em A symbolic calculus for {F}ourier integral operators},
  in Geometric and spectral analysis, vol.~630 of Contemp. Math., Amer. Math.
  Soc., Providence, RI, 2014, pp.~275--290.

\bibitem{SU-TATVariable}
{\sc P.~Stefanov and G.~Uhlmann}, {\em Thermoacoustic tomography with variable
  sound speed}, Inverse Probl., 25 (2009), pp.~075011, 16.

\bibitem{SU-TATBrain}
\leavevmode\vrule height 2pt depth -1.6pt width 23pt, {\em Thermoacoustic
  tomography arising in brain imaging}, Inverse Probl., 27 (2011), pp.~045004,
  26.

\bibitem{StolkThesis}
{\sc C.~C. Stolk}, {\em On the modeling and inversion of seismic data}, PhD
  thesis, University of Utrecht, 2001.

\bibitem{Stolk04}
\leavevmode\vrule height 2pt depth -1.6pt width 23pt, {\em A pseudodifferential
  equation with damping for one-way wave propagation in inhomogeneous acoustic
  media}, Wave Motion, 40 (2004), pp.~111--121.

\bibitem{StDeHoop02}
{\sc C.~C. Stolk and M.~V. de~Hoop}, {\em Microlocal analysis of seismic
  inverse scattering in anisotropic elastic media}, Comm. Pure Appl. Math., 55
  (2002), pp.~261--301.

\bibitem{Tataru}
{\sc D.~Tataru}, {\em Unique continuation for solutions to {PDE}'s; between
  {H}{\"o}rmander's theorem and {H}olmgren's theorem}, Comm. Partial
  Differential Equ., 20 (1995), pp.~855--884.

\bibitem{Taylor75}
{\sc M.~E. Taylor}, {\em Reflection of singularities of solutions to systems of
  differential equations}, Comm. Pure Appl. Math., 28 (1975), pp.~457--478.

\bibitem{Hij87}
{\sc J.~van~der Heijden}, {\em Propagation of transient elastic waves in
  stratified anisotropic media}, PhD thesis, Technische Universiteit Delft,
  1987.

\bibitem{Wap}
{\sc K.~Wapenaar, J.~Thorbecke, J.~van~der Neut, F.~Broggini, E.~Slob, and
  R.~Snieder}, {\em Marchenko imaging}, Geophysics, 79 (2014), pp.~WA39--WA57.

\end{thebibliography}

\end{document}